\documentclass[11pt]{amsart}
\pdfoutput=1
\usepackage[margin=1in]{geometry}
\usepackage[english]{babel}
\usepackage[utf8]{inputenc}
\usepackage{csquotes}
\usepackage[T1]{fontenc} \usepackage{ae} \usepackage{aecompl}
\usepackage{etoolbox}
\usepackage{enumerate}
\usepackage{enumitem}
\usepackage{tabularx}
\usepackage{graphicx}
\usepackage{amsmath}
\usepackage{amsfonts}
\usepackage{amssymb}
\usepackage{stmaryrd}
\usepackage{amsthm}
\usepackage{thmtools}
\usepackage{mathtools}
\usepackage{float}
\usepackage{xparse}	
\usepackage[all]{xy}

\usepackage[hypertexnames=false]{hyperref}
\hypersetup{
    breaklinks = true,
    colorlinks = true,
    citecolor = black,
    linkcolor = black,
    urlcolor = black,
}

\usepackage[style=alphabetic, maxbibnames=4, isbn=false, url=false]{biblatex}	
\usepackage{xurl}
\newtheorem{thm}{Theorem}[section]
\newtheorem{lema}[thm]{Lemma}
\newtheorem{prop}[thm]{Proposition}
\newtheorem{coro}[thm]{Corollary}
\declaretheorem[style=definition,name=Definition,sibling=thm]{Def}
\declaretheorem[style=definition,name=Remark,sibling=thm]{rmk}

\newcommand{\C}{\mathbb{C}}
\newcommand{\Q}{\mathbb{Q}}
\newcommand{\Z}{\mathbb{Z}}
\newcommand{\N}{\mathbb{N}}
\newcommand{\F}{\mathbb{F}}

\newcommand{\Hom}[2]{\operatorname{Hom}(#1,#2)}

\newcommand{\PGL}{\operatorname{PGL}}
\newcommand{\GL}{\operatorname{GL}}
\newcommand{\SL}{\operatorname{SL}}
\newcommand{\U}{\operatorname{U}}

\newcommand{\Irr}[1][G]{\operatorname{Irr}(#1)}
\newcommand{\Stab}[1][\chi]{\operatorname{Stab}_{#1}}
\newcommand{\bil}[1][G]{\mathcal{B}({#1},\U(1))}
\newcommand{\bilplus}[1][G]{\mathcal{B}^+({#1},\U(1))}

\newcommand{\tr}{\operatorname{Tr}}
\newcommand{\spec}{\operatorname{Spec }}
\newcommand{\ov}[1]{\overline{#1}}
\newcommand{\ord}[1]{\operatorname{ord}(#1)}
\newcommand{\gr}{\operatorname{gr}}
\newcommand{\lcm}{\operatorname{lcm}}

\NewDocumentCommand{\frob}{}{\operatorname{Frob}_q}

\newcommand{\mcA}{\mathcal{A}}
\newcommand{\mcP}{\mathcal{P}}
\newcommand{\mcC}{\mathcal{C}}
\newcommand{\mcD}{\mathcal{D}}
\newcommand{\mcJ}{\mathcal{J}}
\newcommand{\mcG}{\mathcal{G}}
\newcommand{\mcX}{\mathcal{X}}
\newcommand{\mcY}{\mathcal{Y}}
\newcommand{\mcM}{\mathcal{M}}
\newcommand{\mcF}{\mathcal{F}}
\newcommand{\mcE}{\mathcal{E}}
\newcommand{\mcI}{\mathcal{I}}
\newcommand{\mcO}{\mathcal{O}}
\newcommand{\mcL}{\mathcal{L}}

\NewDocumentCommand{\preBetti}{O{R} O{\mcC}}{X^{#2}(#1)}
\NewDocumentCommand{\betti}{O{\SL_n} O{R} O{\mcC}}{\mcM_B^{#3}(#1(#2))}
\NewDocumentCommand{\abbreviatedBetti}{O{\SL_n}}{\mcM_B^\mcC(#1)}
\NewDocumentCommand{\Dolbeault}{O{G} O{d}}{\mcM_{\text{Dol}}^{#2}(#1)}
\newcommand{\Betti}[1][G]{\mcM_{\text{B}}^d(#1)}

\newcommand{\pare}[1]{ { \left( #1 \right)}}
\renewcommand{\d}{{\mathrm{d}}}

\title{Counting stringy points on a family of character varieties}
\author{Lucas de Amorin}
\email{ldamorin@dm.uba.ar}
\author{Martin Mereb}
\email{mmereb@dm.uba.ar}
\address{Departamento de Matemática-IMAS, Facultad de Ciencias Exactas y Naturales, Universidad de Buenos Aires, Argentina}

\begin{document}

\begin{abstract}
    We provide explicit formulas for the number of stringy points over finite fields of parabolic type A character varieties with generic semisimple monodromy. This leads to formulas for their stringy E-polynomials. In particular, they satisfy the Betti Topological Mirror Symmetry Conjecture of T. Hausel and M. Thaddeus and a refinement regarding isotypic components. Our proof is based on a Frobenius-type formula for Clifford-type settings and an analysis of it in a specific set-up related to regular wreath products with cyclic groups.
\end{abstract}

\maketitle

\setcounter{tocdepth}{1}
\tableofcontents

\section{Introduction}

The Hitchin system (\cite{H}, \cite{Hitchin}) has been an object of intense research in the last decades due to its rich geometry and its connections to Mirror Symmetry and the geometric Langlands program (\cite{KW}, \cite{Ngo}, \cite{GWZ}, among others). Fix a curve $C$ and a reductive group $G$. A Higgs $G$-principal bundle is a pair $(\mcE,\Theta)$ where $\mcE$ is a $G$-principal bundle over $C$ and $\Theta$ is a section of $\operatorname{ad}(\mcE)\otimes\Omega_C$. 
The Dolbeault moduli space $\Dolbeault[G][]$ parametrizes these objects with fixed cohomological invariants under a suitable stability condition (\cite{SimpsonI}).
It is endowed with a fibration $\pi_G:\Dolbeault[G][]\to \mcA_G$ to the so-called Hitchin base, given by taking the characteristic polynomial of $\Theta$. This is the Hitchin system.

Among its astonishing properties, there is one we highlight below. Let $\check{G}$ be the Langlands dual group to $G$. The Hitchin bases $\mcA_{G}$ and $\mcA_{\check{G}}$ are isomorphic. The fibers of $\pi_G$ and $\pi_{\check{G}}$ are generically dual abelian varieties (\cite{HT}, \cite{DP}). This fits into the picture proposed by A. Strominger, S.-T. Yau, and E. Zaslow \cite{SYZ} for Mirror Symmetry. Thus, one expects $\Dolbeault[\check{G}][]$ to be the mirror partner of $\Dolbeault[G][]$. As a particular feature, there should be equality between their Hodge numbers. A precise statement for $\SL_n$ and $\PGL_n$ was proposed by T. Hausel and M. Thaddeus in \cite{HT-TMS}.  

Let $d$ be a positive integer and fix a line bundle $L$ of degree $d$ on $C$. A twisted $\SL_n$-Higgs bundle is a $\GL_n$-Higgs bundle $(\mcE,\Theta)$ with $\det\mcE\simeq L$ and $\Theta$ traceless. As before, there is a coarse moduli space $\Dolbeault[\SL_n]$ parametrizing them. The tensor product induces an action of the Jacobian variety $\mathrm{Jac}^0(C)$ on $\Dolbeault[\GL_n]$. Only its subgroup $\mathrm{Jac}^0(C)[n]\simeq (\Z/n\Z)^{2g}$ of $n$-torsion line bundles preserves the determinant, and it therefore acts on $\Dolbeault[\SL_n]$. The twisted $\PGL_n$-Dolbeault moduli space $\Dolbeault[\PGL_n]$ is the quotient $\Dolbeault[\SL_n]/\mathrm{Jac}^0(C)[n]$. If $d$ and $n$ are coprime, $\Dolbeault[\SL_n]$ is smooth and $\Dolbeault[\PGL_n]$ is an orbifold. The Topological Mirror Symmetry conjecture of Hausel and Thaddeus \cite{HT-TMS} proposes
\[h^{p,q}(\Dolbeault[\SL_n])=h^{p,q}_{st}(\Dolbeault[\PGL_n],\mcG)\]
whenever $\gcd(d,n)=1$, where $\mcG$ is a natural unitary gerbe on $\Dolbeault[\PGL_n]$ and $h_{st}^{p,q}$ are the stringy Hodge numbers of V. Batyrev and D. Dais \cite{BD}, which incorporate correction terms to take into account the singularities. This conjecture was first proved by M. Groechenig, D. Wyss, and P. Ziegler \cite{GWZ} by means of $p$-adic integration; their method also gives the equality for other Lie groups and gerbes. Several other proofs and generalizations followed: F. Loeser and D. Wyss \cite{LW} proved a motivic version, D. Maulik and J. Shen \cite{MS} gave a proof via perverse sheaves which produces an actual isomorphism in cohomology, and V. Hoskins and S. Pepin Lehalleur \cite{HL} lifted it to an isomorphism of Voevodsky motives. The parabolic case, which is the one related to the character varieties considered here, was proved in ranks two and three by P. Gothen and A. Oliveira \cite{GO} and in general by X. Su, B. Wang, and X. Wen \cite{SWW} and by S. Shen \cite{She}.

There exists another moduli space $\Betti$ parametrizing twisted $G$-representations of the fundamental group of $C$ (see definition below). Simpson's non-Abelian Hodge theory (\cite{SimpsonII}, \cite{SimpsonNC}, \cite{Boalch}) implies that $\Betti$ and $\Dolbeault$ are diffeomorphic if $d$ and $n$ are coprime. In particular, the cohomology rings of $\Betti$ and $\Dolbeault$ agree. However, their Hodge structures disagree. The Betti Topological Mirror Symmetry conjecture \cite[Conjecture 3.3]{Hausel} predicts
\[h^{p,q}(\Betti[\SL_n])=h^{p,q}_{st}(\Betti[\PGL_n],\mcG)\]
as before. We prove a version of this conjecture for groups that are isogenous to $\SL_n$ and discrete torsions, which are gerbes of a particular kind. Furthermore, we provide explicit formulas for these numbers.

Define the parabolic $\SL_n$-character variety as the GIT quotient
\[\abbreviatedBetti[\SL_n]=\{(x,y,z)\in \SL_n(\C)^g\times \SL_n(\C)^g\times\mcC:[x,y]z=1\}\sslash\SL_n(\C) \]
where $\mcC\subset \SL_n$ is a conjugacy class, $g$ is the genus of $C$, $[x,y]=[x_1,y_1]\cdots[x_g,y_g]$ with $[u,v]=uvu^{-1}v^{-1}$, and $\SL_n(\C)$ acts by simultaneous conjugation. For a divisor $d$ of $n$, there is an associated cyclic central subgroup $F_d\subset \SL_n$ of size $d$. The parabolic $\SL_n/F_d$-character variety is the orbifold
\[\abbreviatedBetti[\SL_n/F_d]=[\abbreviatedBetti[\SL_n]/F_d^{2g}]\]
where the action of $F_d^{2g}$ on $\SL_n(\C)^{2g}$ is given by left multiplication. We think of its points as twisted $\SL_n/F_d$-representations of the fundamental group of $C$. Finally, let $E_{st}(X;u,v)$ be the generating function of $h_{st}^{p,q}(X)$. For its precise definition and variants, see Section \ref{sec:E-pols}. We can state our theorem now.

\begin{thm}
    Let $n$ be a natural number. For any divisor $d$ of $n$, $F_d$-discrete torsion $\mcD$ of order $d/r$, and any generic semisimple conjugacy class $\mcC$ of $\SL_n(\C)$,
    \[E_{st}^{\mcD}(\abbreviatedBetti[\SL_n/F_d]; u,v) = \sum_{\tau} \frac{\left((-1)^n(uv)^{\frac{n^2}{2}}\mathcal{H}_{\tau'}(uv)\right)^{2g-1}}{(uv-1)^{2g-1+n-1}}\sum_{s|n} s^{2g} C_{s,\tau} \Phi_g(n,d,s,r)\]
    where $\tau$ runs over all multi-partition types of size $n$, $\Phi_g$ is an arithmetic function, $\mathcal{H}_{\tau'}$ are polynomials associated to $\tau$, and $C_{s,\tau}$ are combinatorial constants determined by $s$ and $\tau$.  
\end{thm}

To be precise, we require that $\mcC$ is regular semisimple and that no non-empty proper subset of its eigenvalues has product one. The polynomials $\mathcal{H}_{\tau'}$ are the normalized hook polynomials of \cite[Definition 2.1.13]{MM}. The constants $C_{s,\tau}$ are the constants $C_{s,1,\tau}$ defined and computed in Section \ref{counting-twisted-points}. The function $\Phi_g$ is given by
\[\Phi_g(n,d,s,r)=\prod_{p|n\text{  prime}} \frac{(p-p^{1-2g})p^{\phi_g(p,n,d,s,r)} + p^{1-2g}-1}{p-1}\]
where 
\[\phi_g(p,n,d,s,r)=\min(v_p(s),v_p(n)-v_p(d),v_p(d),v_p(n)-v_p(s),v_p(r))\]
and $v_p$ is the $p$-adic valuation. In particular, the symmetries of $\Phi_g$ imply the Betti Topological Mirror Symmetry.

\begin{coro}
    Let $n$ be a natural number, $d$ be a divisor of $n$ and $\mcC$ be a generic semisimple conjugacy class of $\SL_n(\C)$. Then 
    \[E_{st}^{\hat\mcD}(\abbreviatedBetti[\SL_n/F_d]; u,v) = E_{st}^{\check\mcD}(\abbreviatedBetti[\SL_n/F_{\frac{n}{d}}]; u,v)\]
    for any $F_d$-discrete torsion $\hat\mcD$ and $F_{\frac{n}{d}}$-discrete torsion $\check\mcD$ whose images under the natural surjections $F_d\to F_{\gcd(d,\frac{n}{d})}$ and $F_{\frac{n}{d}}\to F_{\gcd(d,\frac{n}{d})}$ have the same order.
\end{coro}

We can further refine this. The stringy Hodge numbers are built from the contributions of the so-called twisted sectors $\abbreviatedBetti^a/F_d^{2g}$ where $\abbreviatedBetti^a$ is the fixed-point set of $a\in F_d^{2g}$. On the other hand, if our $F_d$-discrete torsion is the restriction of an $F_n$-one, the induced gerbe on $\abbreviatedBetti[\SL_n/F_d]$ carries an $F_{\frac{n}{d}}^{2g}$-equivariant structure. Therefore, one can look at the isotypic component in cohomology with respect to a character $\xi:F_{\frac{n}{d}}^{2g}\to \C^\times$.  This induces an isotypic contribution $E_{st}^{\mcD}(\abbreviatedBetti[\SL_n/F_d]; u,v)_\xi$ of $E_{st}^{\mcD}(\abbreviatedBetti[\SL_n/F_d]; u,v)$. For its precise definition see Section \ref{sec:E-pols}. For explicit formulas regarding these polynomials see Theorem \ref{stringy-contributions}. Here we only state their relationship.

\begin{thm}\label{intro:isotopyc=stringy}
    Let $n$ be a natural number, $d$ be a divisor of $n$ and $\mcC$ be a generic semisimple conjugacy class of $\SL_n(\C)$. For any $F_n$-discrete torsion $\mcD$, character $\xi$ of $F_{\frac{n}{d}}^{2g}$ and $a\in F_{\frac{n}{d}}^{2g}$,
    \[E_{st}^{\mcD}(\abbreviatedBetti[\SL_n/F_d]; u,v)_\xi =E(\abbreviatedBetti^a/F_{\frac{n}{d}}^{2g},\mcL_{\mcD,a}; u,v)(uv)^{F(a)}\]
    provided that $\ord\xi=\ord a$, where $\mcL_{\mcD,a}$ is the local system on $\abbreviatedBetti^a/F_{\frac{n}{d}}^{2g}$ associated with $\mcD\big|_{F_{\frac{n}{d}}^{2}}$ and $F(a)$ is one-half the codimension of $\abbreviatedBetti^a$. In other words, the $\xi$-isotypic part of the stringy $E$-polynomial of $[\abbreviatedBetti/F_d^{2g}]$ equals the contribution of the sector $a$ of the dual orbifold $[\abbreviatedBetti/F_{\frac{n}{d}}^{2g}]$.
\end{thm}

There is a connection between this result and Theorem 0.5 of \cite{MS}. Hodge numbers can be defined through Deligne's Hodge and weight filtrations on the cohomology ring, called $F$ and $W$ respectively. In the case of $\Dolbeault$, one also has the so-called perverse filtration $P$, induced by the Hitchin map. The $P=W$ conjecture of M. de Cataldo, T. Hausel, and L. Migliorini \cite{dCHM} proposes that
\[P_kH^*(\Dolbeault)=W_{2k}H^*(\Betti)=W_{2k+1}H^*(\Betti)\]
for any $k$. There are two very recent proofs of it for $\GL_n$: \cite{HMMS-PW} and \cite{MS-PW}. Consequently, it holds for $\PGL_n$ \cite{PW-PGL}, if $\gcd(d,n)=1$, and for $\SL_n$ \cite{PW-SL}, if $n$ is prime. Together with Theorem 0.5 of \cite{MS}, they imply Theorem \ref{intro:isotopyc=stringy} for $\mcC$ central when $n$ is prime. We remark, however, that our methods are different, and their result misses our explicit formulas.

We expect that with the methods of \cite[Section 8]{Mellit} one should be able to weaken the generic condition on $\mcC$. In particular, this would recover the case where $\mcC$ is central.
This is part of work in progress.

Finally, let us state two geometric consequences of the previous theorems. They are compatible with previous results of T. Hausel, E. Letellier, and F. Rodriguez-Villegas: \cite[Theorem 1.2.5]{HLRV1} and \cite[Theorem 1.1.1]{HLRV2}. Pick $z\in \mcC$ and $a\in F_d^{2g}$. Let $\hat{\Upsilon}_{a}$ be the set whose elements are maps $V:\langle a\rangle\subset F_d^{2g} \to \operatorname{Grass}(n,\frac{n}{\ord a})$ such that $\C^n=\bigoplus_{i\in \langle a\rangle} V_i$ and $z$ acts fixing this decomposition. In this set $\langle a\rangle$ acts by translations. Define $\Upsilon_a=\hat{\Upsilon}_a/\langle a\rangle$. In Lemma \ref{irred-comp} we provide a decomposition of $\abbreviatedBetti[\SL_n]^a$ parametrized by $\Upsilon_a$. When the order of $a$ is coprime with $\frac{n}{d}$, it induces the decomposition of the twisted sector $\abbreviatedBetti^a/F_d^{2g}$ into irreducible components.

\begin{coro}\label{cor:irred-comp}
    Let $n$ be a natural number, $d$ be a divisor of $n$, and $\mcC\subset \SL_n(\C)$ be a generic semisimple conjugacy class. For any $a\in F_d^{2g}$ of order coprime with $\frac{n}{d}$, the irreducible components of the twisted sector $\abbreviatedBetti^a/F_d^{2g}$ are parametrized by $\Upsilon_a$. In particular, $\abbreviatedBetti[\SL_n]$ is irreducible.
\end{coro}

\begin{coro}
     Let $n$ be a natural number. For any divisor $d$ of $n$ and any generic semisimple conjugacy class $\mcC$ of $\SL_n(\C)$, the stringy Euler characteristic of $\abbreviatedBetti[\SL_n/F_d]$ vanishes if $n,g>1$ and is $1$ if $n=1$.
\end{coro}

Our approach is closely related to \cite{HV} and \cite{MM} where they compute the $E$-polynomial of certain families of character varieties. 
The main difference here is the inclusion of stringy contributions, also present in unpublished works with Hausel and Rodriguez-Villegas. Other computations of $E$-polynomials of character varieties in the literature, such as \cite{HLRV1}, \cite{Baraglia}, \cite{MaMu} and \cite{GP}, do not include stringy contributions either. Well-known results on $p$-adic Hodge theory imply that the (stringy) $E$-polynomials of a variety are determined by the (stringy) point-counting function of a spread out; see \cite[Appendix by N. Katz]{HV}, \cite{Ito} and \cite[Section 2.5]{GWZ}. In our case, the previous theorems reduce to computing $|({\betti[\SL_n][\ov{\F}_q]}^a)^{b^{-1}\frob}|$ for $a,b\in F_n^{2g}$ and suitable $q$, where $\betti[\SL_n][\ov{\F}_q]$ is obtained by replacing $\C$ with $\ov{\F}_q$ in the definition of $\abbreviatedBetti[\SL_n]$. This is the content of Theorem \ref{number-twisted-points}. Its proof proceeds by reducing to a character computation. However, there is a new difficulty. In \cite{HV} and \cite{MM}, they use a classical formula of Frobenius \cite{Frobenius} to compute the number of solutions to
\[[x_1,y_1]\cdots[x_g,y_g]z=1\]
where $x_i,y_i$ vary in a finite group $G$. In our case, there is an extra condition $\pi(x_i)=a_i$, $\pi(y_i)=b_i$ for a given quotient morphism $\pi:G\to K$ to an abelian group and given tuples $\ov a=(a_1,\ldots,a_g)$ and $\ov b=(b_1,\ldots,b_g)$ in $K^g$. Call this number $N_z(\ov a,\ov b)$ and let $H=\ker\pi$. As in Frobenius' formula, character theory shows that, if $z\in H$ and the coordinates of $\ov a$ and $\ov b$ generate $K$,
\[N_z(\ov a,\ov b) = \sum_{\theta\in\Irr[H]}\left(\frac{|H|}{\theta(1)}\right)^{2g-1}\Omega_\theta(\ov a,\ov b)\theta(z)\]
where $\Irr[H]$ is the set of irreducible characters of $H$,
\[\Omega_\theta(\ov a,\ov b)=\prod_{i=1}^g\Omega_\theta(a_i,b_i)\]
for tuples $\ov a,\ov b\in K^g$, and
\[\Omega_\theta(a,b)=\frac{\theta(1)}{|H|^{2}} \sum_{x,y\in H}\theta([xx_a,yy_b])\]
for $a,b\in K$ and chosen $x_a,y_b$ with $\pi(x_a)=a$ and $\pi(y_b)=b$. When $K$ is trivial or cyclic and the coordinates of $\ov a,\ov b$ generate $K$, $\Omega_\theta(\ov a,\ov b)$ is $1$ if $\theta$ is fixed by $K$ and $0$ otherwise. However, this is not our case. We prove several properties regarding them in Section \ref{sec:The formula}.

\begin{thm}\label{intro:omega-properties}
    Let $G$ be a finite group and $H\subset G$ be a normal subgroup such that the quotient $K=G/H$ is abelian. Let $\theta$ be any irreducible character of $H$. If $\ov a,\ov b\in K^g$ are tuples whose coordinates generate $K$, $\Omega_\theta(\ov a,\ov b)$ vanishes unless $\theta(kxk^{-1})=\theta(x)$ for any $x\in H$ and $k\in G$. In this case,
    \begin{enumerate}
        \item $\Omega_\theta$ is bilinear and anti-symmetric,
        \item $\Omega_\theta(x,x)=1$ for any $x\in K$,
        \item $|\ker \Omega_\theta|=\frac{|K|}{e(\chi)^2}$ where $e(\chi)$ is the ramification index of any irreducible character $\chi$ of $G$ over $\theta$, and,
        \item if $e(\chi)=1$, $\xi\mapsto\Omega_{\xi\theta}(a,b)$ is a quadratic function on $\{\xi\in \Hom{H}{\C^\times}:\xi\theta\in\Irr[H]^K\}$ for any $a,b\in K$.
    \end{enumerate}
\end{thm}

We use these properties to get rid of $\Omega_\theta$ in the setting coming from our original problem. This is done in Section \ref{sec:The case} after a study of its Clifford theory. We classify all fixed characters in Theorem \ref{clasf-char} and find their ramification in Theorem \ref{ramification}. The final version of the character formula is given in Theorem \ref{reduced-character-formula}. 

As a final remark, by our formula for $|({\betti[\SL_n][\ov{\F}_q]}^a)^{b^{-1}\frob}|$ one could expect that the classes of $\abbreviatedBetti^a$ and $\abbreviatedBetti^{a'}$ agree in the Grothendieck group of $F_n$-varieties. A first step towards this question is Lemma \ref{irred-comp}.

The rest of the paper is organized as follows. In Section \ref{sec:preliminaries} we recall a few preliminaries. In Section \ref{sec:stringy-points}, we compute $|({\betti[\SL_n][\ov{\F}_q]}^a)^{b^{-1}\frob}|$ using the formulas of Sections \ref{sec:The formula} and \ref{sec:The case}. We finish in Section \ref{sec:stringy-E-polynomial} by computing the stringy $E$-polynomial and its variants and checking all the previously mentioned results. Additionally, in Section \ref{sec:consequences} we prove the geometric consequences mentioned before.

\medskip\noindent\textit{Acknowledgements.} The help of Python was key to finding mistakes in previous versions of this paper. We wish to thank Guido Arnone, 
Philip Boalch, Charly di Fiore, Tam\'as Hausel, Fernando Rodriguez-Villegas, Laura Schaposnik and Dimitri Wyss for useful discussions. The first named author was supported by a PhD fellowship from CONICET. The second named author was partially supported by CONICET grant PIP 11220210100220CO (2021--2024). 

\medskip\noindent\textit{AI disclosure.}
The first publicly available version of this manuscript contained the main results and overall approach and did not use LLMs. For the current version, the authors used Claude Opus 5.5 and Fable 5.1 to correct typos, grammar, notation inconsistencies, bibliography errors, and minor gaps in arguments, and to suggest small simplifications. In one revision, a model identified a gap in the proof of Proposition \ref{magic-formula} and, when prompted, proposed a fix; after the authors verified and digested the argument, the model, with further guidance, wrote the current Subsection \ref{sec:proof-magic} and Lemma \ref{lem:twist-block-like}. The authors have verified the correctness of the manuscript and take full responsibility for its contents.

\section{Preliminaries}\label{sec:preliminaries}

\subsection{Clifford theory}\label{sec:clifford}

We first recall a few preliminaries regarding Clifford theory. A reference is \cite{Green2}. Let 
\[1\to H\to G\xrightarrow{\pi} K\to 1\]
be an exact sequence of finite groups with $K$ abelian. We call this a Clifford setting. Clifford theory relates the character theory of $G$ and $H$. We will denote by $\Irr[G]$ the set of irreducible characters of a group $G$.

In a Clifford setting, there are two important actions. The first one is of $K$ on $\Irr[H]$ and it is given by $(a\cdot\theta)(x) =\theta(axa^{-1})$. For the second one, note that the quotient map $\pi:G\to K$ induces an inclusion $\Irr[K]\subset \Irr[G]$. This gives an action of $\Irr[K]$ on $\Irr$ by multiplication: $(\xi\cdot \chi)(x)=\xi(\pi(x))\chi(x)$. Denote by $[\chi]$ and $\Stab$ the orbit and the stabilizer of a character $\chi$ respectively, and by $[\theta]$ and $\Stab[\theta]$ the $K$-orbit and the stabilizer in $K$ of $\theta\in\Irr[H]$. We write $\Irr[H]^K$ for the set of characters of $H$ fixed by $K$.

\begin{thm}[Theorem 1 of \cite{Green2}]\label{clifford}
    Let $H\to G\xrightarrow{\pi} K$ be a Clifford setting. Let $\chi\in\Irr$. Then there exists an integer $e(\chi)$, called the ramification index of $\chi$, such that
    \[\chi\big|_H = e(\chi)\sum_{\theta'\in [\theta]} \theta '\]
    where $\theta$ is any irreducible constituent of $\chi\big|_H$. This integer satisfies
        \[e(\chi)^2=\frac{|K|}{|[\chi]|\cdot|[\theta]|}\]
    and, if $K$ is cyclic, $e(\chi)=1$. Finally, let $\chi'$ be another character of $G$. Then $\chi'\big|_H=\chi\big|_H$ if and only if $\chi'\in [\chi]$. 
\end{thm}

Note additionally that $\theta(1)=\theta'(1)$ for any $\theta'\in [\theta]$. Hence, $\chi(1)=e(\chi)\,|[\theta]|\,\theta(1)$. We say that $\chi$ lies over $\theta$ if $\theta$ is an irreducible constituent of $\chi\big|_H$. All the characters over $\theta$ have the same ramification index, which we denote by $e(\theta)$. We say that $\theta$ is unramified, or that it does not ramify in $G$, if $e(\theta)=1$. When several Clifford settings are involved, we write $e_{G/H}(\chi)$ for the ramification index of $\chi$ with respect to $H\to G\to G/H$.

\begin{lema}\label{cor:cyclic-stab-implies-not-ramification}
    Let $G\to K$ be a Clifford setting and $\chi\in \Irr$ with $e(\chi)^2=|\Stab|$. If $\Stab$ has at most two generators, then $\Stab\simeq \Z/e(\chi)\Z\times\Z/e(\chi)\Z$.
\end{lema}
\begin{proof}
    If $e(\chi)=1$, there is nothing to prove. Otherwise, write $\Stab=\Z/n\Z\times\Z/m\Z$, let $C\subset\Stab$ be one of the two factors, of order $c$, and let $H_C\subset G$ be the preimage of $\bigcap_{\xi\in C}\ker\xi$. Then $G/H_C\simeq\Irr[C]$ is cyclic and $C$ fixes $\chi$, so, by Theorem \ref{clifford}, $\chi\big|_{H_C}$ is a sum of $c$ distinct irreducible characters of the same degree. As $|\Stab|=e(\chi)^2$ means that $\chi\big|_H=e(\chi)\theta$ by Theorem \ref{clifford}, each of them restricts to the same multiple of $\theta$, so $c$ divides $e(\chi)$. Hence $n,m\mid e(\chi)$, and $nm=e(\chi)^2$ forces $n=m=e(\chi)$.
\end{proof}

\subsection{Characters of finite general linear groups}

Let us review some aspects of Green's classification of characters of $\GL_{n}(\F_q)$. For the basics on partitions, multi-partitions, and their associated numbers or polynomials we refer to \cite[Section 2.1]{MM}. Let $\mcP$ be the collection of all partitions. Recall that the principal series of $\GL_{n}(\F_q)$ consists of the irreducible constituents of the induced characters $\operatorname{Ind}_B^{\GL_n(\F_q)}\psi$, where $B$ is the subgroup of upper triangular matrices and $\psi$ is a character of the diagonal torus $T\simeq(\F_q^\times)^n$, regarded as a character of $B$ through the projection $B\to T$. 

\begin{thm}[J. Green]\label{Green}
    The irreducible constituents of the principal series of $\GL_{n}(\F_q)$ are in bijection, $\Lambda\mapsto\chi_\Lambda$, with multi-partitions $\Lambda:\Irr[\F_q^\times] \to \mcP$ such that
    \[\sum_{\xi\in \Irr[\F_q^\times]} |\Lambda(\xi)| = n.\]
    All other irreducible characters of $\GL_{n}(\F_q)$ vanish on split regular semisimple elements. Moreover, the value of $\chi_\Lambda$ at a diagonal regular matrix $(\mu_1,\ldots,\mu_{n})$ is given by
    \begin{equation}\label{eq:green-character-formula}
    \pare{
    \prod_{\xi\in \Irr[\F_q^\times]} \chi^{\Lambda(\xi)}(1)
    }
    \pare{
    \sum_{\substack{\{1,\ldots,n\}=\sqcup_\xi I_\xi \\ |I_\xi|=|\Lambda(\xi)|}}
    \,\,\,
    \prod_{\xi\in \Irr[\F_q^\times]}
    \,\,\,
    \prod_{i\in I_\xi}\xi(\mu_i)
    }
    \end{equation}
    where $\chi^{\Lambda(\xi)}$ is the character of $S_{|\Lambda(\xi)|}$ associated with $\Lambda(\xi)$. 
\end{thm}
\begin{proof}
    This is a direct application of Theorems 13, 14, 12, 7, Lemma 5.1, and Example 1 at the end of page 423 of \parencite{Green}. 
\end{proof}

The type $\tau:\mcP \to \N_0$ of a multi-partition $\Lambda$ associates to each partition $\lambda$ the size $m_\lambda$ of the fiber of $\Lambda$ over it. We write $\lambda\in\tau$ if $m_\lambda>0$, and we also write $\tau$ as the multiset of partitions in which each $\lambda$ appears $m_\lambda$ times. For example, $\tau=\{\{1\},\{1\}\}$ is the type given by $\tau(\{1\})=2$ and $\tau(\lambda)=0$ for every other partition $\lambda$. The constant factor in \eqref{eq:green-character-formula} only depends on the type of the multi-partition and equals
\[d_\tau:=\left(\prod_\xi \chi^{\Lambda(\xi)}(1)\right) = \prod_\xi \frac{|\Lambda(\xi)|!}{\prod h_{\Lambda(\xi)}(i,j)}\]
where $h_{\Lambda(\xi)}(i,j)$ is the hook length of the cell $(i,j)$ in the Ferrers diagram of $\Lambda(\xi)$. 

The dimension of the associated representation turns out to be determined by the type too, and we denote it by $\tau(1)$ (do not confuse with $\tau(\{1\})$, the value of $\tau$ at the partition $\{1\}$). In \parencite[page 286]{Macdonald} it is proven to be equal to
\[ \tau(1)=\frac{\prod^{n}_{i=1}(q^i - 1)}{q^{-n(\tau')}H_\tau(q)} \]
where $H_\tau(q)$ is the hook polynomial associated with $\tau$, and $n(\tau')=\sum_\lambda m_\lambda n(\lambda')$ with $n(\lambda)=\sum_i(i-1)\lambda_i$. In addition, for any type $\tau$,
\[\frac{|\GL_{n}(\F_q)|}{\tau(1)}=(-1)^{n}q^{\frac{n^2}{2}}\mathcal{H}_{\tau'}(q)\]
where $\tau'$ is the dual type to $\tau$ and $\mathcal{H}_{\tau'}$ is the normalized hook polynomial \cite[Definition 2.1.13]{MM}. 

In the Clifford setting $\SL_n(\F_q)\to \GL_n(\F_q)\to \F_q^\times$, the action of $\Irr[\F_q^\times]$ has a nice description in the language of multi-partitions. In \parencite[Proposition at the end of page 133]{Green2}, it is proved that it agrees with the one given by
\[(\alpha\cdot \Lambda)(\xi)=\Lambda(\alpha^{-1}\xi)\]
for $\alpha,\xi\in\Irr[\F_q^\times]$ and $\Lambda:\Irr[\F_q^\times]\to \mcP$. In particular, note that the size of $\Stab[\Lambda]$ always divides $n$.

\subsection{E-polynomials}\label{sec:E-pols}

In this subsection, we summarize what we need about 
E-polynomials and ways to compute them.
For a detailed exposition see \cite[Appendix by N. Katz]{HV} and \cite[Section 2]{GWZ}. We start with the definition of the E-polynomial and its variants.

Let $X$ be a complex algebraic variety. In \cite{HodgeII} and \cite{HodgeIII}, Deligne defined a mixed Hodge structure (MHS) on the compactly supported cohomology $H_c^*(X,\C)$. There exist two filtrations 
\[W_{-1}=0 \subset W_0\subset \cdots \subset W_{2j}=H_c^{j}(X,\C)\]
and
\[F^0=H_c^{j}(X,\C)\supset F^1 \supset \cdots \supset F^j \supset F^{j+1} = 0 \]
on each $H^j_c(X,\C)$ such that in each graded piece of $W$ the induced filtration by $F$ is a Hodge structure; $\gr^W_k H^j_c(X,\C)=F^p\oplus\ov{F^{k+1-p}}$ for all $k$ and $p$. Employing this structure one defines the compactly supported virtual Hodge numbers
\[h^{p,q}_c(X):=\sum_k (-1)^k \dim_\C(\gr_F^p\gr_{p+q}^W H^{k}_c(X,\C))\]
and the $E$-polynomial
\[E(X;u,v) := \sum_{p,q}h_c^{p,q}(X)u^pv^q\]
There are several generalizations. First, one can replace $\C$ by any rank-one local system $\mcL$ with finite monodromy to obtain the twisted version $E(X,\mcL;u,v)$. Assume further that there is a finite group $\Gamma$ acting on $X$ and an equivariant structure on $\mcL$. Since Deligne's MHS is functorial, the filtrations $W$ and $F$ restrict to $H^*_c(X,\mcL)^\Gamma$ and to any isotypic component. Hence, one can define
\[E(X/\Gamma,\mcL;u,v):=\sum_{p,q,k}(-1)^k \dim_\C(\gr_F^p\gr^W_{p+q}H_c^{k}(X,\mcL)^\Gamma) u^pv^q\]
and
\[E(X/\Gamma,\mcL;u,v)_\xi:=\sum_{p,q,k}(-1)^k \dim_\C(\gr_F^p\gr^W_{p+q}H_c^{k}(X,\mcL)_\xi) u^pv^q\]
where $H_c^{k}(X,\mcL)_\xi$ is the isotypic component with respect to a character $\xi:\Gamma\to\C^\times$. Remark that $E(X/\Gamma,\mcL;u,v)_\xi = E(X/\Gamma,\mcL_\xi;u,v)$ where $\mcL_\xi:=\mcL\otimes\C_\xi$ and $\C_\xi$ is the trivial local system with equivariant structure given by $\xi^{-1}$.

Now let us consider an orbifold $\mathfrak{X}=[X/\Gamma]$. Orbifold cohomology was introduced in \cite{VW}, \cite{CR}, and \cite{R}, among others. The main novelty was to consider stringy contributions. This leads to the definition of the stringy E-polynomial of an orbifold. Let us review this. 

The inertia stack or ghost loop space $I\mathfrak{X}$ consists of loops in $\mathfrak{X}$ which are trivial in the quotient space $X/\Gamma$, see \cite{LU}. More concretely, its points are pairs $(x,\gamma)$ with $\gamma\in\Gamma$ and $\gamma x=x$, up to the action $h\cdot(x,\gamma)=(hx,h\gamma h^{-1})$ of $\Gamma$. As a stack,
\[I\mathfrak{X} \simeq \bigsqcup_{[\gamma]\in [\Gamma]} [X^\gamma/ C(\gamma)]\]
where $[\Gamma]$ is the set of conjugacy classes of $\Gamma$, $X^\gamma$ is the $\gamma$-fixed point set and $C(\gamma)$ is the centralizer of $\gamma$. Orbifold cohomology is the cohomology of $I\mathfrak{X}$ up to some shifts, called Fermionic shifts. They are locally constant functions on $X^\gamma/C(\gamma)$ given by 
\[F(\gamma)=\sum w_j\]
where $\gamma$ acts on $TX\big|_{X^\gamma}$ with eigenvalues $e^{2\pi iw_j}$, $w_j \in [0, 1)$. One defines then the stringy E-polynomial as
\[E_{st}(\mathfrak{X};u,v)=\sum_{[\gamma]\in[\Gamma]}\sum_{\mathfrak{Y}} E(\mathfrak{Y};u,v) (uv)^{F(\gamma)\big|_{\mathfrak{Y}}} \]
where $\mathfrak{Y}$ runs over all connected components of $X^\gamma/C(\gamma)$. We remark that this definition does not depend on the chosen presentation $\mathfrak{X}=[X/\Gamma]$. Furthermore, under suitable assumptions, $E_{st}(\mathfrak{X};u,v)$ is nothing but the E-polynomial of a crepant resolution, see \cite{BD}.

There is also a twisted version. There is a transgression map (see \cite[Section 2.2]{GWZ}) which assigns to each gerbe $\mcG \in H_{\mathrm{\acute{e}t}}^2(\mathfrak{X},\C^\times)$ a family of local systems $\mcL_{\mcG}\in H_{\mathrm{\acute{e}t}}^1(I\mathfrak{X},\C^\times)$. One defines
\[E_{st}^\mcG(\mathfrak{X};u,v)=\sum_{[\gamma]\in[\Gamma]}\sum_{\mathfrak{Y}} E(\mathfrak{Y},\mcL_{\mcG}\big|_{\mathfrak{Y}};u,v) (uv)^{F(\gamma)\big|_{\mathfrak{Y}}} \]
for such a gerbe. An important case is that of the so-called discrete torsions $\mcD\in H^2(\Gamma,\U(1))$, see \cite{R}. They induce gerbes on $\mathfrak{X}$ by the natural maps $H^2(\Gamma,\U(1))\to H^2(\mathfrak{X},\U(1))\to H^2(\mathfrak{X},\C^\times)$ and hence one can define $E_{st}^\mcD(\mathfrak{X};u,v)$. 

Finally, assume that $\Gamma$ is abelian, that its action on $X$ extends to a bigger abelian group $\hat{\Gamma}$, and that the $\Gamma$-equivariant structure on a gerbe $\mcG$ extends to $\hat\Gamma$. This induces a $\hat{\Gamma}/\Gamma$-equivariant structure on $\mcL_\mcG$. Hence, one has, for any character $\xi:\hat{\Gamma}/\Gamma\to\C^\times$, the isotypic contributions
\[E_{st}^\mcG(\mathfrak{X};u,v)_\xi=\sum_{\gamma\in\Gamma}E(X^{\gamma}/\Gamma,\mcL_{\mcG,\gamma};u,v)_\xi (uv)^{F(\gamma)}\]
where $\mcL_{\mcG,\gamma}=\mcL_\mcG\big|_{X^\gamma/\Gamma}$ and the subscript $\xi$ denotes the $\xi$-isotypic part of $H^*_c(X^\gamma,\mcL_{\mcG,\gamma})^\Gamma$ under $\hat\Gamma/\Gamma$, as long as the Fermionic shifts are constant. In particular, for any discrete torsion $\mcD \in H^2(\hat\Gamma,\U(1))$ one gets $E_{st}^{\mcD}(\mathfrak{X};u,v)_\xi$.

It is a well-known fact that $E$-polynomials can be computed by counting points over finite fields. Results of this type, which recover virtual Hodge numbers from point counts over finite fields by means of $p$-adic Hodge theory, go back to T. Ito \cite{Ito} and N. Katz \cite[Appendix]{HV}; the twisted equivariant version that we need is due to M. Groechenig, D. Wyss, and P. Ziegler \cite[Theorem 2.15]{GWZ}. For the sake of simplicity, we give a precise statement just for the case that we are going to use. Let $X$ be a variety over a finite field $\F_q$ equipped with an action of a finite group $\Gamma$. Define
\[ \#^\xi_\Gamma X = \frac{1}{|\Gamma|}\sum_{\gamma\in \Gamma}|X^{\gamma\frob} | \ov{\xi(\gamma)}\]
for any character $\xi:\Gamma\to \ov{\Q}_\ell^\times\simeq \C^\times$. Remark that 
\[ \#^\xi_\Gamma X = \sum_{i\in \Z} (-1)^i\tr(\frob, H^i_c(X,(\ov{\Q}_\ell)_\xi)^\Gamma) \]
by the Grothendieck--Lefschetz trace formula.

A spread out of a $\Gamma$-variety $X$ over $\C$ is a $\Gamma$-scheme $\mcX/R$ over a finitely generated $\Z$-algebra $R\subset\C$ such that $\mcX\times_{\spec R}\spec\C=X$. For a morphism $\phi:R\to \F_q$ to a finite field, write $\mcX_\phi:=\mcX\times_{\spec R}\spec \F_q$. Given a character $\xi:\Gamma\to\C^\times$, we say that $\mcX$ has $\xi$-twisted polynomial count if there exists a polynomial $P\in \Q[q]$ such that
\[\#^\xi_\Gamma \mcX_\phi = P(q) \]
for any such $\phi$.

\begin{thm}\label{KGWZ}
    Let $Y$ be a quasi-projective variety over $\C$ with an action of a finite group $\Gamma$ and $\xi:\Gamma\to\C^\times$ a character. Fix an abstract isomorphism $\C\simeq \ov{\Q}_\ell$. If there exists a spread out $\mcY/R$ of $Y$ which has $\xi$-twisted polynomial count $P$, then
    \[E(Y/\Gamma;u,v)_\xi=P(uv).\]
\end{thm}
\begin{proof}
    This follows from \cite[Theorem 2.15]{GWZ} by the argument of (2) implies (3) in \cite[Theorem 6.1.2]{HV}.
\end{proof}

In the case of a discrete torsion $\mcD\in H^2(\Gamma,\U(1))$ and an orbifold $\mathfrak{X}=[X/\Gamma]$, the local system $\mcL_{\mcD,\gamma}$ associated with $\mcD$ in $X^\gamma/C(\gamma)$ trivializes when pulled back to $X^\gamma$, as explained in \cite[end of Section 3]{R}. Hence, we can apply the previous theorem as long as we are able to count twisted points of $X^\gamma$ for any $\gamma$.

\subsection{Type A parabolic character varieties}\label{prelim:ch-var}

Given positive integers $n$ and $g$, a ring $R$, and a conjugacy class $\mcC\subset \SL_n(R)$, define
\[\preBetti=\{(x,y,z)\in \SL_n(R)^g\times \SL_n(R)^g\times\mcC:[x,y]z=1\}\]
and
\[\betti=\preBetti/\PGL_n(R).\]
For a general ring $R$, we regard $\preBetti$ as the affine $R$-scheme defined by these equations, and $\betti$ as its GIT quotient by $\PGL_n$ over $R$ \cite{Seshadri}. Over a field $k$ we also write $\betti[\SL_n][k]$ for the set of orbits $\preBetti[k]/\PGL_n(k)$.

For a divisor $d|n$, let $F_d \subset \mathbb{G}_m$ be the $d$-torsion part of the multiplicative group and let $T\simeq\mathbb{G}_m^{n-1}$ be the diagonal torus of $\PGL_n$. If $k$ is an algebraically closed field whose characteristic does not divide $d$, $F_d(k)$ is the cyclic group of order $d$. Let 
\[\betti[(\SL_n/F_d)] = [\betti/F_d(R)^{2g}]\]
be the orbifold quotient of $\betti$ under the action of $F_d(R)^{2g}$ given by left multiplication on $(x,y)$. We abbreviate $\abbreviatedBetti=\betti[\SL_n][\C]$ and $\abbreviatedBetti[\SL_n/F_d]=\betti[(\SL_n/F_d)][\C]$.

\begin{lema}\label{lem:free-smooth}
    Let $k$ be an algebraically closed field whose characteristic does not divide $n$ and $\mcC\subset \SL_n(k)$ be a regular semisimple conjugacy class such that there is no non-empty proper subset of its eigenvalues whose product is one. Then the action of $\PGL_n(k)$ on $\preBetti[k]$ is free and $\preBetti[k]$ is smooth. Consequently, $\betti[\SL_n][k]$ is smooth.
\end{lema}
\begin{proof}
    Same proof as that of \parencite[Theorem 2.2.5]{HV}.
\end{proof}

\section{A character formula} \label{sec:The formula}

\subsection{The formula}

In this section, we analyze the following. Let $H\to G\xrightarrow{\pi} K$ be a Clifford setting. Write $G_k$ for the inverse image of $k\in K$. In particular, $G_1=H$. More generally, $G_{\ov k}=G_{k_1}\times\cdots \times G_{k_m}$ for any tuple $\ov k\in K^m$. Assume that $z\in G$, $w$ is a word in $m$ variables and $\ov a\in K^m$. We want a formula in terms of characters of $G$ or $H$ for
\[N_z(w,\ov a):=|\{\ov x\in G_{\ov a}: w(\ov x)z=1\}|\]
and, especially, for $w=[x_1,y_1]\cdots[x_g,y_g]$, in which case we write $N_z(\ov a,\ov b)$ for $N_z(w,(\ov a,\ov b))$ with $\ov a,\ov b\in K^g$. This case is the content of Theorem \ref{character-formula}. First, we introduce a few definitions and notations.

For an arbitrary group $G$, a function $\Omega:G\times G\to \U(1)$ is called
\begin{itemize}
    \item bilinear if $\Omega(xy,z)=\Omega(x,z)\Omega(y,z)$ and $\Omega(x,yz)=\Omega(x,y)\Omega(x,z)$ for any $x,y,z\in G$,
    \item anti-symmetric if $\Omega(x,y)=\Omega(y,x)^{-1}$ for any $x,y\in G$, and
    \item symmetric if $\Omega(x,y)=\Omega(y,x)$ for any $x,y\in G$.
\end{itemize}
Let $\bil$ be the set of all anti-symmetric bilinear functions $G\times G\to \U(1)$. The multiplication of $\U(1)$ induces a multiplication on $\bil$ making it an abelian group. 

Note that for $\Omega\in\bil$, $\Omega(x,1)=\Omega(1,x)=1$ and $\Omega(x,x)=\pm 1$ for any $x\in G$. We call $\bilplus\subset \bil$ the subgroup of those $\Omega$ for which $\Omega(x,x)=1$ for any $x\in G$. For example, if $G\simeq \Z/n\Z \times \Z/m\Z$, an $\Omega\in \bilplus$ is of the form
    \[\Omega(x,y)= \omega^{x_1y_2-x_2y_1} \]
where $\omega$ is a root of unity of order dividing $\gcd(n,m)$. Define $\ker\Omega$ as the subgroup of $G$ consisting of those $y$ for which $\Omega(x,y)=1$ for every $x\in G$. In the previous example, $\ker\Omega = \ord\omega\,\Z/n\Z \times \ord\omega\,\Z/m\Z $, which has size $\frac{nm}{(\ord\omega)^2}$. 

For a subgroup $G'\subset G$ we think of $\bilplus[G']$ and $\bil[G']$ as sets of functions $G\times G\to \C$ by extending by zero. In addition, we set
\[\Omega(\ov x,\ov y)=\prod_{i=1}^m\Omega(x_i,y_i)\]
for tuples $\ov x,\ov y \in G^m$ and $\Omega\in \bil$. In particular, let
\[[\ov x,\ov y]=[x_1,y_1]\cdots [x_g,y_g]\]
for tuples $\ov x,\ov y\in G^g$, where $[u,v]=uvu^{-1}v^{-1}$. 

\begin{thm}\label{character-formula}
    Let $H\to G\xrightarrow{\pi} K$ be a Clifford setting. There exists a natural map $\theta \in \Irr[H]^K \mapsto \Omega_\theta\in \bilplus[K]$ such that
    \[N_z(\ov a,\ov b)=\left|\left\{\begin{array}{cc}
         (\ov x,\ov y)\in G^g\times G^g:\\
         {[}\ov x,\ov y]z=1,\pi(\ov x)=\ov a,\pi(\ov y)=\ov b
    \end{array}\right\}\right|=\sum_{\theta\in\Irr[H]^K}\left(\frac{|H|}{\theta(1)}\right)^{2g-1}\Omega_\theta(\ov a,\ov b)\theta(z)\]
    for any positive integer $g$, $z\in H$ and tuples $\ov a,\ov b\in K^g$ whose coordinates generate $K$. Here $[\ov x,\ov y]$ is regarded as a word in $2g$ variables. In addition, 
    \[|\ker\Omega_\theta| = \frac{|K|}{e(\theta)^2}\]
    where $e(\theta)$ is the ramification index of $\theta$ in $G$.  
\end{thm}

The following proposition is what allows us to get rid of $\Omega_\theta$ in Lemma \ref{sum-omega}, through Lemma \ref{lem:twist-block-like}. The previous theorem is proved in the next subsection, and the following proposition and its corollary in Section \ref{sec:proof-magic}. A function $f:G\to\U(1)$ is called a quadratic function if $f(1)=1$ and
\[B_f(x,y)=\frac{f(xy)}{f(x)f(y)}\]
is bilinear and symmetric. Characters are quadratic functions $f$ with $B_f=1$.

\begin{prop}\label{magic-formula}
    Let $H\to G\to K$ be a Clifford setting and $\theta\in \Irr[H]^K$ be an unramified character. Then
    \[\Xi_\theta:=\{\xi\in \Hom{H}{\C^\times}:\xi\theta\in\Irr[H]^K\}\]
    is a subgroup of $\Hom{H}{\C^\times}$ and, for any tuples $\ov a,\ov b\in K^g$, $\xi\mapsto\Omega_{\xi\theta}(\ov a,\ov b)$ is a quadratic function on $\Xi_\theta$.
\end{prop}

\begin{coro}\label{coro:magic-cyclic}
    Let $H\to G\to K$ be a Clifford setting and $\theta\in\Irr[H]^K$ be an unramified character. If $\Xi_\theta$ is cyclic, then $\xi\mapsto\Omega_{\xi\theta}(\ov a,\ov b)$ is a character of $\Xi_\theta$ for any tuples $\ov a,\ov b\in K^g$.
\end{coro}

\subsection{Proof of Theorem \ref{character-formula}}

Set 
\[\int_X f(\ov x)\d \ov x = \frac{1}{|G|^m}\sum_{\ov x\in X}f(\ov x)\]
for any subset $X\subset G^m$ and function $f:G^m\to \C$; note that the normalization is by $|G|^m$ and not by $|X|$. Then
\begin{align*}
    N_z(w,\ov a)&= |G|^{m}\int_{G_{\ov a}}\delta(w(\ov x)z)\d \ov x \\
    &=|G|^{m-1}\sum_{\chi\in\Irr} \chi(1) \int_{G_{\ov a}}\chi(w(\ov x)z) {\mathrm{d}} \ov x
\end{align*}
for any word $w$ in $m$ variables and $\ov a\in K^m$, where $\delta:G\to\C$ is one at the identity and zero elsewhere. We have used that $\delta= \sum_{\chi\in\Irr} \frac{\chi(1)}{|G|}\chi$. Note that if $w(\ov a)\pi(z)=1$ in $K$, then $w(\ov x)z\in H$ for every $\ov x\in G_{\ov a}$, so the integral only depends on $\chi\big|_H$ and not on $\chi$.

\begin{lema}\label{lem:schur-average}
    Let $G\to K$ be a Clifford setting and $\chi$ be an irreducible character of $G$. Then
    \begin{itemize}
        \item The following equality holds 
        \[ \int_{G_{\ov a}} \chi(w(\ov x)z)\d \ov x=\frac{1}{\chi(1)}\left(\int_{G_{\ov a}}\chi(w(\ov x))\d \ov x\right)\chi(z)\]
        for any word $w$ in $m$ variables, $\ov a\in K^m$ and $z\in G$,
        \item Let $w^{(1)}$, $w^{(2)}$ be two words in $m_1$ and $m_2$ variables, respectively. Then
        \[ \int_{G_{\ov a^{(1)}}}\int_{G_{\ov a^{(2)}}}  \chi(w^{(1)}(\ov x)w^{(2)}(\ov y))\d \ov x \d \ov y=\frac{1}{\chi(1)}\left(\int_{G_{\ov a^{(1)}}}\chi(w^{(1)}(\ov x))\d \ov x\right)\left(\int_{G_{\ov a^{(2)}}}  \chi(w^{(2)}(\ov y))\d \ov y\right)\]
        for any $\ov a^{(1)}\in K^{m_1},\ov a^{(2)}\in K^{m_2}$.
    \end{itemize}
\end{lema}
\begin{proof}
    Let $\rho$ be the representation associated with $\chi$ and consider
    \[\rho(w,\ov a):=\int_{G_{\ov a}}\rho(w(\ov x))\d \ov x\]
    which is a $G$-invariant endomorphism as each $G_k$ is invariant under conjugation. Therefore, 
    \[\rho(w,\ov a)=\frac{\tr \rho(w,\ov a)}{\chi(1)} \operatorname{Id}=\frac{1}{\chi(1)}\left(\int_{G_{\ov a}}\chi(w(\ov x))\d \ov x\right)\operatorname{Id}\]
    and the lemma follows.
\end{proof}

Applying this to the word $w = [x_1,y_1]\cdots [x_g,y_g]$ in $2g$ variables, we get
\begin{align*}
    N_z(\ov a,\ov b) = |G|^{2g-1}\sum_{\chi\in\Irr} \frac{\chi(z)}{\chi(1)^{g-1}} \prod_{i=1}^g\left(\int_{G_{a_{i}}}\int_{G_{b_{i}}}\chi([x,y])\d x \d y\right)
\end{align*}
for any $\ov a,\ov b\in K^{g}$. Thus, setting
\[\Omega_\chi(a,b)= \frac{|K|^2\chi(1)}{e(\chi)^2} \int_{G_a}\int_{G_b} \chi([x,y])\d x \d y\]
for $a,b\in K$, we have
\begin{align*}
     N_z(\ov a,\ov b) &=\frac{1}{|K|^{2g}} \sum_{\chi\in\Irr} \frac{|G|^{2g-1}}{\chi(1)^{2g-1}}e(\chi)^{2g} \Omega_\chi(\ov a,\ov b)\chi(z)\\
     &= \sum_{\theta\in\Irr[H]} \frac{|H|^{2g-1}}{(\theta(1)|[\theta]|)^{2g-1}} \left(\sum_{\chi\text{ over }\theta}\frac{e(\chi)^{2}}{|K|} \Omega_\chi(\ov a,\ov b)\right)\theta(z)
\end{align*}
for any $z\in H$, where the second sum runs over all characters $\chi$ of $G$ that have $\theta$ as an irreducible constituent; here we have used that $\chi(z)=e(\chi)\sum_{\theta'\in[\theta]}\theta'(z)$ and $\chi(1)=e(\chi)\,|[\theta]|\,\theta(1)$ (Theorem \ref{clifford}).

Since $K$ is abelian, $[x,y]\in H$ for all $x,y\in G$. Hence, if $\theta$ is an irreducible constituent of $\chi\big|_H$, Theorem \ref{clifford} gives
\begin{equation}\label{eq:Omega-chi-theta}
\Omega_\chi(a,b)=|[\theta]|\sum_{\theta'\in[\theta]}\Omega_{\theta'}(a,b),\qquad \Omega_{\theta'}(a,b):=\frac{\theta'(1)}{|H|^2}\sum_{x\in G_a,\,y\in G_b}\theta'([x,y]),
\end{equation}
where $\Omega_{\theta'}$ is the function of the introduction. In particular, $\Omega_\chi$ only depends on $\theta$, not on the choice of $\chi$ over $\theta$, and $\Omega_\chi=\Omega_\theta$ if $\theta\in\Irr[H]^K$.

Therefore, Theorem \ref{character-formula} will follow by proving that $\Omega_\chi(\ov a,\ov b)$ vanishes for all tuples $\ov a,\ov b\in K^g$ whose coordinates generate $K$ unless $\theta$ is fixed by $K$ and that, in this case, $\Omega_\chi$ is an element of $\bilplus[K]$ whose kernel has size $\frac{|K|}{e(\chi)^2}$. Indeed, if $\theta$ is fixed by $K$, then $|[\theta]|=1$, and $\frac{|K|}{e(\chi)^2}=|[\chi]|$ is the number of characters over $\theta$ and, therefore, the term in parentheses is $\Omega_\chi(\ov a,\ov b)$.

Let us look at its normalized Fourier transform $\mcF\Omega_\chi: \Irr[K]\times \Irr[K]\to\C$ defined by
\[\mcF\Omega_\chi(\xi_1,\xi_2):=\chi(1)\int_G\int_G\xi_1(x)\xi_2(y)\chi([x,y])\d x\d y,\]
where, as usual, $\xi_i(x)$ stands for $\xi_i(\pi(x))$. They are related by
\begin{equation}\label{eq:fourier-omega}
\Omega_\chi(a,b)=\frac{1}{e(\chi)^2}\sum_{\xi_1,\xi_2\in\Irr[K]}\mcF\Omega_{\chi}(\xi_1,\xi_2)\ov{\xi_1(a)}\ov{\xi_2(b)}
\end{equation}
for any $a,b\in K$.

\begin{prop}\label{key-argument}
    Let $G\to K$ be a Clifford setting. For any $\chi\in\Irr$, $\mcF\Omega_\chi \in \bilplus[{\Stab[\chi]}]$.
\end{prop}
\begin{proof}
    Let $\rho: G\to \GL(V)$ be the irreducible representation associated with $\chi$. Since $\chi([x,y])=\tr\big(\rho(xyx^{-1})\rho(y)^{-1}\big)$, we can write
    \begin{align*}
        \mcF\Omega_\chi(\xi_1,\xi_2) &= \chi(1)\tr\left(\int_G\xi_2(y)\left(\int_G\xi_1(x)\rho(xyx^{-1})\d x\right)\rho(y)^{-1}\d y\right)\\
        &= \chi(1)\tr\left(\int_G\xi_2(y)\,U_y\,\rho(y)^{-1}\d y\right),
    \end{align*}
    where
    \[U_y := \int_G \xi_1(x)\rho(xyx^{-1})\d x\]
    for $y\in G$. The rest of the proof consists in computing $U_y$ first and then the outer integral. For this, we are going to use two standard consequences of Schur's lemma and of Schur's orthogonality relations:
    \begin{align}
        \int_G\rho(x)M\rho(x)^{-1}\d x&=\frac{\tr(M)}{\chi(1)}\operatorname{Id},\text{ and}\label{eq:schur-average}\\
        \int_G\tr(\rho(x)M)\rho(x)^{-1}\d x&=\frac{M}{\chi(1)}\label{eq:schur-orthogonality}
    \end{align}
    for any endomorphism $M$ of $V$. Note that
    \[\xi_1(h)\rho(h)U_y\rho(h)^{-1}= U_y\]
    for any $h\in G$. In other words, if $V_{\xi_1}$ denotes $V$ with the action $g\cdot v= \xi_1(g)gv$, then $U_y: V\to V_{\xi_1}$ is an equivariant morphism. Hence, by Schur's lemma, $U_y=0$ if $\xi_1\chi\not=\chi$, and in that case $\mcF\Omega_\chi(\xi_1,\xi_2)=0$. From now on, assume that $\xi_1\in \Stab$ and pick an isomorphism $\phi_1 :V\to V_{\xi_1}$, so that $\xi_1(x)\rho(x)=\phi_1\rho(x)\phi_1^{-1}$ for $x\in G$. By \eqref{eq:schur-average} with $M=\phi_1^{-1}\rho(y)$,
    \begin{align*}
        U_y&=\int_G\xi_1(x)\rho(x)\rho(y)\rho(x)^{-1}\d x\\
        &=\phi_1\int_G\rho(x)\phi_1^{-1}\rho(y)\rho(x)^{-1}\d x\\
        &=\phi_1\,\frac{\tr(\phi_1^{-1}\rho(y))}{\chi(1)}\operatorname{Id}\\
        &=\frac{\tr(\rho(y)\phi_1^{-1})}{\chi(1)}\,\phi_1.
    \end{align*}

    Substituting in the formula for $\mcF\Omega_\chi$ above, it follows that
    \begin{align*}
        \mcF\Omega_\chi(\xi_1,\xi_2) &= \int_G \xi_2(y)\tr(\rho(y)\phi_1^{-1})\tr(\phi_1 \rho(y)^{-1}) \d y.
    \end{align*}
    Replacing $(x,y)$ by $(y,x)$ in the definition and using $\chi([y,x])=\ov{\chi([x,y])}$ gives $\mcF\Omega_\chi(\xi_1,\xi_2)=\ov{\mcF\Omega_\chi(\xi_2^{-1},\xi_1^{-1})}$, which vanishes if $\xi_2\notin\Stab$. Otherwise, let $\phi_2:V\to V_{\xi_2}$ be an isomorphism, so that $\xi_2(y)\rho(y)=\phi_2\rho(y)\phi_2^{-1}$ for $y\in G$. Then, by the cyclicity of the trace and \eqref{eq:schur-orthogonality} applied to $\rho$ and $M=\phi_2^{-1}\phi_1^{-1}\phi_2$,
    \begin{align*}
        \mcF\Omega_\chi(\xi_1,\xi_2)&=\int_G\tr\big(\xi_2(y)\rho(y)\phi_1^{-1}\big)\tr(\phi_1\rho(y)^{-1})\d y\\
        &=\int_G\tr(\phi_2\rho(y)\phi_2^{-1}\phi_1^{-1})\tr(\phi_1\rho(y)^{-1})\d y\\
        &=\tr\left(\phi_1\int_G\tr(\rho(y)\phi_2^{-1}\phi_1^{-1}\phi_2)\rho(y)^{-1}\d y\right)\\
        &=\frac{1}{\chi(1)}\tr(\phi_1\phi_2^{-1}\phi_1^{-1}\phi_2)\\
        &=\frac{1}{\chi(1)}\tr(\phi_2\phi_1\phi_2^{-1}\phi_1^{-1}).
    \end{align*}
    Now, by Schur's lemma, $\phi_2\phi_1\phi_2^{-1}\phi_1^{-1}$ is multiplication by $\mcF\Omega_\chi(\xi_1,\xi_2)$. Note that $\mcF\Omega_\chi$ does not depend on the choice of $\phi_1$ and $\phi_2$.

    We now check the properties of $\mcF\Omega_\chi$. We already proved that it vanishes outside $\Stab\times\Stab$. It is clear that $\mcF\Omega_\chi(\xi_1,\xi_1)=1$ because $\phi_1$ and $\phi_2$ commute in this case. To check bilinearity, note that if $\phi_3:V\to V_{\xi_3}$ is an isomorphism, $\phi_2\phi_3:V\to V_{\xi_2\xi_3}$ is an isomorphism and 
    \begin{align*}
        \mcF\Omega_\chi(\xi_2\xi_3,\xi_1)&=\phi_1(\phi_2\phi_3)\phi_1^{-1}(\phi_2\phi_3)^{-1}\\
        &=\phi_1\phi_2\phi_1^{-1}(\phi_1\phi_3\phi_1^{-1}\phi_3^{-1})\phi_2^{-1}\\
        &=\mcF\Omega_\chi(\xi_3,\xi_1)\mcF\Omega_\chi(\xi_2,\xi_1).
    \end{align*}
    Finally, the anti-symmetry follows from $(\phi_2\phi_1\phi_2^{-1}\phi_1^{-1})^{-1}=\phi_1\phi_2\phi_1^{-1}\phi_2^{-1}$. Together with the previous identity, it gives the bilinearity.
\end{proof}

\begin{lema}\label{ker-FOmega}
    Let $G\to K$ be a Clifford setting and $\chi\in\Irr$. Then
    \[|\ker \mcF\Omega_\chi| = \frac{|\Stab|}{e(\chi)^2}\]
    where $\ker \mcF\Omega_\chi\subset \Stab$ is the set of those $\xi\in\Stab$ for which $\mcF\Omega_\chi(-,\xi)\equiv 1$ on $\Stab$. In particular, $\mcF\Omega_\chi$ is non-degenerate if and only if $|\Stab|=e(\chi)^2$.
\end{lema}
\begin{proof}
    For a fixed $\xi$, $\mcF\Omega_\chi(-,\xi)$ is a character of $\Stab$ by Proposition \ref{key-argument}. Hence,
    \[
    \sum_{\xi_1,\xi_2\in\Stab} \mcF\Omega_\chi(\xi_1,\xi_2) = 
    |\Stab|\cdot|\ker \mcF\Omega_\chi|.
    \]
    On the other hand, $\mcF\Omega_\chi$ vanishes outside $\Stab\times\Stab$ by Proposition \ref{key-argument}, and $\sum_{\xi\in\Irr[K]}\xi(x)$ is $|K|$ if $x\in H$ and zero otherwise. Hence, by the definition of $\mcF\Omega_\chi$,
    \begin{align*}
        \sum_{\xi_1,\xi_2\in\Stab} \mcF\Omega_\chi(\xi_1,\xi_2) &= \sum_{\xi_1,\xi_2\in\Irr[K]} \mcF\Omega_\chi(\xi_1,\xi_2) = \chi(1)|K|^2\int_H\int_H \chi(xyx^{-1}y^{-1})\d x\d y\\
        &= \chi(1)e(\chi)\sum_\theta \frac{1}{\theta(1)} = \left(\frac{|\Stab|}{e(\chi)}\right)^2
    \end{align*}
    where $\theta$ runs over the irreducible constituents of $\chi\big|_H$. For the last two equalities, we have used Theorem \ref{clifford}, that $\sum_{x,y\in H}\theta(xyx^{-1}y^{-1})=\frac{|H|^2}{\theta(1)}$ for any $\theta\in\Irr[H]$, a classical computation going back to Frobenius \cite{Frobenius}, and that $|\Stab|=\frac{|K|}{|[\chi]|}=e(\chi)^2|[\theta]|$.
\end{proof}

\begin{lema}\label{lem:omega-vanishing-chi}
    Let $G\to K$ be a Clifford setting and $\chi\in\Irr$. Let $\ov a,\ov b\in K^g$ be tuples whose coordinates generate $K$. Then $\Omega_\chi(\ov a,\ov b)= 0$ unless $|\Stab|=e(\chi)^2$. In this case, there exists a surjective morphism $()^*:K\to \Stab$ such that
    \[\Omega_\chi(a,b)=\mcF\Omega_\chi(b^*,a^*)\]
    for any $a,b\in K$. In particular, $\Omega_\chi\in\bilplus[K]$ and $|\ker\Omega_\chi|=\frac{|K|}{e(\chi)^2}$.
\end{lema}
\begin{proof}
    Write $S=\Stab$. By \eqref{eq:fourier-omega} and Proposition \ref{key-argument},
    \[\Omega_\chi(a,b)=\frac{1}{e(\chi)^2}\sum_{\xi_1,\xi_2\in S}\mcF\Omega_\chi(\xi_1,\xi_2)\ov{\xi_1(a)}\,\ov{\xi_2(b)}.\]
    Let $\kappa\in\ker\mcF\Omega_\chi$. By anti-symmetry, $\mcF\Omega_\chi(\kappa,-)\equiv1$ on $S$. So replacing $\xi_1$ by $\kappa\xi_1$ (resp. $\xi_2$ by $\kappa\xi_2$) gives $\Omega_\chi(a,b)=\ov{\kappa(a)}\,\Omega_\chi(a,b)$ (resp. $\Omega_\chi(a,b)=\ov{\kappa(b)}\,\Omega_\chi(a,b)$). Hence, if $\Omega_\chi(\ov a,\ov b)=\prod_i\Omega_\chi(a_i,b_i)\neq0$, every $\kappa\in\ker\mcF\Omega_\chi$ is trivial at the $a_i$ and $b_i$. These generate $K$, so $\ker\mcF\Omega_\chi$ is trivial, that is, $|S|=e(\chi)^2$ by Lemma \ref{ker-FOmega}.

    Assume now that $|S|=e(\chi)^2$, so that $\mcF\Omega_\chi$ is non-degenerate. For $a\in K$, the map $\xi\mapsto\xi(a)$ is a character of $S$, so there is a unique $a^*\in S$ with $\xi(a)=\mcF\Omega_\chi(a^*,\xi)$ for all $\xi\in S$. The map $()^*$ is a morphism. It is surjective because every character of the subgroup $S\subset\Irr[K]$, such as $\mcF\Omega_\chi(s,-)$, extends to a character of $\Irr[K]$, that is, to the evaluation at some $a\in K$. As $\ov{\xi_1(a)}=\mcF\Omega_\chi(\xi_1,a^*)$, the sum $\sum_{\xi_1\in S}\mcF\Omega_\chi(\xi_1,\xi_2)\ov{\xi_1(a)}=\sum_{\xi_1\in S}\mcF\Omega_\chi(\xi_1,\xi_2a^*)$ is $|S|$ if $\xi_2=(a^*)^{-1}$ and zero otherwise. Hence,
    \[\Omega_\chi(a,b)=\frac{|S|}{e(\chi)^2}\,\ov{(a^*)^{-1}(b)}=a^*(b)=\mcF\Omega_\chi(b^*,a^*).\]
    As $()^*$ is a morphism, $\Omega_\chi\in\bilplus[K]$. As $()^*$ is surjective and $\mcF\Omega_\chi$ is non-degenerate, $b\in\ker\Omega_\chi$ if and only if $b^*=1$, so $|\ker\Omega_\chi|=\frac{|K|}{|S|}=\frac{|K|}{e(\chi)^2}$.
\end{proof}

The last result finishes the proof of Theorem \ref{character-formula}: by Theorem \ref{clifford}, $|\Stab|=\frac{|K|}{|[\chi]|}=e(\chi)^2|[\theta]|$, so $|\Stab|=e(\chi)^2$ if and only if $\theta\in\Irr[H]^K$, and in this case $\Omega_\theta=\Omega_\chi$ by \eqref{eq:Omega-chi-theta}. Moreover, $\Omega_{\theta'}=\Omega_\theta$ for every $\theta'\in[\theta]$, as the conjugation by an element of $G$ preserves each coset $G_a$. Hence, $\Omega_\chi=|[\theta]|^2\,\Omega_\theta$ by \eqref{eq:Omega-chi-theta}, and Lemma \ref{lem:omega-vanishing-chi} shows that $\Omega_\theta(\ov a,\ov b)=0$ unless $\theta\in\Irr[H]^K$, for any tuples $\ov a,\ov b\in K^g$ whose coordinates generate $K$. We end this subsection with two observations for future use.

\begin{lema}\label{lem:Omega-two-generators}
    Let $H\to G\to K$ be a Clifford setting and $\theta\in\Irr[H]^K$. Assume that $K$ has at most two generators and choose an isomorphism $K\simeq \Z/n\Z\times \Z/m\Z$. Then 
    $\Omega_\theta(x,y)=\omega^{x_1y_2-x_2y_1}$ for any $x=(x_1,x_2),y=(y_1,y_2)\in\Z/n\Z\times\Z/m\Z$, 
    where $\omega$ is a root of unity of order $e(\theta)$.
\end{lema}
\begin{proof}
    Since $\Omega_\theta\in\bilplus[K]$, we already know from the example in Section \ref{sec:The formula} that
    \[\Omega_\theta(x,y)=\omega^{x_1y_2-x_2y_1}\]
    for some root of unity $\omega$, and that its kernel has size $\frac{nm}{(\ord\omega)^2}$. The result follows by comparing with $|\ker\Omega_\theta|=\frac{|K|}{e(\theta)^2}$. 
\end{proof}

\begin{lema}\label{lem:Omega-towers}
    Let $H\to G\xrightarrow{\pi} K$ be a Clifford setting, $K'\subset K$ a subgroup and $G'=\pi^{-1}(K')$. For $\theta\in\Irr[H]^K$, the function $\Omega_\theta$ of the Clifford setting $H\to G'\to K'$ is the restriction to $K'\times K'$ of the one of $H\to G\to K$.
\end{lema}
\begin{proof}
    By \eqref{eq:Omega-chi-theta}, both are given by $\frac{\theta(1)}{|H|^2}\sum_{x\in G_a,\,y\in G_b}\theta([x,y])$ for $a,b\in K'$, and the cosets $G_a$ and $G_b$ are the same in $G'$ and in $G$.
\end{proof}

\subsection{Proof of Proposition \ref{magic-formula}}\label{sec:proof-magic}

For an irreducible representation $\rho$ of $H$ and $g\in G$, we call an isomorphism from $\rho$ to $h\mapsto\rho(ghg^{-1})$ an intertwiner for $g$ with respect to $\rho$, or to the character it affords. If this character is fixed by $K$, then intertwiners exist for every $g\in G$ and, by Schur's lemma, they are unique up to a scalar.

\begin{lema}\label{lem:omega-commutators}
    Let $\theta\in\Irr[H]^K$ be afforded by $\rho:H\to\GL(V)$. For any $a,b\in K$, $x\in G_a$ and $y\in G_b$, and any intertwiners $A_x$ and $A_y$ for $x$ and $y$ with respect to $\rho$,
    \begin{equation}\label{eq:omega-intertwiners}
        \rho([x,y])=\Omega_\theta(a,b)\,[A_x,A_y].
    \end{equation}
\end{lema}
\begin{proof}
    Both $\rho([x,y])$ and $[A_x,A_y]$ are intertwiners for $[x,y]$ with respect to $\rho$. Hence, $\rho([x,y])=c(x,y)[A_x,A_y]$ for some $c(x,y)\in\C^\times$, which does not depend on the choice of $A_x$ and $A_y$, as commutators do not change when the operators are rescaled. Let us prove that $c(x,y)$ only depends on $\pi(x)$ and $\pi(y)$. If $h\in H$, then
    \begin{align*}
        \rho([xh,y])&=\rho(x[h,y]x^{-1})\rho([x,y])\\
        &=A_x\rho([h,y])A_x^{-1}\rho([x,y])\\
        &=c(x,y)\,A_x\rho(h)A_y\rho(h)^{-1}A_y^{-1}A_x^{-1}A_xA_yA_x^{-1}A_y^{-1}\\
        &=c(x,y)\,[A_x\rho(h),A_y],
    \end{align*}
    where we have used that $\rho(xzx^{-1})=A_x\rho(z)A_x^{-1}$ and $\rho(yzy^{-1})=A_y\rho(z)A_y^{-1}$ for $z\in H$. Now, note that $A_x\rho(h)$ is an intertwiner for $xh$. This implies that $c(xh,y)=c(x,y)$. As $[y,x]=[x,y]^{-1}$, we have $c(y,x)=c(x,y)^{-1}$, so also $c(x,yh)=c(x,y)$, and our claim follows.

    By this claim and the definition of $\Omega_\theta$ in \eqref{eq:Omega-chi-theta}, it suffices to show that
    \[\sum_{x\in G_a,\,y\in G_b}\tr[A_x,A_y]=\frac{|H|^2}{\theta(1)}\]
    for some choice of the intertwiners. Fix $x_0\in G_a$ and $y_0\in G_b$ with intertwiners $A_{x_0}$ and $A_{y_0}$, and put $A_{x_0h}=A_{x_0}\rho(h)$ and $A_{y_0h}=A_{y_0}\rho(h)$ for $h\in H$. By \eqref{eq:schur-average}, applied to $H$ and $\rho$,
    \[\sum_{y\in G_b}A_yA_x^{-1}A_y^{-1}=\frac{|H|\tr(A_x^{-1})}{\theta(1)}\operatorname{Id}\]
    for every $x\in G_a$. Hence, by \eqref{eq:schur-orthogonality}, applied to $H$ and $\rho$,
    \[\sum_{x\in G_a,\,y\in G_b}\tr[A_x,A_y]=\frac{|H|}{\theta(1)}\sum_{h\in H}\tr\big(\rho(h)A_{x_0}\big)\tr\big(\rho(h)^{-1}A_{x_0}^{-1}\big)=\frac{|H|^2}{\theta(1)}.\]
\end{proof}

We now study how $\Omega_{\xi\theta}$ depends on $\xi\in\Xi_\theta$. By Lemma \ref{lem:omega-commutators}, applied to $\xi\theta$, $\Omega_{\xi\theta}$ can be computed from intertwiners with respect to $\xi\theta$. We build them from an extension to $G$ of a representation $\rho$ affording $\theta$, and from operators intertwining $\rho$ and $\lambda\rho$ for the linear characters $\lambda$ with $\lambda\theta=\theta$.

From now on, fix an unramified character $\theta\in\Irr[H]^K$ afforded by $\rho:H\to\GL(V)$. As $\theta$ is fixed by $K$ and unramified, any $\chi\in\Irr$ over $\theta$ satisfies $\chi\big|_H=\theta$ by Theorem \ref{clifford}. Hence, there is a representation $\tilde\rho:G\to\GL(V)$ with $\tilde\rho\big|_H=\rho$, and $\tilde\rho(x)$ is an intertwiner for $x$ with respect to $\rho$ for every $x\in G$. For $u\in\GL(V)$ and $g\in G$, we write $u^g=\tilde\rho(g)^{-1}u\,\tilde\rho(g)$. Define
\[\Lambda_\theta=\{\lambda\in\Hom{H}{\C^\times}:\lambda\theta=\theta\}.\]
For $\lambda\in\Lambda_\theta$, we call an operator $T\in\GL(V)$ such that $T\rho(h)T^{-1}=\lambda(h)\rho(h)$ for every $h\in H$ an intertwiner for $\lambda$. It exists because $\lambda\rho$ affords $\lambda\theta=\theta$, and it is unique up to a scalar by Schur's lemma. Note that if $T$ and $T'$ are intertwiners for $\lambda$ and $\lambda'$, then $TT'$ is an intertwiner for $\lambda\lambda'$, and $T^g$ is an intertwiner for $k\cdot\lambda$ for $g\in G_k$. Finally, for $\xi\in\Hom{H}{\C^\times}$ and $a\in K$, put $\partial_a\xi=(a\cdot\xi)\xi^{-1}$.

\begin{lema}\label{lem:twist-basics}
    In the previous situation, the following hold.
    \begin{enumerate}
        \item $\Xi_\theta$ is a subgroup of $\Hom{H}{\C^\times}$. More precisely, $\xi\in\Xi_\theta$ if and only if $\partial_a\xi\in\Lambda_\theta$ for every $a\in K$.
        \item $\Lambda_\theta\subset\Xi_\theta$, and both are $K$-stable subgroups.
        \item For every $a\in K$, $\xi\mapsto\partial_a\xi$ is a group morphism, and $\partial_{ab}\xi=(b\cdot\partial_a\xi)\,\partial_b\xi$ for $a,b\in K$.
    \end{enumerate}
\end{lema}
\begin{proof}
    Note that $a\cdot(\xi\theta)=(a\cdot\xi)(a\cdot\theta)=(a\cdot\xi)\theta$, as $\theta$ is fixed by $K$. Hence, $\xi\theta$ is fixed by $a$ if and only if $\partial_a\xi\in\Lambda_\theta$. As $\partial_a$ is a morphism (since $K$ acts by automorphisms), $\Xi_\theta=\bigcap_{a\in K}\partial_a^{-1}(\Lambda_\theta)$ is a subgroup. This proves (1) and the first half of (3). If $\lambda\in\Lambda_\theta$, then $\lambda\theta=\theta$ is fixed by $K$, so $\lambda\in\Xi_\theta$. The subgroups are $K$-stable because $(b\cdot\lambda)\theta=b\cdot(\lambda\theta)$ and $(b\cdot\xi)\theta=b\cdot(\xi\theta)$. For the second half of (3), $\partial_{ab}\xi=\big(b\cdot(a\cdot\xi)\big)(b\cdot\xi)^{-1}\,(b\cdot\xi)\xi^{-1}$.
\end{proof}

\begin{lema}\label{lem:twist-omega}
    In the previous situation, let $\xi\in\Xi_\theta$, $a,b\in K$, $x\in G_a$ and $y\in G_b$. For $k\in K$, let $t_k$ be an intertwiner for $\partial_k\xi$. Then $d_\xi(x,y)=t_a^y\,t_b\,\big(t_b^x\,t_a\big)^{-1}$ is a scalar, and
    \begin{equation}\label{eq:twist-omega-s}
        \Omega_{\xi\theta}(a,b)=\frac{\xi([x,y])}{d_\xi(x,y)}.
    \end{equation}
\end{lema}
\begin{proof}
    For $g\in G_k$, put $A_g=\tilde\rho(g)t_k$. For $h\in H$,
    \[A_g\,\xi(h)\rho(h)\,A_g^{-1}=(k\cdot\xi)(h)\,\rho(ghg^{-1})=(\xi\rho)(ghg^{-1}),\]
    so $A_g$ is an intertwiner for $g$ with respect to $\xi\rho$, which affords $\xi\theta\in\Irr[H]^K$. Now, $A_xA_y=\tilde\rho(xy)\,t_a^y\,t_b$ and $A_yA_x=\tilde\rho(yx)\,t_b^x\,t_a$. Both $t_a^yt_b$ and $t_b^xt_a$ are intertwiners for $\partial_{ab}\xi$, as $\partial_{ab}\xi=(b\cdot\partial_a\xi)\,\partial_b\xi=(a\cdot\partial_b\xi)\,\partial_a\xi$ by Lemma \ref{lem:twist-basics}(3). Hence, $d_\xi(x,y)$ is a scalar and, as $\tilde\rho(xy)\tilde\rho(yx)^{-1}=\rho([x,y])$, we have $[A_x,A_y]=d_\xi(x,y)\,\rho([x,y])$. The result follows from \eqref{eq:omega-intertwiners} for $\xi\theta$ and $\xi\rho$.
\end{proof}

\begin{thm}\label{thm:twist-quadratic}
    Let $H\to G\to K$ be a Clifford setting and $\theta\in\Irr[H]^K$ be an unramified character. Then, for any tuples $\ov a,\ov b\in K^g$, the map $\xi\mapsto\Omega_{\xi\theta}(\ov a,\ov b)$ is a quadratic function on $\Xi_\theta$ which is constant on the cosets of $\Lambda_\theta$.
\end{thm}
\begin{proof}
    The map is constant on the cosets of $\Lambda_\theta$ because $\xi\lambda\theta=\xi\theta$ for $\lambda\in\Lambda_\theta$. As a product of quadratic functions is a quadratic function, it is enough to consider $g=1$. Fix $a,b\in K$, $x\in G_a$ and $y\in G_b$, and put $f(\xi)=\Omega_{\xi\theta}(a,b)$. By Theorem \ref{character-formula}, applied to $\xi\theta\in\Irr[H]^K$, $f$ takes values in $\U(1)$, and, by Lemma \ref{lem:twist-omega} with $t_k=\operatorname{Id}$ for all $k$, $f(1)=1$. Clearly, $B_f$ is symmetric, so it remains to show that it is multiplicative in its first variable.

    Let $\xi_1,\xi_2\in\Xi_\theta$, and let $p_k$ and $q_k$ be intertwiners for $\partial_k\xi_1$ and $\partial_k\xi_2$, respectively. By Lemma \ref{lem:twist-basics}(3), $p_kq_k$ is an intertwiner for $\partial_k(\xi_1\xi_2)$. Moreover, let $u$ and $v$ be products of operators of the form $p_k^g$ and $q_{k'}^{g'}$, with $k,k'\in K$ and $g,g'\in G$. Then $u$ and $v$ are intertwiners for some $\lambda,\mu\in\Lambda_\theta$, so $uv$ and $vu$ are both intertwiners for $\lambda\mu$, and the commutator $[u,v]=uvu^{-1}v^{-1}=uv\,(vu)^{-1}$ is a scalar. We apply Lemma \ref{lem:twist-omega} to $\xi_1$, $\xi_2$ and $\xi_1\xi_2$, with the choices $p_k$, $q_k$ and $p_kq_k$. As $q_a^yp_b=[q_a^y,p_b]\,p_bq_a^y$ and $q_b^xp_a=[q_b^x,p_a]\,p_aq_b^x$, we get
    \begin{align*}
        d_{\xi_1\xi_2}(x,y)&=p_a^y\,q_a^y\,p_b\,q_b\,\big(p_b^x\,q_b^x\,p_a\,q_a\big)^{-1}\\
        &=\frac{[q_a^y,p_b]}{[q_b^x,p_a]}\,\big(p_a^y\,p_b\big)\big(q_a^y\,q_b\big)\big(q_b^x\,q_a\big)^{-1}\big(p_b^x\,p_a\big)^{-1}\\
        &=\frac{[q_a^y,p_b]}{[q_b^x,p_a]}\,\big(p_a^y\,p_b\big)\,d_{\xi_2}(x,y)\,\big(p_b^x\,p_a\big)^{-1}\\
        &=\frac{[q_a^y,p_b]}{[q_b^x,p_a]}\,d_{\xi_1}(x,y)\,d_{\xi_2}(x,y),
    \end{align*}
    where, in the last step, we have used that $d_{\xi_2}(x,y)$ is a scalar. As $\xi\mapsto\xi([x,y])$ is a character, \eqref{eq:twist-omega-s} gives
    \[B_f(\xi_1,\xi_2)=\frac{d_{\xi_1}(x,y)\,d_{\xi_2}(x,y)}{d_{\xi_1\xi_2}(x,y)}=\frac{[q_b^x,p_a]}{[q_a^y,p_b]}.\]
    Replacing $\xi_1$ by $\xi_1\xi'_1$, with the choice $p_kp'_k$, where $p'_k$ is an intertwiner for $\partial_k\xi'_1$, and using that $[u,vv']=[u,v][u,v']$ when $[u,v]$ and $[u,v']$ are scalars, we get $B_f(\xi_1\xi'_1,\xi_2)=B_f(\xi_1,\xi_2)B_f(\xi'_1,\xi_2)$. Hence, $B_f$ is bilinear, and $f$ is a quadratic function.
\end{proof}

\begin{coro}\label{coro:twist-character}
    Let $H\to G\to K$ be a Clifford setting, $\theta\in\Irr[H]^K$ be an unramified character and $\Xi\subset\Xi_\theta$ be a subgroup. Let $\Lambda_\Xi\subset\Lambda_\theta$ be the subgroup generated by $\partial_a\xi$ for $a\in K$ and $\xi\in\Xi$. If intertwiners for elements of $\Lambda_\Xi$ commute with each other, then $\xi\mapsto\Omega_{\xi\theta}(\ov a,\ov b)$ is a character of $\Xi$ for any tuples $\ov a,\ov b\in K^g$.
\end{coro}
\begin{proof}
    Let $\xi_1,\xi_2\in\Xi$. In the previous proof, $q_a^y$ and $p_b$ are intertwiners for $b\cdot\partial_a\xi_2$ and $\partial_b\xi_1$, and $b\cdot\partial_a\xi_2=\partial_{ab}\xi_2\,(\partial_b\xi_2)^{-1}$ lies in $\Lambda_\Xi$ by Lemma \ref{lem:twist-basics}(3). Hence, $[q_a^y,p_b]=1$ and, similarly, $[q_b^x,p_a]=1$. Thus, $B_f$ is trivial on $\Xi\times\Xi$, that is, each $\xi\mapsto\Omega_{\xi\theta}(a_i,b_i)$ is a character of $\Xi$.
\end{proof}

\begin{coro}\label{coro:twist-cyclic}
    Let $H\to G\to K$ be a Clifford setting and $\theta\in\Irr[H]^K$ be an unramified character. If $\Xi_\theta$ is cyclic, or more generally if $\Lambda_\theta$ is cyclic, then $\xi\mapsto\Omega_{\xi\theta}(\ov a,\ov b)$ is a character of $\Xi_\theta$ for any tuples $\ov a,\ov b\in K^g$.
\end{coro}
\begin{proof}
    As $\Lambda_\theta\subset\Xi_\theta$ by Lemma \ref{lem:twist-basics}(2), $\Lambda_\theta$ is cyclic if $\Xi_\theta$ is. If $\Lambda_\theta=\langle\lambda\rangle$ and $T$ is an intertwiner for $\lambda$, then $T^i$ is an intertwiner for $\lambda^i$, so every intertwiner for an element of $\Lambda_\theta$ is a multiple of a power of $T$. Hence, these intertwiners commute with each other, and Corollary \ref{coro:twist-character} applies to $\Xi=\Xi_\theta$.
\end{proof}

\begin{proof}[Proof of Proposition \ref{magic-formula} and Corollary \ref{coro:magic-cyclic}]
    By Lemma \ref{lem:twist-basics}(1), $\Xi_\theta$ is a subgroup of $\Hom{H}{\C^\times}$. The other claim of Proposition \ref{magic-formula} is Theorem \ref{thm:twist-quadratic}. Corollary \ref{coro:magic-cyclic} is a particular case of Corollary \ref{coro:twist-cyclic}.
\end{proof}

This completes the proof of Theorem \ref{intro:omega-properties}: its vanishing claim was shown after Lemma \ref{lem:omega-vanishing-chi}, items (1)--(3) are part of Theorem \ref{character-formula}, and item (4) is Theorem \ref{thm:twist-quadratic} with $g=1$, since $e(\chi)=1$ means that $\theta$ is unramified and the set in (4) is $\Xi_\theta$.

\section{The case of block-like groups}\label{sec:The case}

\subsection{Setup}

In the next subsections, we study a particular kind of Clifford setting, built out of a
cyclic one by taking products, a wreath product, and a further extension. Our main example, which is the one used in Section \ref{sec:stringy-points}, is a subgroup of $\SL_n(\ov{\F}_q)$ whose
elements are scalar multiples of elements of $\GL_{\frac{n}{A}}(\F_q) \wr F_A$; see the end of this subsection.

Let us first recall the regular wreath product of a group $G_0$ and a cyclic group $F_A$ of order $A$. Consider the product $G_0^A$. The group $F_A$ acts on it by cyclically permuting the factors; if $F_A=\langle \sigma \rangle$, $(\sigma \cdot g)_i=g_{i+1}$ is such an action. The regular wreath product is $G_0\wr F_A:=G_0^A\rtimes F_A$, the corresponding extension of $G_0^A$. There is a natural inclusion $G_0^A\to G_0\wr F_A$, $g\mapsto (g,1)$.

Now, let us assume that $G_0$ fits in a Clifford setting
\[H_0\to G_0\xrightarrow{\pi_0}K_0\]
with $K_0$ cyclic.
Then, we get another Clifford setting
\begin{equation}\label{block-clifford}
    H\to G_0\wr F_A\to F_A\times K_0
\end{equation}
where the second map is given by $(g_1,\ldots, g_A,\sigma^j)\mapsto(\sigma^j,\pi_0(g_1)\cdots \pi_0(g_A))$ and $H$ is simply the kernel of this map.

We are going to twist the previous construction. Fix central elements $c_1,\ldots, c_A$ of $G_0$ such that $\pi_0(c_1\cdots c_A)^A = 1$. Let $\hat \sigma =(c_1,\ldots, c_A, \sigma)\in G_0\wr F_A$. Note that $\hat\sigma^A\in H$ and that $\hat\sigma$ acts by conjugation on $G_0^A$ as $\sigma$ does. Define $\hat H$ as the subgroup generated by $H$ and $\hat\sigma$. We have another Clifford setting
\begin{equation}\label{hatH-clifford}
    \hat H \to G_0\wr F_A\to K_0
\end{equation}
where the second map is given by $(g_1,\ldots,g_A,\sigma^j)\mapsto\pi_0(g_1)\cdots\pi_0(g_A)\,\pi_0(c_1\cdots c_A)^{-j}$.
Now let $G$ be a group containing $\hat H$ such that $H$ and $\hat H$ are normal in $G$ and $K=G/H$ is abelian. We will say that
\begin{equation}\label{block-like-clifford}
    H\to G\xrightarrow{\pi} K
\end{equation}
is a block-like Clifford setting if it behaves in a similar fashion to $H\to G_0\wr F_A\to F_A\times K_0$ in the following sense:
\begin{enumerate}
    \item\label{block-like:actions} There exist a quotient $p:K_0\to \ov{K}_0$, an action of $F_A\times \ov{K}_0$ on $\Irr[H]$, and an injective morphism $\iota: K\to F_A\times \ov{K}_0$ such that
    \[ \xymatrix{
    K \ar[rrd]\ar[rr]^{\iota} && F_A\times\ov{K}_0\ar[d] && F_A\times K_0 \ar[ll]_{\operatorname{id}_{F_A}\times p}\ar[lld]\\
    && S(\Irr[H]) &&
    } \]
    commutes, where the vertical arrows are given by the actions; the right and left ones are the Clifford ones for (\ref{block-clifford}) and (\ref{block-like-clifford}) respectively.
    \item\label{block-like:hatH} There exists an action of $\ov{K}_0$ on $\Irr[\hat H]$ such that
    \[ \xymatrix{
    K \ar[rrd]\ar[rr]^{\mathrm{pr}_2\circ\iota} && \ov{K}_0\ar[d] && K_0 \ar[ll]_{p}\ar[lld]\\
    && S(\Irr[\hat H]) &&
    } \]
    commutes, where $\mathrm{pr}_2$ is the second projection and the vertical arrows are given by the actions; the right one is the Clifford one for (\ref{hatH-clifford}), and the left one is induced by conjugation.
    \item\label{block-like:sigma} $\iota(\pi(\hat\sigma))=(\sigma,1)$.
\end{enumerate}
By condition (\ref{block-like:sigma}), $\iota(K)$ contains $F_A\times\{1\}$. Hence, $\iota(K)=F_A\times F_B$, where $F_B$ is the image of $K$ in $\ov{K}_0$, a cyclic group of a certain order $B$, and $\iota$ induces an isomorphism $K\simeq F_A\times F_B$. We fix a generator $\delta$ of $F_B$ and an element $\gamma\in G$ whose image in $K$ corresponds to $(1,\delta)$, and identify $K$ with $\Z/A\Z\times \Z/B\Z$, so that $\sigma$ and $\delta$ correspond to $(1,0)$ and $(0,1)$.
More generally, we also consider intermediate extensions $H\subset G'\subset G$ such that both maps $K'=G'/H\to F_A$ and $K'\to F_B$ are surjective.

There is a trivial example where one puts $\hat\sigma=\sigma$, $G=G_0\wr F_A$ and $\ov{K}_0=K_0$. A more intricate one is the following. It is the situation where we are going to apply the results of this section. Fix an integer $n$, one of its divisors $A$, and a prime power $q$ with $n|q-1$. Consider $G_0=\GL_{\frac{n}{A}}(\F_q)$, $H_0=\SL_{\frac{n}{A}}(\F_q)$ and $K_0=\F_q^\times$. The map $\pi_0$ is the determinant. Here $G_0\wr F_A$ is the subgroup of $\GL_n(\F_q)$ generated by the invertible $A\times A$-block-diagonal matrices, i.e., matrices of the form 
\[\left(\begin{array}{cccccc}
    M_1 & 0 & \ldots & 0 & 0 \\
    0 & M_2 & \ddots & \vdots & \vdots \\
    \vdots & \ddots & \ddots & 0 & 0 \\
    0 & \hdots & 0 & M_{A-1} & 0 \\
    0 & \hdots & 0 & 0 & M_A
\end{array}\right)\]
with $M_i\in \GL_{\frac{n}{A}}(\F_q)$, and the permutation matrix $\sigma$
\[\left(\begin{array}{ccccc}
    0 & \operatorname{Id}_{\frac{n}{A}} & 0  & \hdots & 0 \\
    0 & 0 & \operatorname{Id}_{\frac{n}{A}} &  \ddots & \vdots \\ 
    \vdots & \ddots  & \ddots & \ddots & 0 \\  
    0 & 0 & 0 &  \ddots & \operatorname{Id}_{\frac{n}{A}} \\
    \operatorname{Id}_{\frac{n}{A}} & 0 & 0  & \hdots & 0 
\end{array}\right)\]
which has order $A$. The group $H$ is thus that of block-diagonal matrices of determinant one. We take $\hat{H}$ as the subgroup of matrices of determinant one of $G_0\wr F_A$. This is a valid choice. Indeed, let $\hat\sigma$ be the matrix obtained from $\sigma$ by multiplying its lower-left block by $(-1)^{A-1}$. It is the product of $\sigma$ and a central element of $G_0^A$, and $\det\hat\sigma=(-1)^{\frac{n}{A}(A-1)}\det\sigma=1$, as $\sigma$ is the permutation matrix of a product of $\frac{n}{A}$ disjoint cycles of length $A$. Hence, $\hat{H}$ is generated by $H$ and $\hat\sigma$, as any matrix of $G_0\wr F_A$ is the product of a block-diagonal matrix and a power of $\hat\sigma$. Moreover, $\hat\sigma^A=(-1)^{A-1}\operatorname{Id}_n\in H$. Note that $\sigma\notin\hat{H}$ if $A$ is even and $\frac{n}{A}$ is odd. 

An extension $G$ giving a block-like Clifford setting can be built as follows. Note that any $n$-th power $\mu^n\in K_0=\F_q^\times$ is the image in $K_0$ of the scalar matrix $\mu\operatorname{Id}_n\in G_0^A$, which is central in $G_0\wr F_A$. Hence, $n$-th powers act trivially on $\Irr[\hat{H}]$ and on $\Irr[H]$, and both actions factor through $\ov{K}_0:=\F_q^\times / (\F_q^\times)^n \simeq F_n$. Let $\gamma_0\in \F_{q^n}^\times$ with $\gamma_0^n\in\F_q$ and $\gamma$ be the diagonal matrix
\[\left(\begin{array}{cccccc}
    \gamma_0 & 0 & \ldots & 0 & 0 \\
    0 & \gamma_0 & \ddots & \vdots & \vdots \\
    \vdots & \ddots & \ddots & 0 & 0 \\
    0 & \hdots & 0 & \gamma_0 & 0 \\
    0 & \hdots & 0 & 0 & \gamma_0^{1-n}
\end{array}\right)\]
We let $G$ be the subgroup of $\SL_n(\ov{\F}_q)$ generated by $H$, $\hat\sigma$ and $\gamma$. In the next lemma we check that $K=G/H$ is isomorphic to $F_A\times \langle \pi(\gamma) \rangle$, where $\pi:G\to K$ is the quotient map. Finally, $\gamma$ is the product of a central matrix of $\GL_n(\ov{\F}_q)$ and the diagonal matrix $(1,\ldots,1,\gamma_0^{-n})$ of $G_0^A$, whose image in $\ov{K}_0$ is $\delta:=[\gamma_0^{-n}]$. Hence, the conjugation by $\gamma$ agrees with the conjugation by an element of $G_0^A$ whose image in $\ov{K}_0$ is $\delta$ and, in particular, $\pi(\gamma)$ acts on $\Irr[H]$ and on $\Irr[\hat H]$ as $\delta$. Thus, by the next lemma, $\pi(\hat\sigma^i\gamma^j)\mapsto(\sigma^i,\delta^j)$ defines an injective morphism $\iota:K\to F_A\times\ov{K}_0$, and $H\to G\to K$ is a block-like Clifford setting. Note that the order of $\delta$, and that of $\pi(\gamma)$, is the least $B$ for which $\gamma_0\in\F_{q^B}$. Indeed, $\zeta=\gamma_0^{q-1}$ satisfies $\zeta^n=(\gamma_0^n)^{q-1}=1$, so $\zeta\in\F_q$ and $\frob^j(\gamma_0^i)=\zeta^{ij}\gamma_0^i$. Hence $\gamma_0\in\F_{q^j}$ if and only if $\ord\zeta\mid j$, if and only if $\gamma_0^j\in\F_q$. The latter holds if and only if $\gamma^j\in H$, and also, as $\mu_n\subset\F_q$, if and only if $\gamma_0^{-nj}\in(\F_q^\times)^n$, that is, $\delta^j=1$. 

\begin{lema}\label{lem:example-is-block-like}
    In the situation described above, $H$ is normal in $G$ and $K\simeq F_A\times F_B$.
\end{lema}
\begin{proof}
    Note that $\gamma$ is a product of a central matrix in $\GL_n(\ov{\F}_q)$ and one in $G_0^A$. Namely, $\gamma_0\operatorname{Id}_n$ and the diagonal matrix $(1,\ldots,1,\gamma_0^{-n})$. In addition, $\hat\sigma$ acts on $G_0^A$ as $\sigma$. Hence, $H$ is normal in $G$, as it is normal in $G_0^A$ and $\sigma$-invariant. Let $K=G/H$. Note that $\gamma^B\in H$ but none of its lower powers is, and similarly for $\hat\sigma^A$. Moreover, $[\hat\sigma,\gamma]=[\sigma,(1,\ldots,1,\gamma_0^{-n})]$ lies in $G_0^A$ and has determinant one, so it lies in $H$. This shows that $K$ is abelian and has two generators of orders $A$ and $B$. 
    
    Now assume that $\hat\sigma^i\gamma^j\in H$. Taking Frobenius we find that $\frob(\gamma_0^{j})=\gamma_0^{j}$. Then $B|j$ and $\hat\sigma^i\in H$ which implies $A|i$. Hence $K=F_A\times F_B$.  
\end{proof}

\subsection{Classification of fixed characters}

The main result of this subsection is a classification of all $K'$-fixed characters of $H$ in the setting discussed above. We remark that characters of wreath products are classified in general, see for example \cite[Subsection 2.4.1]{Wreath}. Our approach is somewhat similar but focusing on invariant characters makes it simpler. 

Note that $\Irr[G_0^A]=\Irr[G_0]^A$. Hence, given irreducible characters $\eta_0$ and $\xi$ of $G_0$ and $K_0$ respectively, we can build an irreducible character $\eta_{\eta_0,\xi}$ of $G_0^A$ by
\[\eta_{\eta_0,\xi}(g_1,\ldots,g_A)=\prod_{i=1}^A \xi^{i-1}(\pi_0(g_i))\eta_0(g_i)\]
for $(g_1,\ldots,g_A)\in G_0^A$. Note that
\[(\sigma\cdot\eta_{\eta_0,\xi})(g_1,\ldots,g_A)=\xi^{A}(\pi_0(g_1))\,\xi^{-1}\left(\pi_0(g_1)\cdots\pi_0(g_A)\right)\eta_{\eta_0,\xi}(g_1,\ldots,g_A),\]
and $\Stab[\eta_{\eta_0,\xi}]=\Stab[\eta_0]$. Denote by $I_{A,B}$ the set of pairs $(\eta_0,\xi)$ such that $\xi^A\in \Stab[\eta_0]$ and $|\Stab[\eta_0]|$ divides $\frac{|\ov{K}_0|}{B}$. The previous construction gives us a map $\Theta: I_{A,B}\to \Irr[H]/K_0$ given by sending a pair $(\eta_0,\xi)$ to the irreducible constituents of $\eta_{\eta_0,\xi}\big|_H$. 

\begin{thm}\label{clasf-char}
    In a block-like Clifford setting, $\Irr[H]^{K'} = \Irr[H]^K$ and $\Theta$ induces a bijective correspondence between $\Irr[H]^K/K_0$ and $I_{A,B}/\sim$ where $(\eta_0',\xi')\sim (\eta_0,\xi)$ if and only if $\eta_0$ is in the $\Irr[K_0]$-orbit of $\eta_0'$ and $\xi^{-1}\xi'\in \Stab[\eta_0]$. 
\end{thm}

We write $\theta_{\eta_0,\xi}$ for any irreducible constituent of $\eta_{\eta_0,\xi}\big|_H$ and $\chi_{\eta_0,\xi}$ for any irreducible character of $G'$ over $\theta_{\eta_0,\xi}$. The proof is done in several steps.

\begin{lema}\label{lem:fixed-by-sigma-delta}
    In a block-like Clifford setting, let $\theta$ be an irreducible character of $H$ fixed by $\sigma\delta^\ell$ for some integer $\ell$. Then $\theta$ is fixed by $\sigma$ and there exist irreducible characters $\eta_0$ of $G_0$ and $\xi$ of $K_0$ with $\xi^A\cdot\eta_0=\eta_0$ such that $\theta$ is an irreducible constituent of $\eta_{\eta_0,\xi}\big|_H$.
\end{lema}
\begin{proof}
    If $A=1$, there is nothing to prove. Hence, we may assume that $A>1$. Consider the embedding $H \subset G_0^A$, which is a cyclic extension. There exists a character $\eta=(\eta_i)\in \Irr[G_0^A]=\Irr[G_0]^A$ such that $\eta\big|_H$ has $\theta$ as an irreducible constituent. The action of $\sigma$ is given by cyclically shifting the $\eta_i$. 

    The action of $K$ commutes with the action of $K_0$. Hence, if $\sigma\delta^\ell$ fixes $\theta$, it fixes $\eta\big|_H$ as well. Moreover, $\delta^\ell$ fixes $\eta\big|_H$ because $K_0$ fixes it. Hence, $\sigma$ fixes $\eta\big|_H$. Thus, $\sigma\cdot\eta$ and $\eta$ are irreducible characters of $G_0^A$ with the same restriction to $H$, and Theorem \ref{clifford}, applied to $H\subset G_0^A$, gives $\xi\in\Irr[K_0]$ with $\sigma\cdot\eta=\xi^{-1}\eta$. As $\sigma\cdot(\eta_1,\ldots,\eta_A)=(\eta_A,\eta_1,\ldots,\eta_{A-1})$, this means that $\eta_j=\xi\eta_{j-1}$ for $2\le j\le A$ and $\eta_1=\xi\eta_A$. Hence, $\eta_j=\xi^{j-1}\eta_1$ for $1\le j\le A$ and $\xi^A\eta_1=\eta_1$. Set $\eta_0=\eta_1$, so that $\eta=\eta_{\eta_0,\xi}$. 

    To finish we need to prove that $\theta$ is fixed by $\sigma$. We are going to use the subgroup $H_0^A$ of $H$. The restriction of $\eta$ to this subgroup is easy to describe. Write $\eta_0\big|_{H_0}$ as a sum of its irreducible constituents $\kappa_1,\ldots,\kappa_m$. Note that $m=|\Stab[\eta_0]|=|\Stab[\eta]|$. We have
    \[\eta(g_1,\ldots,g_A)=\prod \eta_0(g_i) \]
    for $(g_i)\in H_0^A$. Hence, the irreducible constituents of 
    $\eta\big|_{H_0^A}$  are the products 
    \[\kappa_{\phi}(g_1,\ldots,g_A)= \prod_{i=1}^A \kappa_{\phi(i)}(g_i) \]
    where $\phi$ runs over $(\Z/m\Z)^A$. In particular, $\eta$ is unramified over $H_0^A$. Hence, as $\eta\big|_H$ is $\sigma$-invariant, it is enough to show that $\theta\big|_{H_0^A}$ is $\sigma$-invariant. 

    We may assume that $\kappa_i=\beta^{i-1}\kappa_1$ for $\beta$ a generator of $\Z/m\Z = K_0/\Stab[\kappa_1]$. The quotient $H/H_0^A$ is the subgroup of $K_0^A$ of tuples with product one. The action on $\kappa_\phi$ factors through the subgroup $K''\subset (\Z/m\Z)^A$ of tuples of sum zero and is given by translating $\phi$. Hence, the orbits are contained inside the $m$ cosets of $K''$. But each $\theta$ gives us a different orbit and there are $m$ of them. Hence, 
    \[\theta\big|_{H_0^A} = \sum_{\phi\in \phi_0+K''}\kappa_\phi\]
    for some $\phi_0\in(\Z/m\Z)^A$, which is invariant under $\sigma$.
\end{proof}

\begin{lema}\label{lem:fixed-by-Kprime}
    In a block-like Clifford setting, let $\theta$ be an irreducible character of $H$ fixed by $K'$. Then it is fixed by $K$ and there exists a pair $(\eta_0,\xi)\in I_{A,B}$ such that $\theta$ is an irreducible constituent of $\eta_{\eta_0,\xi}\big|_H$. 
\end{lema}
\begin{proof}
    As $K'\to F_A$ is surjective, $K'$ contains $\sigma\delta^{\ell}$ for some integer $\ell$. Hence, we can apply Lemma \ref{lem:fixed-by-sigma-delta} with $\sigma\delta^{\ell}$ to find $\eta_0$ and $\xi$. The lemma also ensures that $\theta$ is fixed by $\sigma$. As $K'\to F_B$ is surjective, $K$ is generated by $K'$ and $\sigma$. This proves that $K$ fixes $\theta$.
    
    Finally, since $\ov{K}_0$ is cyclic, $\theta$ is fixed by $\delta$ if and only if $\ord\delta$ divides $\frac{|\Stab[\theta]||\ov{K}_0|}{|K_0|}$, where $\Stab[\theta]$ is the stabilizer of $\theta$ in $K_0$, that is, if and only if $|\Stab[\eta_0]|$ divides $\frac{|\ov{K}_0|}{\ord\delta}$. Recall that $\ord\delta=B$.
\end{proof}

\begin{lema}\label{lem:IAB-fixed}
    In a block-like Clifford setting, let $(\eta_0,\xi)\in I_{A,B}$. Then any irreducible constituent $\theta$ of $\eta_{\eta_0,\xi}\big|_H$ is fixed by $K$.
\end{lema}
\begin{proof}
    We have $\sigma\cdot\eta_{\eta_0,\xi}=\xi^{-1}\eta_{\eta_0,\xi}$, as $\xi^A\cdot\eta_0=\eta_0$. Hence, the argument of the proof of Lemma \ref{lem:fixed-by-sigma-delta} shows that $\theta$ is fixed by $\sigma$. It is also fixed by $\delta$, because $|\Stab[\eta_0]|$ divides $\frac{|\ov{K}_0|}{B}$ (see the last paragraph of the proof of Lemma \ref{lem:fixed-by-Kprime}). Since $K\simeq F_A\times F_B$ is generated by $\pi(\hat\sigma)$ and $\pi(\gamma)$, which act on $\Irr[H]$ as $\sigma$ and $\delta$, the claim follows.    
\end{proof}

\begin{proof}[Proof of Theorem \ref{clasf-char}]
    The previous lemmas imply that $\Theta$ corestricts to a surjective map between $I_{A,B}$ and $\Irr[H]^K/K_0$. Note that if $(\eta_0,\xi)$ and $(\eta'_0,\xi')$ are pairs inducing an orbit $[\theta]$, we must have $\eta_{\eta'_0,\xi'}=\epsilon\eta_{\eta_0,\xi}$ for some $\epsilon\in\Irr[K_0]$. Taking all $g_i=1$ for $i>1$, we find $\eta'_0=\epsilon\eta_0$. Now with $g_2$ free, we get 
    \[\xi'\epsilon\eta_0 = \xi'\eta'_0 = \epsilon\xi\eta_0\]
    and, therefore, $\xi^{-1}\xi'$ fixes $\eta_0$. Conversely, if $\eta'_0=\epsilon\eta_0$ and $\xi'=\xi\lambda$ with $\epsilon\in\Irr[K_0]$ and $\lambda\in\Stab[\eta_0]$, then $\eta_{\eta'_0,\xi'}(g)=\epsilon(\pi_0(g_1)\cdots\pi_0(g_A))\,\eta_{\eta_0,\xi}(g)$, so both pairs give the same restriction to $H$. 
\end{proof}

\subsection{Ramification index}

\begin{lema}\label{stab-description}
    In a block-like Clifford setting, let $\chi$ be an irreducible character of $G'$ with $e(\chi)^2=|\Stab|$. Then $\Stab$ is the subgroup of elements of $\Irr[K']$ of order dividing $e(\chi)$. 
\end{lema}
\begin{proof}
    We know by Lemma \ref{cor:cyclic-stab-implies-not-ramification} that $\Stab$ must be $\Z/e(\chi)\Z\times \Z/e(\chi)\Z$. Now, an element of $F_A\times F_B$ has order dividing $e(\chi)$ if and only if each of its coordinates has. Hence, there are at most $e(\chi)^2$ elements of such order in $F_A\times F_B$ because, in a cyclic group, there are at most $e(\chi)$ elements of order dividing $e(\chi)$. Hence, $\Irr[K']\simeq K'$ has at most $e(\chi)^2$ elements of order dividing $e(\chi)$, and $\Stab$ consists of them.
\end{proof}

\begin{thm}\label{ramification}
    In a block-like Clifford setting, let $(\eta_0,\xi)\in I_{A,B}$. Then, 
    \[e_{G'/H}(\chi_{\eta_0,\xi})=\frac{m}{\gcd\left(\frac{|\ov{K}_0|R}{B|\Stab[\eta_0]|},m\right)}\]
    where $m$ is the order of $\xi$ in $\Irr[K_0]/\Stab[\eta_0]$ and $R=|K/K'|$. 
\end{thm}
\begin{proof}
    Recall from Lemma \ref{lem:IAB-fixed} that $\theta:=\theta_{\eta_0,\xi}$ must be fixed by $K=F_A\times F_B$. We shall deal with the case $K'=K$ first. Set $\chi=\chi_{\eta_0,\xi}$ so that $\chi\big|_H=e(\chi)\theta$. Let us use the subgroup $\hat{H}$ of $G$ which is the preimage of $F_A$. Then $G/\hat{H}\simeq F_B$ and $\hat{H}/H\simeq F_A$ are cyclic. Since in any cyclic extension the ramification indices are always one, $e(\chi)$ will be the size of the $F_B$-orbit of any $\chi'\in\Irr[\hat{H}]$ which is an irreducible constituent of $\chi\big|_{\hat{H}}$. Equivalently, the minimum positive $i$ such that $\gamma^i$ stabilizes $\chi'$. 
    
    Let us use also the group $G_0\wr F_A$ generated by $G_0^A$ and $\sigma$. Recall $(G_0\wr F_A)/\hat{H}=K_0$. There exists an irreducible character $\hat\eta$ of $G_0\wr F_A$ such that $\hat\eta\big|_{G_0^A}$ is the sum of the $\sigma$-orbit of $\eta:=\eta_{\eta_0,\xi}$. The restriction $\hat\eta\big|_{\hat{H}}$ has an irreducible constituent which restricts to $\theta$. Note that the action of $G/\hat{H}$ commutes with the one of $\Irr[\hat{H}/H]$ because $K$ is commutative. Hence, we can assume that $\chi'$ is an irreducible constituent of $\hat\eta\big|_{\hat{H}}$. 
    
    By condition (\ref{block-like:hatH}) of block-like Clifford settings, the action of $\gamma$ on $\Irr[\hat{H}]$ agrees with the one given by $\delta\in\ov{K}_0$, and $\gamma^i$ stabilizes $\chi'$ if and only if $\frac{\ord\delta}{\gcd(i,\ord\delta)}$ divides $\frac{|\Stab[\chi']||\ov{K}_0|}{|K_0|}$. Equivalently, $|\Stab[\hat\eta]|\mid\gcd\left(\frac{i|\ov{K}_0|}{B},|\ov{K}_0|\right)$. Hence,
    \[e(\chi)=\frac{|\Stab[\hat\eta]|}{\gcd\left(\frac{|\ov{K}_0|}{B},|\Stab[\hat\eta]|\right)}\]

    Now, note that each irreducible constituent of $\hat\eta\big|_{\hat{H}}$ is still irreducible when restricted to $H$. Hence, $|\Stab[\hat\eta]|$ is the number of irreducible constituents of $\hat\eta\big|_H$, counted with multiplicity. Note that the action of $\sigma$ and the one of $\Irr[G_0^A/H]$ commute by construction. Therefore,
    \[|\Stab[\hat\eta]|=|\Stab[\eta]|\cdot|\sigma\text{-orbit of }\eta|\]
    Finally, since $\sigma\cdot\eta=\xi^{-1}\eta$ and $\Stab[\eta]=\Stab[\eta_0]$, $|\sigma\text{-orbit of }\eta|$ is $m$, the order of $\xi$ in $\Irr[K_0]/\Stab[\eta_0]$. In conclusion,
    \[e(\chi)=\frac{m}{\gcd\left(\frac{|\ov{K}_0|}{B|\Stab[\eta_0]|},m\right)}\]
    because $|\Stab[\eta_0]|$ divides $\frac{|\ov{K}_0|}{B}$.

    Now, let us deal with the general case. Pick an irreducible character $\hat\chi$ of $G$ such that $\chi$ is an irreducible constituent of $\hat\chi\big|_{G'}$. Note that $G/G'\simeq K/K'\simeq \Z/R\Z$ is cyclic, as $K$ is generated by $K'$ and $F_A$ because $K'\to F_B$ is surjective, and the actions of $\Irr[K']$ and $G/G'$ commute. Hence, 
    \[e_{G/H}(\hat \chi)= e_{G'/H}(\chi)\,|\text{$\Z/R\Z$-orbit of }\chi|=e_{G'/H}(\chi)\,|\Stab[\hat\chi]\cap\Irr[\Z/R\Z]|\]
    Now, $\Irr[\Z/R\Z]$ embeds in $\Irr[K]$ as the characters whose restriction to $K'$ is trivial. As $\theta$ is fixed by $K$, Theorem \ref{clifford} gives $e_{G/H}(\hat\chi)^2=|\Stab[\hat\chi]|$, so, by Lemma \ref{stab-description}, $\Stab[\hat\chi]$ consists of the elements of $\Irr[K]$ of order dividing $e_{G/H}(\hat\chi)$. Hence, $\Stab[\hat\chi]\cap\Irr[\Z/R\Z]$ consists of the elements of the cyclic group $\Irr[\Z/R\Z]$ of order dividing $e_{G/H}(\hat\chi)$, and
    \[|\Stab[\hat\chi]\cap\Irr[\Z/R\Z]|=\gcd(R,e_{G/H}(\hat\chi)).\]
    To finish, note that
    \[\gcd\left(\frac{|\ov{K}_0|}{B|\Stab[\eta_0]|},m\right)\gcd\left(R,\frac{m}{\gcd\left(\frac{|\ov{K}_0|}{B|\Stab[\eta_0]|},m\right)}\right)=\gcd\left(\frac{R|\ov{K}_0|}{B|\Stab[\eta_0]|},m\right)\]
\end{proof}

By Theorem \ref{ramification}, $e_{G'/H}(\chi_{\eta_0,\xi})$ only depends on $\eta_0$ and $\xi$. Accordingly, we write $e_{G'/H}(\eta_0,\xi)$ for this number. By Lemma \ref{stab-description}, the group $\Stab[\chi_{\eta_0,\xi}]\subset \Irr[K']$ only depends on $\eta_0$ and $\xi$ too.

\subsection{Reduced character formula}

We need to understand how to sum $\Omega^{G'/H}_{\theta_{\eta_0,\xi}}(-,-)$ over $\theta$, where the superscript indicates that $\Omega_\theta$ is taken with respect to the Clifford setting $H\to G'\to K'$.

\begin{lema}\label{lem:twist-block-like}
    In a block-like Clifford setting, let $H\subset G'\subset G$ be an intermediate extension and let $(\eta_0,\xi)\in I_{A,B}$. Let $\epsilon\in\Hom{H}{\C^\times}$ be the restriction to $H$ of the character $g\mapsto\prod_{i=1}^A\xi^{i-1}(\pi_0(g_i))$ of $G_0^A$. Then
    \[\Omega^{G'/H}_{\epsilon^j\theta}(\ov a,\ov b)=\Omega^{G'/H}_{\epsilon\theta}(\ov a,\ov b)^j\]
    for every irreducible constituent $\theta$ of $\eta_{\eta_0,1}\big|_H$, every integer $j$ and all tuples $\ov a,\ov b\in K'^g$.
\end{lema}
\begin{proof}
    Let $\eta=\eta_{\eta_0,1}$ and let $\theta$ be an irreducible constituent of $\eta\big|_H$. We will see below that the characters $\epsilon^j\theta$ are fixed by $K$, so, by Lemma \ref{lem:Omega-towers}, we may assume that $G'=G$. We are going to use Corollary \ref{coro:twist-character}, for the Clifford setting $H\to G\to K$ and $\Xi=\langle\epsilon\rangle$. The claim is that $\zeta\mapsto\Omega^{G/H}_{\zeta\theta}(\ov a,\ov b)$ is a character of $\langle\epsilon\rangle$. Hence, we have to show that $\theta$ is fixed by $K$ and unramified in $G$, that $\epsilon\in\Xi_\theta$, and that intertwiners for elements of $\Lambda_{\langle\epsilon\rangle}$ commute with each other, where $\Lambda_{\langle\epsilon\rangle}\subset\Lambda_\theta$ is the subgroup generated by the $\partial_k(\epsilon^i)=(\partial_k\epsilon)^i$ for $k\in K$ and $i\in\Z$.

    As $(\eta_0,1)\in I_{A,B}$, $\theta$ is fixed by $K$ by Lemma \ref{lem:IAB-fixed} and, since the integer $m$ of Theorem \ref{ramification} is $1$ for $\xi=1$, that theorem shows that $\theta$ is unramified in $G$. The character $\epsilon$ relates the characters associated with $(\eta_0,1)$ and $(\eta_0,\xi^j)$: we have $\eta_{\eta_0,\xi^j}\big|_H=\epsilon^j\,\eta\big|_H$, so $\epsilon^j\theta$ is an irreducible constituent of $\eta_{\eta_0,\xi^j}\big|_H$. Moreover, $(\eta_0,\xi^j)\in I_{A,B}$, as $(\xi^j)^A=(\xi^A)^j\in\Stab[\eta_0]$ and the other condition defining $I_{A,B}$ only involves $\eta_0$, so $\epsilon^j\theta$ is fixed by $K$, again by Lemma \ref{lem:IAB-fixed}. Hence, $\epsilon\in\Xi_\theta$.

    For $1\le t\le A$, let $\lambda_t$ be the restriction to $H$ of $\xi^A\circ\pi_0\circ\mathrm{pr}_t$, where $\mathrm{pr}_t:G_0^A\to G_0$ is the $t$-th projection. We are going to show that $\Lambda_{\langle\epsilon\rangle}$ is generated by $\lambda_1,\ldots,\lambda_A$. For this, let us compute $\partial_k\epsilon$ for $k\in K$. The group $G_0\wr F_A$ normalizes $H$, and for $u\in G_0\wr F_A$ and $\zeta\in\Hom{H}{\C^\times}$ we also write $(u\cdot\zeta)(h)=\zeta(uhu^{-1})$ and $\partial_u\zeta=(u\cdot\zeta)\zeta^{-1}$. By condition (\ref{block-like:actions}) of block-like Clifford settings, if $\sigma^j$ is the image of $k$ in $F_A$, then $k$ acts on $\Irr[H]$ as the conjugation by $\sigma^ju$ for some $u\in G_0^A$. Since $\epsilon$ is the restriction of a character of $G_0^A$, it is fixed by the conjugation by $G_0^A$, so $\partial_k\epsilon=\partial_{\sigma^j}\epsilon$. Recall that $(\sigma\cdot g)_i=g_{i+1}$. For $g\in H$,
    \[(\sigma\cdot\epsilon)(g)=\xi^{A-1}(\pi_0(g_1))\prod_{i=2}^{A}\xi^{i-2}(\pi_0(g_i))=\xi^A(\pi_0(g_1))\,\xi^{-1}\big(\pi_0(g_1)\cdots\pi_0(g_A)\big)\,\epsilon(g)=\xi^A(\pi_0(g_1))\,\epsilon(g).\]
    Hence, $\partial_\sigma\epsilon=\lambda_1$. Similarly, $\partial_{\sigma^j}\epsilon=\lambda_1\cdots\lambda_{j'}$, where $j'$ is the residue of $j$ modulo $A$. Hence, $\Lambda_{\langle\epsilon\rangle}$ is generated by the products $\lambda_1\cdots\lambda_j$ for $0\le j<A$, that is, by $\lambda_1,\ldots,\lambda_A$, as $\lambda_1\cdots\lambda_A$ is trivial on $H$. In particular, the $\lambda_t$ lie in $\Lambda_\theta$, which contains $\Lambda_{\langle\epsilon\rangle}$ by Lemma \ref{lem:twist-basics}(1).

    We choose the intertwiners $T_{\lambda_t}$ as follows. Let $\rho_0:G_0\to\GL(V_0)$ afford $\eta_0$. As $\xi^A\in\Stab[\eta_0]$, there is $T_0\in\GL(V_0)$ with $T_0\rho_0(g)T_0^{-1}=\xi^A(\pi_0(g))\rho_0(g)$ for $g\in G_0$. The character $\eta$ is afforded by $V_0^{\otimes A}$, with $G_0^A$ acting factorwise, and $T^{(t)}=\operatorname{Id}\otimes\cdots\otimes T_0\otimes\cdots\otimes\operatorname{Id}$, with $T_0$ in the $t$-th factor, conjugates the action of $g$ into $\xi^A(\pi_0(g_t))$ times it. Since $G_0^A/H\simeq K_0$ is cyclic, $\eta\big|_H$ is a sum of distinct irreducible characters by Theorem \ref{clifford}. Let $V_\theta\subset V_0^{\otimes A}$ be the unique subrepresentation of $H$ affording $\theta$. Then $T^{(t)}V_\theta$ is a subrepresentation of $H$ affording $\lambda_t^{-1}\theta=\theta$, so $T^{(t)}V_\theta=V_\theta$, and the restriction $T_{\lambda_t}$ of $T^{(t)}$ to $V_\theta$ is an intertwiner for $\lambda_t$. These restrictions commute with each other. As every intertwiner for an element of $\Lambda_{\langle\epsilon\rangle}$ is a multiple of a product of the $T_{\lambda_t}$ and their inverses, these intertwiners commute with each other.
\end{proof}

Recall the identification of $K$ with $\Z/A\Z\times \Z/B\Z$. For $\ov a,\ov b\in K^g$, define 
\[\Delta(\ov a,\ov b)=\frac{\gcd(A,B)}{|\langle \sum a^{(1)}_ib^{(2)}_i-a_i^{(2)}b_i^{(1)} \rangle |}\] 
where the superscript indicates the coordinates in $\Z/A\Z\times \Z/B\Z$ and the computation is done in $\Z/\gcd(A,B)\Z$. This number does not depend on the choice of the isomorphism. Define also 
\[e_A(\eta_0):=\frac{\gcd(A|\Stab[\eta_0]|,|K_0|)}{\gcd(A|\Stab[\eta_0]|,\frac{|\ov{K}_0|}{B})}\]
for $\eta_0\in \Irr[G_0]$.

\begin{lema}\label{sum-omega}
    In a block-like Clifford setting, let $\eta_0\in\Irr[G_0]$ be such that $|\Stab[\eta_0]|$ divides $\frac{|\ov{K}_0|}{B}$, and let $z\in H$ be an element such that each of its coordinates has a $\gcd(A|\Stab[\eta_0]|,|K_0|)$-th root in $G_0$. Then, for any $\ov a,\ov b\in K'^g$,
    \[\sum_{[\xi]} \sum_\theta \Omega_\theta^{G'/H}(\ov a,\ov b)\theta(z) = \left\{\begin{array}{cl}
        \gcd\pare{
        A,\frac{|K_0|}{|\Stab[\eta_0]|}
        }
        \prod\limits_{i=1}^A\eta_0(z_i) 
     &  \text{if }e_A(\eta_0)|\Delta(\ov a,\ov b)\\
        0 & \text{if not}
    \end{array}\right.\]
    where the first sum runs over all $[\xi]\in \Irr[K_0] /\Stab[\eta_0]$ of order dividing $A$ and the second one runs over all irreducible constituents of $\eta_{\eta_0,\xi}\big|_H$.
\end{lema}
\begin{proof}
    Recall from Lemma \ref{lem:Omega-towers} that $\Omega$ is compatible with towers, so we can assume that $G'=G$. The first sum runs over a cyclic group of size $m:=\gcd\left(A,\frac{|K_0|}{|\Stab[\eta_0]|}\right)$. Let $[\xi]$ be a generator and let $\epsilon\in\Hom{H}{\C^\times}$ be the restriction to $H$ of $g\mapsto\prod_{i=1}^A\xi^{i-1}(\pi_0(g_i))$, so that $\eta_{\eta_0,\xi^j}\big|_H=\epsilon^j\,\eta_{\eta_0,1}\big|_H$ for every $j$. Note that $\epsilon(z)=1$ as $\ord\xi\mid m|\Stab[\eta_0]| = \gcd(A|\Stab[\eta_0]|,|K_0|)$ and each coordinate of $z$ has a $\gcd(A|\Stab[\eta_0]|,|K_0|)$-th root. Hence, $\sum_\theta\theta(z)=\prod_{i=1}^A \eta_0(z_i)$.

    Moreover, the irreducible constituents of $\eta_{\eta_0,\xi^j}\big|_H$ are the characters $\epsilon^j\theta$, where $\theta$ runs over the irreducible constituents of $\eta_{\eta_0,1}\big|_H$. Since $(\eta_0,\xi)\in I_{A,B}$, as $|\Stab[\eta_0]|$ divides $\frac{|\ov{K}_0|}{B}$ and $\xi^A\in\Stab[\eta_0]$, Lemma \ref{lem:twist-block-like} gives $\Omega^{G/H}_{\epsilon^j\theta}=\big(\Omega^{G/H}_{\epsilon\theta}\big)^j$ for any such $\theta$. Let us sum over $\epsilon\theta,\ldots,\epsilon^{m}\theta$ for a fixed $\theta= \theta_{\eta_0,1}$. We have
    \[\sum_{i=1}^m \Omega_{\epsilon\theta}^{G/H}(\ov a,\ov b)^i\epsilon^i(z)\theta(z)=\theta(z)\sum_{i=1}^m \Omega_{\epsilon\theta}^{G/H}(\ov a,\ov b)^i\]

     Let $\omega=\Omega^{G/H}_{\epsilon\theta}((1,0),(0,1))$. Recall from Lemma \ref{lem:Omega-two-generators} that it has order $e_{G/H}(\eta_0,\xi)$, and note that this order divides $\ord{[\xi]}$ as well as $\gcd(A,B)$. Then,  
    \begin{align*}
        \sum_{i=1}^m \Omega_{\epsilon\theta}^{G/H}(\ov a,\ov b)^i&=\sum_{i=1}^{m} \omega^{i\Delta'(\ov a,\ov b)},\qquad \Delta'(\ov a,\ov b)=\sum_{j=1}^g a^{(1)}_jb^{(2)}_j-a^{(2)}_jb^{(1)}_j,
    \end{align*}
    which is zero unless $\omega^{\Delta'(\ov a,\ov b)}=1$, in which case it is $m$. As $\ord\omega$ divides $\gcd(A,B)$, this happens if and only if $\ord\omega$ divides $\Delta(\ov a,\ov b)$. This condition does not depend on the choice of $\theta$. Thus the summation of the statement is
    \begin{align*}
        m\sum_\theta \theta(z)
        =\gcd\left(A,\frac{|K_0|}{|\Stab[\eta_0]|}\right)\prod\limits_{i=1}^A\eta_0(z_i)
    \end{align*}
    provided that $e_{G/H}(\eta_0,\xi)|\Delta(\ov a,\ov b)$. 
    
    Finally, recall from Theorem \ref{ramification} with $R=1$ that $e_{G/H}(\eta_0,\xi)$ is 
    \begin{align*}
         \frac{\gcd\left(A,\frac{|K_0|}{|\Stab[\eta_0]|}\right)}{\gcd\left(\frac{|\ov{K}_0|}{B |\Stab[\eta_0]|},\gcd\left(A,\frac{|K_0|}{|\Stab[\eta_0]|}\right)\right)}
        = \frac{\gcd\left(A|\Stab[\eta_0]|,|K_0|\right)}{\gcd\left(\frac{|\ov{K}_0|}{B },\gcd\left(A|\Stab[\eta_0]|,|K_0|\right)\right)}
        =  \frac{\gcd\left(A|\Stab[\eta_0]|,|K_0|\right)}{\gcd\left(A|\Stab[\eta_0]|,\frac{|\ov{K}_0|}{B}\right)}
    \end{align*}
    So it agrees with $e_A(\eta_0)$.
\end{proof}

Consistently with the previous lemma, denote by $\Irr[G_0]_{\ov a,\ov b}\subset \Irr[G_0]$ the subset of those characters $\eta$ such that $|\Stab[\eta]|$ divides $\frac{|\ov{K}_0|}{B}$ and $e_A(\eta) | \Delta(\ov a,\ov b)$. In addition, for $z\in G_0^A$ and $\eta\in \Irr[G_0]$, write 
\[\eta(z)=\prod_{i=1}^A\eta(z_i)\]
where $z_i$ is the $i$-th coordinate of $z$.

\begin{thm}\label{reduced-character-formula}
    In a block-like Clifford setting, pick an integer $N$ such that $|\Stab[\eta]|$ divides $N$ for every $\eta\in\Irr[G_0]$. Let $z\in H$ be an element each of whose coordinates has a $\gcd(AN,|K_0|)$-th root in $G_0$. For any positive integer $g$ and tuples $\ov a,\ov b\in K^g$ whose coordinates generate $K'$, the following identity holds 
    \[\left|\left\{\begin{array}{cc}
         (\ov x,\ov y)\in G^g\times G^g:\\
         {[}\ov x,\ov y]z=1,\pi(\ov x)=\ov a,\pi(\ov y)=\ov b
    \end{array}\right\}\right|=\sum_{\eta}\gcd\left(A,\frac{|K_0|}{|\Stab[\eta]|}\right)\left(\frac{|\Stab[\eta]|}{|K_0|} \right)^{2g}\left(\frac{|G_0|}{\eta(1)}\right)^{A(2g-1)} \eta(z)\]
    where $\eta$ runs over all characters in $\Irr[G_0]_{\ov a,\ov b}$.
\end{thm}
\begin{proof}
    We apply Lemma \ref{sum-omega} and Theorem \ref{clasf-char} to the formula of Theorem \ref{character-formula} for $H\to G'\to K'$:     \begin{align*}
        \sum_{\theta\in\operatorname{Irr}(H)^{K'}}\left(\frac{|H|}{\theta(1)}\right)^{2g-1}\Omega_\theta^{G'/H}(\ov a,\ov b)\theta(z) &=\sum_{(\eta,\xi)\in I_{A,B}/\sim}\left(\frac{|G_0|^A|\Stab[\eta]|}{|K_0|\eta(1)^A}\right)^{2g-1}\sum_{\theta\in \Theta(\eta,\xi)}\Omega^{G'/H}_{\theta}(\ov a,\ov b)\theta(z)\\
        &=\sum_{\eta}\gcd\left(A,\frac{|K_0|}{|\Stab[\eta]|}\right)\left(\frac{|\Stab[\eta]|}{|K_0|} \right)^{2g}\left(\frac{|G_0|}{\eta(1)}\right)^{A(2g-1)} \eta(z)
    \end{align*}
    where in the last sum $\eta$ runs over $\Irr[G_0]_{\ov a,\ov b}$.
\end{proof}

\begin{rmk}
    We believe that this result should be useful to compute the stringy $E$-polynomials of other families of character varieties. For example, for $D_n$, one may want to consider $A=2$ or $4$, $H_0=\operatorname{Spin}_{\frac{2n}{A}}(\F_q)$ and $G_0$ a double cover of
    \[\{M\in \GL_{\frac{2n}{A}}(\F_q):MM^T=\omega^2\operatorname{Id}, \omega\in \F_q^\times\}.\]
\end{rmk}

\section{Stringy points of type A parabolic character varieties}\label{sec:stringy-points}

\subsection{Fixed point sets}

From now on, fix positive integers $n$, $d$, and $g$ with $d|n$. Recall the notation of Section \ref{prelim:ch-var}. In this subsection, we find a decomposition of $\abbreviatedBetti^a$ for $a\in F_d^{2g}$. From now on, we assume that $\mcC$ is semisimple, regular, and such that no non-empty proper subset of its eigenvalues has product one. We summarize this by saying that $\mcC$ is nice.

\begin{Def}
    Fix an algebraically closed field $k$ whose characteristic does not divide $d$. Given $a\in F_d^{2g}$, define
    \[X_{a}=
    \left\{\begin{array}{c}
     (x,y,z)\in \SL_n(k)^g\times \SL_n(k)^g\times\mcC:[x,y]z=1,\\ a\cdot (x,y)\in\PGL_n(k)\cdot (x,y)
    \end{array}\right\}\]
    and call $X_{a,z}$ the fiber of $X_{a}$ over $z\in \mcC$.
\end{Def}

By its very definition, $\betti[\SL_n][k]^{a}$ is nothing but $X_a/\PGL_n(k)$ or $X_{a,z}/T(k)$ for chosen $z\in \mcC$.  Fix $a\in F_d^{2g}$ and $z\in \mcC$. Fix a basis where $z$ is diagonal. Denote by $A$ the order of $a$. It agrees with the size of the subgroup of $F_d$ generated by its coordinates. Identify it with $F_A$.

Put 
\[H_A=\{(x_1,\ldots,x_A)\in \GL_{\frac{n}{A}}(k)^A:\det(x_1)\cdots \det(x_A)=1\},\] and $\hat{H}_{A}$ the subgroup of $\SL_n(k)$ generated by $H_A$ and 
the matrix
\[\hat\sigma=\left(\begin{array}{ccccc}
0 & \operatorname{Id}_{\frac{n}{A}} & 0 & \cdots & 0 \\
0 & 0 & \operatorname{Id}_{\frac{n}{A}} & \cdots & 0 \\
\vdots & \vdots & \vdots & \ddots & \vdots \\
0 & 0 & 0 & \cdots & \operatorname{Id}_{\frac{n}{A}}\\
(-1)^{A-1}\operatorname{Id}_{\frac{n}{A}} & 0 & 0 & \cdots & 0
\end{array}\right).\]
Call $\pi:\hat{H}_A\to F_A$ the quotient by $H_A$. These are the groups $H$ and $\hat{H}$ of the example in Section \ref{sec:The case}, over $k$.

Let $\hat{\Upsilon}_{A}$ be the set whose elements are maps $V:F_A\to \operatorname{Grass}_{k}(n,\frac{n}{A})$ such that $k^n=\bigoplus_{i\in F_A} V_i$ and $z$ fixes this decomposition. It has an action of $F_A$ by translations. Define $\Upsilon_A=\hat{\Upsilon}_A/F_A$. Choosing an order of the eigenvalues of $z$ induces an isomorphism
\[\Upsilon_A \simeq S_n/\langle (S_{\frac{n}{A}})^A, \operatorname{shift}\rangle\]
where $\operatorname{shift}$ is the permutation $j\mapsto j+\frac{n}{A}\pmod n$. 

\begin{lema}\label{irred-comp}
    Let $\mcC\subset \SL_n(k)$ be a nice conjugacy class, $z\in \mcC$ and $a\in F_d^{2g}$. Then
    \[X_{a,z}\simeq\bigsqcup_{[\upsilon]\in \Upsilon_A} \left\{\begin{array}{c}
        (x,y)\in \hat{H}_{A}^g\times \hat{H}_{A}^g: 
       {[}x,y]z^\upsilon = 1, \pi(x,y) = a\end{array}
        \right\}\]
    where $z^\upsilon$ is the tuple of $H_A$ whose $i$-th coordinate is $z\big|_{\upsilon(i)}$. This decomposition is $F_d^{2g}$- and $T(k)$-equivariant.
\end{lema}
\begin{proof}
    Pick $(x,y)\in X_{a,z}$ and $w\in \GL_n(k)$ such that $w(x,y)w^{-1}=a\cdot (x,y)$. Rewrite the conjugacy condition as $w(x,y)=a\cdot (x,y)w$. Let $E_\lambda$ be an eigenspace of $w$ for some eigenvalue $\lambda$. Then each coordinate of $(x,y)$ maps $E_\lambda$ into $E_{a_i\lambda}$, where $a_i$ is the corresponding coordinate of $a$. It follows that, if we identify $F_A$ with the group generated by the coordinates of $a$, $V=\bigoplus_{i\in F_A} E_{i\cdot \lambda}$ is an invariant subspace for $(x,y)$. Note that $z$ cannot be a commutator in any proper subspace of $k^n$. Hence, the condition on $z$ implies that $V=k^n$. We show that there is an $F_A$-grading $k^n=\bigoplus_{i\in F_A} V_i$ such that $(x,y)$ have degree $a$. Note that $z$ should have degree zero. On the other hand, if such a grading is given and there is a tuple of degree $a$ of linear operators $(x,y)$ with $[x,y]z=1$, we can define $w$ as the action by $i$ in $V_i$. 
    
    Note that the grading is determined by $(x,y)$ up to translation. Indeed, if $W_i$ is another grading, as $z$ is regular semisimple and has degree zero in both gradings, every $W_i$ and every $V_j$ is a sum of eigenlines of $z$, so there must be a nonzero intersection $W_i\cap V_j$. But then $\bigoplus_{c\in F_A} (W_{ci}\cap V_{cj})$ is invariant and has to be $k^n$. In addition, all the $V_i$'s have the same dimension because $\langle a_1,\ldots,a_{2g}\rangle=F_A$. As $z$ should be of degree zero, the possible gradings are in bijection with $\Upsilon_A$. 

    For the final claim, note that the action of $F_d^{2g}$ in $X_{a,z}$ agrees with the induced action of $F_d^{2g}$ by multiplying in $\hat{H}_A$. The same happens for the action of $T$.
\end{proof}

Hence, we obtain an $F_d^{2g}$-equivariant decomposition
\[\betti[\SL_n][k]^a=\bigsqcup_{[\upsilon]\in \Upsilon_A} \betti[\SL_n][k]^a_\upsilon\]
Here $\betti[\SL_n][k]^a_\upsilon$ is the image of the pairs $(x,y)\in X_{a,z}$ that are homogeneous of degree $a$ for a grading in the class $[\upsilon]$. We will prove later that, if $\gcd(A,\frac{n}{d})=1$, the quotients $\abbreviatedBetti^a_\upsilon/F_d^{2g}$ are the irreducible components of $\abbreviatedBetti^a/F_d^{2g}$. Now we only prove the following lemma for further use.

\begin{lema}\label{lem:dim-sectors}
    Let $k$ be an algebraically closed field whose characteristic does not divide $n$ and $\mcC\subset \SL_n(k)$ be a nice conjugacy class. Then $\betti[\SL_n][k]^a_\upsilon$ is smooth and equidimensional of dimension $(2g-1)(\frac{n^2}{A}-1)-n+1$, for any $a\in F_d^{2g}$ and $[\upsilon]\in \Upsilon_A$.
\end{lema}
\begin{proof}
    The action of $\PGL_n(k)$ is free by Lemma \ref{lem:free-smooth}, so it is enough to show that each piece of the decomposition of Lemma \ref{irred-comp} is smooth and equidimensional of dimension $(2g-1)(\frac{n^2}{A}-1)$. Such a piece is the fiber over $1$ of
    \[\mu:\pi^{-1}(a)\to H_A,\qquad (x,y)\mapsto[x,y]z^\upsilon,\]
    where $\pi^{-1}(a)\subset\hat H_A^{2g}$ is a coset of $H_A^{2g}$. Hence it suffices to show that $\mu$ is submersive at every point of this fiber. This is done as \cite[Theorem 2.2.5]{HV}.
\end{proof}

\subsection{Counting stringy points}\label{counting-twisted-points}

We want to compute the number of (twisted) points of a twisted sector of $\betti[(\SL_n/F_d)][\ov{\F}_q]$. Namely, $|(\betti[\SL_n][\ov{\F}_q]^{a})^{b^{-1}\frob}|$ where $a,b\in F_d^{2g}$. To do so, we need a better description in order to apply the results of the previous sections. First, we introduce a twisted version of $X_{a}$.

\begin{Def}
    Fix a finite field $\F_q$ with $d|q-1$. Given $a,b\in F_d^{2g}$, define
    \[X_{a}^b=
    \left\{\begin{array}{c}
     (x,y,z)\in \SL_n(\ov{\F}_q)^g\times \SL_n(\ov{\F}_q)^g\times\mcC:[x,y]z=1,\\ \frob(x,y)=b\cdot (x,y),  a\cdot (x,y)\in\PGL_n(\F_q)\cdot (x,y)
    \end{array}\right\}\]
    which is a subset of $\preBetti[\ov{\F}_q]$. Call $X_{a,z}^b$ the fiber of $X_{a}^b$ over $z\in \mcC$.
\end{Def}

\begin{lema} 
    Let $\mcC\subset \SL_n(\ov{\F}_q)$ be a nice conjugacy class whose eigenvalues lie in $\F_q$. Then for $d|q-1$ and $a,b\in F_d^{2g}$,
    \[({\betti[\SL_n][\ov{\F}_q]}^a)^{b^{-1}\frob}=X_{a}^b/\PGL_n(\F_q)\]
    and, in particular, 
    \[|({\betti[\SL_n][\ov{\F}_q]}^a)^{b^{-1}\frob}|=\frac{|X_{a}^b|}{|\PGL_n(\F_q)|}=\frac{|X_{a,z}^b|}{|T(\F_q)|}\]
    for any $z\in \mcC\cap \SL_n(\F_q)$.
\end{lema}
\begin{proof}
    Note that any element of $X_{a}^b$ is fixed by $b^{-1}\frob$. In addition, the action of $\PGL_n(\F_q)$ preserves $X_{a}^b$. It follows that we have a map $\varphi:X_{a}^b/\PGL_n(\F_q)\to (\betti[\SL_n][\ov{\F}_q]^{a})^{b^{-1}\frob}$. We claim that $\varphi$ is injective. Let $(x_1,y_1),(x_2,y_2)\in X_{a}^b$ with $\varphi(x_1,y_1)=\varphi(x_2,y_2)$. This means that there exists some $w\in \PGL_n(\ov{\F}_q)$ such that $w(x_1,y_1)w^{-1}=(x_2,y_2)$. But then 
    \[\frob(w)(x_1,y_1)\frob(w)^{-1}=(x_2,y_2) \]
    as $\frob$ acts on $(x_1,y_1),(x_2,y_2)$ as $b$. Hence, $\frob(w)^{-1}w$ fixes $(x_1,y_1)$. But the action is free by Lemma \ref{lem:free-smooth}, so $w\in \PGL_n(\F_q)$. Note that the same argument allows us to only require $a\cdot (x,y)\in\PGL_n(\ov{\F}_q)\cdot (x,y)$ in the definition of $X_{a}^b$.
    
    Surjectivity is a consequence of Lang's theorem. Indeed, let us consider $X_a$ for $k=\ov{\F}_q$. Recall that $\betti[\SL_n][\ov{\F}_q]^{a}$ is nothing but $X_a/\PGL_n(\ov{\F}_q)$. Pick $(x,y)\in (\betti[\SL_n][\ov{\F}_q]^{a})^{b^{-1}\frob}$. In this case, the fiber of $X_a\to\betti[\SL_n][\ov{\F}_q]^{a}$ over $(x,y)$ is $b^{-1}\frob$-invariant. We are in the following situation: $b^{-1}\frob$ is an endomorphism of a quotient of the connected group $\PGL_n(\ov{\F}_q)$, compatible with the action in the sense that $b^{-1}\frob(w\cdot p)=\frob(w)\cdot b^{-1}\frob(p)$. Hence, we can apply Lang's theorem to conclude that $X_a^{b^{-1}\frob}\to (\betti[\SL_n][\ov{\F}_q]^{a})^{b^{-1}\frob}$ is surjective.
    
    For the last statement, recall that the action is free and use again Lang's theorem.
\end{proof}

The next step is to rewrite $X_{a,z}^b$ suitably to apply Theorem \ref{reduced-character-formula}. Fix $a,b\in F_d^{2g}$ and $z\in \mcC\cap \SL_n(\F_q)$. Fix a basis where $z$ is diagonal. Denote by $A$ and $B$ the orders of $a$ and $b$ respectively. These numbers agree with the sizes of the subgroups of $F_d$ generated by their coordinates. Identify $F_A$ and $F_B$ with those groups. Pick an isomorphism $\beta:F_B\simeq \Z/B\Z$ and let $\omega$ be the root of unity such that $b_i=\omega^{\beta(b_i)}$. Note that $\omega$ has order $B$.

Assume that $n|q-1$. Let $H_A$ be as in the previous subsection, now with $\F_q$ in place of $k$, and let $G_{A,B}$ be the subgroup of $\SL_n(\ov{\F}_q)$ generated by $H_A$, the matrix $\hat\sigma$ of the previous subsection, and the diagonal matrix $\gamma=\operatorname{diag}(\gamma_0,\ldots,\gamma_0,\gamma_0^{1-n})$, where $\gamma_0\in \ov{\F}_q^\times$ satisfies $\frob(\gamma_0)=\omega^{-1}\gamma_0$. Since $\frob^j(\gamma_0)=\omega^{-j}\gamma_0$ and $B|n$, we have $\gamma_0\in\F_{q^n}$, $\gamma_0^n\in\F_q$, and $B$ is the least positive integer with $\gamma_0\in\F_{q^B}$, as required in Section \ref{sec:The case}. These are the groups $H$ and $G$ of the example in Section \ref{sec:The case}; in particular, $H_A$ is normal in $G_{A,B}$ and their quotient is $F_A \times F_B$. Call $\pi:G_{A,B}\to F_A\times F_B$ the quotient morphism and $\pi_A:G_{A,B}\to F_A$ and $\pi_B:G_{A,B}\to F_B$ the induced maps, where $\pi_B(\gamma)$ is identified with $\omega^{-1}\in F_B$.

\begin{lema}
    Let $\mcC\subset \SL_n(\ov{\F}_q)$ be a nice conjugacy class whose eigenvalues lie in $\F_q$, and $z\in \mcC\cap \SL_n(\F_q)$. Then
    \[X_{a,z}^b\simeq\bigsqcup_{[\upsilon]\in \Upsilon_A} \left\{\begin{array}{c}
        (x,y)\in G_{A,B}^g\times G_{A,B}^g: 
       {[}x,y]z^\upsilon = 1, \\ \pi_A(x,y) = a, \pi_B(x,y)=b\end{array}
        \right\}\]
    where $z^\upsilon$ is the tuple of $H_A$ whose $i$-th coordinate is $z\big|_{\upsilon(i)}$.
\end{lema}
\begin{proof}
    The conjugacy condition on $X_{a,z}^b$ is covered by Lemma \ref{irred-comp} for $k=\ov{\F}_q$. Hence, fix a grading for which $z$ has degree zero and $(x,y)$ acts with degree $a$. We need to explore the condition which involves the Frobenius. But note that $(x,y)\gamma^{\beta(b)}$ is fixed by Frobenius. Hence, $x,y\in G_{A,B}^g$ as $\gamma$ preserves all the possible gradings and has determinant one.
\end{proof}

By the previous lemma, $X_{a,z}^b$ is a disjoint union of sets as in Theorem \ref{reduced-character-formula}, for the block-like Clifford setting $H_A\to G_{A,B}\xrightarrow{\pi}F_A\times F_B$ and with $z^\upsilon$ in place of $z$. Note that the tuples $\ov a,\ov b\in K^g$ of Theorem \ref{reduced-character-formula} are here $\pi(x),\pi(y)\in(F_A\times F_B)^g$, and not our $a,b\in F_d^{2g}$. Ordering the coordinates of $F_d^{2g}$ so that the $(2i-1)$-th and the $2i$-th ones correspond to $x_i$ and $y_i$, respectively, we have
\[\pi(x_i)=(a_{2i-1},b_{2i-1})\qquad\text{and}\qquad \pi(y_i)=(a_{2i},b_{2i}).\]
From now on, we write $\Delta(a,b)$ and $\Irr[\GL_{\frac{n}{A}}(\F_q)]_{a,b}$ for $\Delta(\pi(x),\pi(y))$ and $\Irr[\GL_{\frac{n}{A}}(\F_q)]_{\pi(x),\pi(y)}$. Explicitly, after identifying $F_d$ with $\Z/d\Z$,
\[\Delta(a,b) = \frac{\gcd(A,B)}{|\langle \frac{AB}{d^2}\sum_{i=1}^g(a_{2i-1}b_{2i}-a_{2i}b_{2i-1}) \rangle |}\]
where the computation is done in $\Z/\gcd(A,B)\Z$.

\begin{prop}\label{twisted-points-formula}
    Let $\mcC\subset \SL_n(\ov{\F}_q)$ be a nice conjugacy class and $z\in \mcC\cap \SL_n(\F_q)$. Pick $a,b\in F_d^{2g}$. Assume that each eigenvalue of $\mcC$ has an $n$-th root in $\F_q$ and that $n|q-1$. Then
    \[|({\betti[\SL_n][\ov{\F}_q]}^a)^{b^{-1}\frob}|=\sum_{[\upsilon]\in \Upsilon_A}\sum_{\eta}\frac{A}{(q-1)^{n-1}}\left(\frac{|\Stab[\eta]|}{(q-1)}\right)^{2g}\left(\frac{|\GL_{\frac{n}{A}}(\F_q)|}{\eta(1)}\right)^{A(2g-1)}\eta(z^\upsilon)\]
    where $\eta$ runs over $\Irr[\GL_{\frac{n}{A}}(\F_q)]_{a,b}$, and $A$ and $B$ are the orders of $a$ and $b$ respectively. 
\end{prop}
\begin{proof}
    Apply Theorem \ref{reduced-character-formula}, with $K'$ the subgroup of $F_A\times F_B$ generated by the coordinates of $\pi(x)$ and $\pi(y)$, and the previous lemmas. Here we are using that a character of $\GL_\frac{n}{A}(\F_q)$ has stabilizer under the action of $\Irr[\F_q^\times]$ of size dividing $\frac{n}{A}$. 
\end{proof}

We say that a conjugacy class $\mcC\subset \SL_n(\ov{\F}_q)$ is super nice if it is nice, $n|q-1$, and all eigenvalues of $\mcC$ have an $n$-th root in $\F_q$. In the rest of this section, we will prove the following theorem. In particular, there are explicit formulas for the constants $C_{s,A,\tau}$ appearing in the result.

\begin{thm}\label{number-twisted-points}
    Let $\mcC\subset \SL_n(\ov{\F}_q)$ be a super nice conjugacy class. Pick $a,b\in F_d^{2g}$. Then
    \[|({\betti[\SL_n][\ov{\F}_q]}^a)^{b^{-1}\frob}|= \frac{A}{(q-1)^{n-1}}\sum_\tau \left(\frac{(-1)^nq^{\frac{n^2}{2A}}\mathcal{H}_{\tau'}(q)^{A}}{(q-1)}\right)^{2g-1}\sum_{s} s^{2g}C_{s,A,\tau}\]
    where $\tau$ runs over all multi-partition types whose size is $\frac{n}{A}$, $s$ over all divisors of $\frac{n}{B}$ which satisfy $A|\Delta(a,b)\gcd(A,\frac{n}{Bs})$, and $C_{s,A,\tau}$ are combinatorial constants which do not depend on $q$.  
\end{thm}

Fix $d|n$ and $a,b\in F_d^{2g}$. We are going to apply Green's description of $\Irr[\GL_{\frac{n}{A}}(\F_q)]$ to evaluate the formula of Proposition \ref{twisted-points-formula} and prove Theorem \ref{number-twisted-points}. Define $S_{a,b}$ as the right-hand side of the formula in the previous proposition.
And, for $s|n$ and $\tau$ a multi-partition type, set
\[C_{s,A,\tau}:=\frac{1}{q-1}\sum_{[\upsilon]\in\Upsilon_A}\sum_{\eta} \eta(z^\upsilon)\]
where $\eta$ runs over all irreducible characters of $\GL_{\frac{n}{A}}(\F_q)$ with type $\tau$ and $|\Stab[\eta]|=s$. We recover the original sum by
\begin{align*}
    S_{a,b}&=\frac{A}{(q-1)^{n-1}}\sum_\tau \left(\frac{|\GL_{\frac{n}{A}}(\F_q)|^{A}}{(q-1)\tau(1)^{A}}\right)^{2g-1}\sum_{s} s^{2g}C_{s,A,\tau}\\
    &=\frac{A}{(q-1)^{n-1}}\sum_\tau \left(\frac{(-1)^nq^{\frac{n^2}{2A}}\mathcal{H}_{\tau'}(q)^{A}}{(q-1)}\right)^{2g-1}\sum_{s} s^{2g}C_{s,A,\tau}
\end{align*}
where the first sum runs over all possible types $\tau$ and the second over all divisors $s$ of $\frac{n}{B}$ such that $A|\Delta(a,b)\gcd(A,\frac{n}{Bs})$.

Furthermore, put 
\[\check{C}_{s,A,\tau} = \frac{1}{q-1}\sum_{[\upsilon]\in\Upsilon_A}\sum_{\eta}\eta(z^\upsilon)\]
where $\eta$ runs over all irreducible characters of $\GL_{\frac{n}{A}}(\F_q)$ with type $\tau$ and $|\Stab[\eta]|$ multiple of $s$. We are going to show that $\check{C}_{s,A,\tau}$ vanishes unless $s|\frac{n}{A}$. Hence, by the Möbius inversion formula,
\[C_{s,A,\tau}=\sum_{s|j|\frac{n}{A}} \mu\left(\frac{j}{s}\right)\check{C}_{j,A,\tau}\]
for any $s,A,\tau$.

Fix a basis where $z$ is diagonal, say $(\mu_1,\ldots,\mu_n)$. Recall that this induces an isomorphism of $\Upsilon_A$ with decompositions $\{1,\ldots,n\}=\sqcup_{i=1}^A I_i$ with $|I_i|=\frac{n}{A}$, up to cyclic shift. Fix  $\alpha\in \Irr[\F_q^\times]$ of order $s$. Green's Theorem \ref{Green} shows that
\begin{align*}
    \check{C}_{s,A,\tau} 
    &= \frac{d_\tau^A}{A(q-1)} \sum_{\Lambda}\sum_{\{I_{\xi,i}\}}\prod_{\xi\in\Irr[\F_q^\times]}\prod_{i=1}^A\prod_{j\in I_{\xi,i}}\xi(\mu_{j})
\end{align*}
where $\Lambda$ runs over all multi-partitions $\Irr[\F_q^\times]\to \mcP$ of type $\tau$ fixed by $\alpha$, and $\{I_{\xi,i}\}$ over all partitions of $\{1,\ldots,n\}$ into subsets indexed by $\Irr[\F_q^\times]\times\{1,\ldots,A\}$ such that $|I_{\xi,i}|=|\Lambda(\xi)|$. The factor $\frac{1}{A}$ comes from the quotient by the cyclic shift.

Consider a set $\mcI$ of elements of the form $\iota=(\lambda^{(\iota)},I_1^{(\iota)},\ldots,I_A^{(\iota)})$ where $\lambda^{(\iota)}$ is a non-empty partition and the $I_i^{(\iota)}$'s are disjoint subsets of $\{1,\ldots,n\}$ with size $|\lambda^{(\iota)}|$. We further require that, when varying $i$ and $\iota$, the sets $I_i^{(\iota)}$ form a partition of $\{1,\ldots,n\}$ and that $\sum_{\iota\in \mcI}|\lambda^{(\iota)}|=\frac{n}{A}$. There is a forgetful map $\mcI\to\mcP$, $\iota\mapsto\lambda^{(\iota)}$. This gives us a multi-partition type $\tau(\mcI)$. To give an actual multi-partition $\Lambda$ of type $\tau$ together with the sets $I_{\xi,i}$ is equivalent to giving such an $\mcI$ with $\tau(\mcI)=\tau$ and an injective map $\psi: \mcI\to \Irr[\F_q^\times]$. Indeed, set \[\Lambda(\xi)= \left\{\begin{array}{cl}
    \lambda^{(\psi^{-1}(\xi))} & \text{if } \xi\in\operatorname{Img}\psi \\
    \emptyset & \text{if not}
\end{array} \right.\]
for such $\psi$. Hence 
\[\check{C}_{s,A,\tau}=\frac{d_\tau^A}{A(q-1)}\sum_{\tau(\mcI)=\tau} \sum_{\psi}\prod_\iota \psi(\iota)(\mu_{\iota})\]
where the first sum runs over all $\mcI$ as before whose induced multi-partition has type $\tau$, the second one is over all injections $\psi:\mcI\to\Irr[\F_q^\times]$ such that the induced multi-partition $\Lambda:\Irr[\F_q^\times]\to \mcP$ is stabilized by $\alpha$, and $\mu_\iota = \prod_{i=1}^A\prod_{j\in I_i^{(\iota)}}\mu_j$. 

Fix such an $\mcI$ and let
\[\Sigma(\mcI):=\sum_{\psi}\prod_\iota \psi(\iota)(\mu_{\iota})\]
where $\psi$ varies as before. The condition $\alpha\Lambda=\Lambda$ means that there exists $\rho\in W_{\mcI}\subset S(\mcI)$, the stabilizer of the forgetful map, with $\alpha \psi = \rho \psi$. The choice of such a $\rho$ is unique as the action of $S(\mcI)$ on $\Hom{\mcI}{\Irr[\F_q^\times]}$ is free over the injections. Hence, $\Sigma(\mcI)$ splits into terms of the form
\[ \Sigma(\mcI,\rho):=\sum_{\alpha\psi=\rho\psi}\prod_\iota \psi(\iota)(\mu_\iota)\]
Note that the sum is empty if $s=\ord\alpha$ does not agree with the size of each orbit of $\rho$. Assume that this is the case. In particular, as $\sum |\lambda^{(\iota)}|=\frac{n}{A}$ and each orbit has the same size, this size must divide $\frac{n}{A}$. Hence, $s|\frac{n}{A}$. Recall that $\mu_i$ has an $n$-th root for every $i$. Thus, $\alpha(\mu_\iota)=1$ for every $\iota$ and the value $[\xi](\mu_\iota)$ is defined for every class $[\xi]\in \Irr[\F_q^\times]/\alpha$. The sum becomes
\[s^{|\mcI/\rho|}\sum_{\ov{\psi}} \prod_ {[\iota]}\ov{\psi}([\iota])(\mu_{[\iota]})\]
where $\ov{\psi}$ varies among all injective maps $\mcI/\rho\to \Irr[\F_q^\times]/\alpha$, and $\mu_{[\iota]}=\prod_{\iota'\in [\iota]}\mu_{\iota'}$. Let $\ov{\mcI}=\mcI/\rho$. Note that, for any quotient map $f:\ov{\mcI}\to Q$,
\[\sum_{\ov{\psi}:Q\to \Irr[\F_q^\times]/\alpha} \prod_{c\in Q}\ov{\psi}(c)\left(\prod_{[\iota]\in f^{-1}(c)}\mu_{[\iota]}\right)= \left\{\begin{array}{cl}
   0  & \mathrm{if }\ |Q|>1 \\
   \frac{q-1}{s}  & \mathrm{if }\ |Q|=1
\end{array}\right.\]
as each non-empty proper sub-product of the $\mu_1,\ldots,\mu_n$ is not trivial, while $\mu_1\cdots\mu_n=1$. Hence, by the Möbius inversion formula on the poset of subgroups of $S(\ov{\mcI})$, we get
\[\Sigma(\mcI,\rho)=s^{|\mcI/\rho|}\frac{q-1}{s}\sum_{W}\mu(1,W)\]
where $W$ runs over all subgroups of $S(\ov{\mcI})$ that act transitively on $\ov{\mcI}$. Note that the last expression is linear in $q$. Crapo's closure theorem \cite{Crapo} implies that
\[\sum_{W}\mu(1,W) = (-1)^{|\ov{\mcI}|-1}(|\ov{\mcI}|-1)!\]
Furthermore, note that as each orbit of $\rho$ has size $s$, $|\mcI/\rho|=\frac{|\mcI|}{s}$. Hence,
\[ \Sigma(\mcI,\rho) =\left\{\begin{array}{cl}
    -\frac{(q-1)}{s}(-s)^{\frac{|\mcI|}{s}}\left(\frac{|\mcI|}{s}-1\right)!  & \text{if every orbit of }\rho\text{ has size $s$} \\
    0 & \text{otherwise} 
\end{array}\right.\]
Hence,
\[ \Sigma(\mcI) = -\frac{(q-1)}{s}(-s)^{\frac{|\mcI|}{s}}\left(\frac{|\mcI|}{s}-1\right)!\nu(\mcI,s)\]
where 
\[\nu(\mcI,s) = |\{\rho\in W_{\mcI}:\text{each orbit of }\rho\text{ has size $s$}\}|\]
Note that $\nu(\mcI,s)$ only depends on the type $\tau(\mcI)$ and $s$. Then write $\nu(\tau,s)$ for it, and $W_\tau$ for $W_{\mcI}$. Note that $\nu(\tau,s)=0$ unless $s$ divides the multiplicity of each partition of $\tau$. The size of $\mcI$ is also determined by $\tau$. Indeed, it is the sum of all multiplicities in $\tau$ of all partitions. Call it $\#\tau$. Hence
\[\check{C}_{s,A,\tau} = -\frac{d_\tau^A}{As}(-s)^{\frac{\#\tau}{s}}\left(\frac{\#\tau}{s}-1\right)!\nu(\tau,s)\vartheta(\tau,n)\]
where
\[\vartheta(\tau,n)= |\{\mcI:\tau(\mcI)=\tau\}|.\]

Finally, we obtain that
\[C_{s,A,\tau} =-d_\tau^A\vartheta(\tau,n) \sum_{s|j|\frac{n}{A}}\frac{1}{Aj}(-j)^{\frac{\#\tau}{j}}\left(\frac{\#\tau}{j}-1\right)!\nu(\tau,j)\mu\left(\frac{j}{s}\right)  \]
which is a constant not depending on $q$. This proves Theorem \ref{number-twisted-points}.

To finish this section, we will prove a combinatorial identity that we will need later on. Define $A\cdot\tau$ as the type of size $n$ obtained by repeating $A$ times each partition of $\tau$, that is, by multiplying each multiplicity $m_\lambda$ by $A$. Remark that $A\cdot \tau'=(A\cdot \tau)'$ and $\mathcal{H}_{\tau'}(q)^A=\mathcal{H}_{A\cdot\tau'}(q)$.

\begin{lema}\label{lem:C-rescaling}
    $C_{s,A,\tau}=C_{As,1,A\cdot\tau}$ for any $s,A,\tau$.
\end{lema}
\begin{proof}
    First, note that $(d_\tau )^A=d_{A\cdot\tau}$. Compare the $j$-term on the left-hand side with the $Aj$-one on the right.  We have
    \begin{align*}
        \frac{\#\tau}{j} = \frac{\#(A\cdot \tau)}{Aj} & & \text{ and } & & \mu\left(\frac{j}{s}\right)=\mu\left(\frac{Aj}{As}\right)
    \end{align*}
    So, we need that
    \[\vartheta(\tau,n)\nu(\tau,j) = A^{\frac{\#\tau}{j}}\vartheta(A\cdot\tau,n)\nu(A\cdot\tau,Aj)\]
    for arbitrary $j|\frac{n}{A}$.

    We are going to build a bijection between sets of such cardinalities. Take $\rho\in W_{A\cdot\tau}$. For each of its orbits $\mcO$ define 
    \[\mcJ(\mcO,\rho) :=\{\{\iota_1,\ldots,\iota_{\frac{|\mcO|}{A}}\}\subset \mcO:\rho^k\iota_i\neq \rho^l\iota_j\ \forall i\neq j,\ 0\leq k,l< A\}\]
    and let
    \[\mcJ_\rho :=\prod_{\mcO}\mcJ(\mcO,\rho)\]
    where the product runs over all orbits of $\rho$. Finally, put
    \[W_{A\cdot\tau,Aj} = \{\rho\in W_{A\cdot\tau}:\text{each orbit of }\rho\text{ has size $Aj$}\}\]
    and
    \[\mcJ=\{(\mcI,\rho,J):\tau(\mcI)=A\cdot\tau,\ \rho\in W_{A\cdot\tau,Aj},\ J\in\mcJ_\rho\}.\]

    The size of $\mcJ$ is nothing but
    \[A^{\frac{\#\tau}{j}}\vartheta(A\cdot\tau,n)\nu(A\cdot\tau,Aj)\]
    Indeed, given $\rho\in W_{A\cdot\tau,Aj}$, it has $\frac{\#\tau}{j}$ orbits and for each one, $|\mcJ(\mcO,\rho)|=A$. For the last claim, note that an element of $\mcJ(\mcO,\rho)$ is determined by $\iota_1$. There are $jA$ choices but $\iota_1$ and $\rho^A\iota_1$ give the same element. 

    Now, define
    \[\varphi: \mcJ \to \{\mcI:\tau(\mcI)=\tau\}\times \{\rho\in W_{\tau}:\text{each orbit of }\rho\text{ has size $j$}\}\]
    as follows. Fix an element of $\mcJ$. For each chosen $\iota_i$, collapse the $\{\rho^k\iota_i:0\leq k < A\}$ in $\mcI$. Then $\rho^A$ induces an element of $W_{\tau}$ on the collapsed set. To produce an $\mcI$ with $\tau(\mcI)=\tau$, put 
    $I^{\iota_i}_k=I^{\rho^k\iota_i}$
    for each $0\leq k < A$. It is clear that $\varphi$ is a bijection.
\end{proof}

\section{Stringy E-polynomial of type A parabolic character varieties} \label{sec:stringy-E-polynomial}

\subsection{Main statements and some consequences}\label{sec:consequences}

The proof of the following three theorems will be given in Sections \ref{sec:stringy-contributions}, \ref{sec:stringy-E-pol}, and \ref{sec:isotypic-contributions} respectively. They are lengthy consequences of Theorem \ref{number-twisted-points} thanks to Theorem \ref{KGWZ}. Before going on, let us clarify what we mean by an $F_d$-discrete torsion. Given $\mcD\in H^2(F_d^2,\U(1))\simeq F_d$, we also denote by $\mcD$ the discrete torsion $\sum_{i=1}^g\mathrm{pr}_i^*\mcD$ of $F_d^{2g}$, where $\mathrm{pr}_i:F_d^{2g}\to F_d^2$ is the projection onto the $i$-th pair of coordinates. We call these $F_d$-discrete torsions, and the order of such a torsion is the order of $\mcD$ in $F_d$. An $F_n$-discrete torsion restricts to an $F_d$-discrete torsion for any $d|n$. From now on, we write $C_{s,\tau}$ for $C_{s,1,\tau}$.

\begin{thm}\label{stringy-contributions}
    Let $n$ be a natural number. For any divisor $d$ of $n$, $F_d$-discrete torsion $\mcD$ of order $\frac{d}{r}$, $a\in F_d^{2g}$, and any nice conjugacy class $\mcC$ of $\SL_n(\C)$,
    \begin{align*}
        &E(\abbreviatedBetti^a/F_d^{2g},\mcL_{\mcD,a}; u,v)(uv)^{F(a)} \\
        &= \sum_{\tau} \frac{\left((-1)^n(uv)^{\frac{n^2}{2}}\mathcal{H}_{\tau'}(uv)\right)^{2g-1}}{(uv-1)^{2g-1+n-1}}\sum_{\substack{A|s|n\\ \gcd(s,\frac{n}{d})|r\frac{s}{A}}} C_{s,\tau} \gcd\left(\frac{s}{A},\frac{n}{d}\right)^{2g-1}\gcd\left(s,\frac{n}{d}\right)
    \end{align*}
    where $A$ is the order of $a$, $\mcL_{\mcD,a}$ is the local system associated with $\mcD$, $F(a)$ is one-half the codimension of $\abbreviatedBetti^a$, and $\tau$ runs over all multi-partition types of size $n$.
\end{thm}

\begin{thm}\label{stringy-E-pol}
    Let $n$ be a natural number. For any divisor $d$ of $n$, $F_d$-discrete torsion $\mcD$ of order $\frac{d}{r}$, and any nice conjugacy class $\mcC$ of $\SL_n(\C)$,
    \[E_{st}^{\mcD}(\abbreviatedBetti[\SL_n/F_d]; u,v) = \sum_{\tau} \frac{\left((-1)^n(uv)^{\frac{n^2}{2}}\mathcal{H}_{\tau'}(uv)\right)^{2g-1}}{(uv-1)^{2g-1+n-1}}\sum_{s|n} s^{2g} C_{s,\tau} \Phi_g(n,d,s,r)\]
    In particular, 
    \[E_{st}^{\hat\mcD}(\abbreviatedBetti[\SL_n/F_d]; u,v) = E_{st}^{\check\mcD}(\abbreviatedBetti[\SL_n/F_{\frac{n}{d}}]; u,v)\]
    for any $d$, $\mcC$, $F_d$-discrete torsion $\hat\mcD$ and $F_{\frac{n}{d}}$-discrete torsion $\check\mcD$ whose images under the natural surjections $F_d\to F_{\gcd(d,\frac{n}{d})}$ and $F_{\frac{n}{d}}\to F_{\gcd(d,\frac{n}{d})}$ have the same order.
\end{thm}

Recall that
\[\Phi_g(n,d,s,r)=\prod_{p|n\text{  prime}} \frac{(p-p^{1-2g})p^{\phi_g(p,n,d,s,r)} + p^{1-2g}-1}{p-1}\]
where
\[\phi_g(p,n,d,s,r)=\min(v_p(s),v_p(n)-v_p(d),v_p(d),v_p(n)-v_p(s),v_p(r))\]

\begin{rmk}
    Although $\Phi_g$ is symmetric in $s$ and $\frac{n}{s}$, $C_{s,\tau}$ is not. For example, with $n=4$, $s=1$, $\tau=\{\{1\},\{1\},\{1\},\{1\}\}$, $C_{s,\tau}=0$ and $C_{\frac{n}{s},\tau}=6$. 
\end{rmk}

\begin{thm}\label{isotypic-contributions}
    Let $n$ be a natural number. For any divisor $d$ of $n$, $F_n$-discrete torsion $\mcD$, character $\xi$ of $F_{\frac{n}{d}}^{2g}$, $a\in F_{\frac{n}{d}}^{2g}$, and any nice conjugacy class $\mcC$ of $\SL_n(\C)$,
    \[E_{st}^{\mcD}(\abbreviatedBetti[\SL_n/F_d]; u,v)_\xi =E(\abbreviatedBetti^a/F_{\frac{n}{d}}^{2g},\mcL_{\mcD,a}; u,v)(uv)^{F(a)}\]
    provided that $\ord\xi=\ord a$, where $\mcL_{\mcD,a}$ is the local system on $\abbreviatedBetti^a/F_{\frac{n}{d}}^{2g}$ associated with $\mcD\big|_{F_{\frac{n}{d}}^{2}}$ and $F(a)$ is one-half the codimension of $\abbreviatedBetti^a$. In other words, the $\xi$-isotypic part of the stringy $E$-polynomial of $[\abbreviatedBetti/F_d^{2g}]$ equals the contribution of the sector $a$ of the dual orbifold $[\abbreviatedBetti/F_{\frac{n}{d}}^{2g}]$.
\end{thm}

We now establish some basic consequences of the previous theorems.

\begin{coro}\label{cor:euler-characteristic}
     Let $n$ be a natural number. For any divisor $d$ of $n$ and any nice conjugacy class $\mcC$ of $\SL_n(\C)$, the stringy Euler characteristic of $\abbreviatedBetti[\SL_n/F_d]$ vanishes if $n,g>1$ and is $1$ if $n=1$.
\end{coro}
\begin{proof}
    Recall that the stringy Euler characteristic is equal to $E_{st}(\abbreviatedBetti[\SL_n/F_d];1,1)$. The polynomials $\mathcal{H}_{\tau'}(uv)$ are always exactly divisible by $(uv-1)^n$ for the types in the formula of Theorem \ref{stringy-E-pol}. Hence,
    \[v_{uv-1}\left(\frac{\left((-1)^n(uv)^{\frac{n^2}{2}}\mathcal{H}_{\tau'}(uv)\right)^{2g-1}}{(uv-1)^{2g-1+n-1}}\right) = (2g-1)n- (2g-1+n-1)=2(g-1)(n-1)\]
    and the stringy Euler characteristic vanishes if $n,g>1$. For $n=1$, we have $d=s=r=1$, $\tau=\{\{1\}\}$, $\mathcal{H}_\tau(q)=q^{-\frac{1}{2}}(1-q)$, $C_{s,\tau}=1$, and $E_{st}(\abbreviatedBetti[\SL_n/F_d];1,1)=1$. 
\end{proof}

\begin{coro}\label{cor:irreducible}
     Let $n$ be a natural number. For any divisor $d$ of $n$ and any nice conjugacy class $\mcC$ of $\SL_n(\C)$, $\abbreviatedBetti[\SL_n/F_d]$ is irreducible.
\end{coro}
\begin{proof}
    It is equidimensional by Lemma \ref{lem:dim-sectors}. Therefore, it is enough to check that the leading coefficient of its $E$-polynomial, which counts its irreducible components, is one. Now the degree of $\mathcal{H}_{\tau}$ is
    \[\sum_{\lambda\in\tau}m_\lambda (-\frac{1}{2}\langle \lambda,\lambda\rangle + \sum_{\square \in \lambda} h(\square))\]
    where the second sum is over the boxes of the Ferrers diagram associated with $\lambda$. Recall that
    \[-\frac{1}{2}\langle\lambda,\lambda\rangle =-\frac{1}{2}|\lambda|-\sum_{i}\lambda_i(i-1)\]
    On the other hand, one has
    \[\sum_{\square\in \lambda }h(\square)=\sum_i \left(\frac{\lambda_i(\lambda_i+1)}{2} +\lambda_i(i-1)\right) \]
    by counting how many times a square is counted in a hook length. Hence,
    \[\deg\mathcal H_\tau=\tfrac12\sum_{\lambda\in\tau}m_\lambda\sum_i\lambda_i^2\le\tfrac12\Big(\sum_{\lambda\in\tau}m_\lambda\sum_i\lambda_i\Big)^2=\tfrac{n^2}{2},\]
    with equality only for $\{\{n\}\}$, whose dual type is $\tau=\{\{1,\ldots,1\}\}$. For this type we have 
    \[C_{s,\tau}=\left\{\begin{array}{cl}
        1 & \text{if } s = 1  \\
        0 & \text{if } s > 1 
    \end{array}\right.\]
    This means that the leading coefficient of $E(\abbreviatedBetti[\SL_n/F_d];u,v)$ is 
    \[\gcd\left(\frac{1}{1},\frac{n}{d}\right)^{2g-1}\gcd\left(1,\frac{n}{d}\right)=1.\]
\end{proof}

\begin{coro}\label{cor:components}
     Let $n$ be a natural number, $d$ be a divisor of $n$, and $\mcC$ a nice conjugacy class of $\SL_n(\C)$. For any $a\in F_d^{2g}$ of order $A$ coprime with $\frac{n}{d}$, the irreducible components of the twisted sector $\abbreviatedBetti^a/F_d^{2g}$ are the $\abbreviatedBetti^a_\upsilon/F_d^{2g}$, where $[\upsilon]$ runs over $\Upsilon_A$.
\end{coro}
\begin{proof}
    By the same argument as before, we only need to look at $A\cdot\tau$ where $\tau=\{\{1,\ldots,1\}\}$. For this type we have
    \begin{align*}
        \frac{n^2}{2A}+A\deg \mathcal{H}_{\tau'} &= \frac{n^2}{A} 
    \end{align*}
    and, by Lemma \ref{lem:C-rescaling},
    \begin{align*}
        C_{As,1,A\cdot\tau}=C_{s,A,\tau} = \left\{\begin{array}{cl}
        \frac{\vartheta(\tau,n)}{A} & \text{if } s = 1  \\
        0 & \text{if } s > 1 
    \end{array}\right.
    \end{align*}
    Hence, the number of irreducible components of maximal dimension of $\abbreviatedBetti^a/F_d^{2g}$ is
    \[\frac{\vartheta(\tau,n)}{A}\gcd\left(\frac{A}{A},\frac{n}{d}\right)^{2g-1}\gcd\left(A,\frac{n}{d}\right)=\frac{\vartheta(\tau,n)}{A}.\]
    Note that $\frac{\vartheta(\tau,n)}{A}$ is the number of ways to separate $\{1,\ldots,n\}$ into $A$ subsets of size $\frac{n}{A}$ up to cyclic shift. Similar arguments to that of Theorem \ref{number-twisted-points} show that $\abbreviatedBetti^a_\upsilon$ is non-empty. Since each $\abbreviatedBetti^a_\upsilon$ is $F_d^{2g}$-stable and equidimensional of the same dimension, the claim follows.
\end{proof}

\subsection{Preparations}

In this section, we describe the Fermionic shifts and provide a spread out of $\abbreviatedBetti$. 

\begin{lema}\label{lem:fermionic-shift}
    Let $\mcC\subset \SL_n(\C)$ be a nice conjugacy class. Then 
    \[F(a) = \frac{1}{2}(\dim \abbreviatedBetti-\dim \betti[\SL_n][\C]^a)=\frac{n^2}{2}(2g-1)\left(1-\frac{1}{A}\right)\]
    for any $a\in F_d^{2g}$, where $A$ is the order of $a$.
\end{lema}
\begin{proof}
    This follows from the argument at the end of the proof of \cite[Proposition 8.2]{HT} due to the existence of an invariant symplectic form \cite{Symplectic-form}.
\end{proof}

\begin{lema}
    Let $R\subset \C$ be a finitely generated $\Z$-algebra which contains the eigenvalues $\mu_1,\ldots,\mu_n$ of a nice conjugacy class $\mcC\subset \SL_n(\C)$ and the elements $(\mu_i-\mu_j)^{-1}$ for $i\neq j$. Then, for any morphism $\phi:R\to \F_q$ to a finite field, $\PGL_n(\F_q)\cdot\phi(\mcC\cap \SL_n(R))$ is the $\PGL_n(\F_q)$-orbit of $\operatorname{diag}(\phi(\mu_1),\ldots,\phi(\mu_n))$.
\end{lema}
\begin{proof}
    Any element of $\phi(\mcC\cap\SL_n(R))$ has characteristic polynomial $\prod_i(t-\phi(\mu_i))$, whose roots lie in $\F_q$ and are pairwise distinct, as the $\mu_i-\mu_j$ are invertible in $R$. Hence, such an element is diagonalizable over $\F_q$, and $\phi(\mcC\cap\SL_n(R))$ is contained in the $\PGL_n(\F_q)$-orbit of $\operatorname{diag}(\phi(\mu_1),\ldots,\phi(\mu_n))$. As $\phi(\mcC\cap\SL_n(R))$ contains $\operatorname{diag}(\phi(\mu_1),\ldots,\phi(\mu_n))$, the claim follows.
\end{proof}

\begin{prop}
    Let $\mcC\subset \SL_n(\C)$ be a nice conjugacy class and $R\subset \C$ as in the previous lemma. Then $\betti[\SL_n][R][\mcC\cap \SL_n(R)]$ is a spread out of $\abbreviatedBetti$ and
    \[\betti[\SL_n][R][\mcC\cap \SL_n(R)]_\phi(\F_q) = \betti[\SL_n][\F_q][\PGL_n(\F_q)\cdot\phi(\mcC\cap \SL_n(R))]\]
    for any morphism $\phi: R\to \F_q$ to a finite field.
\end{prop}
\begin{proof}
    The previous lemma implies the analogous statement for $\preBetti[\C]$. Then the first claim follows from Lemma 2 of \cite{Seshadri}. The second one follows from \cite[item (ii) of Theorem 3]{Seshadri} and \cite[Lemma 3.2]{Kac}, as the action is free by Lemma \ref{lem:free-smooth}.
\end{proof}

Thus, setting $R\subset \C$ as the $\Z$-algebra generated by $\frac{1}{n}$, all the $n$-th roots of all eigenvalues of $\mcC$, the elements $(\mu_i-\mu_j)^{-1}$ for $i\neq j$, and all the elements of the form $(\prod_{ i\in I}\mu_i -1)^{-1}$, where $I$ is a non-empty proper subset of $\{1,\ldots,n\}$, we get a spread out with twisted polynomial count by Theorem \ref{number-twisted-points}. 

\subsection{Stringy contributions}\label{sec:stringy-contributions}

In this section, we prove Theorem \ref{stringy-contributions}. Fix a primitive $d$-th root of unity $\omega$ and consider the associated discrete torsion $\mcD \in H^2(F_d^2,\U(1))\simeq F_d$. As in Ruan's example \cite{R}, the discrete torsion $\sum \mathrm{pr}_i^*\mcD$ induces the family of local systems on $\abbreviatedBetti^a$ given by $\xi_a:F_d^{2g} \to \C^\times$
\[\xi_a(b)=\omega^{\sum (a_{2i-1}b_{2i}-a_{2i}b_{2i-1})}\]
For another discrete torsion $r\mcD$, the associated local system is given by $\xi_a(b)^r$. Every $F_d$-discrete torsion is of this form, its order is $\frac{d}{\gcd(r,d)}$, and the computation below only depends on $\gcd(r,d)$.

As we said before, Theorem \ref{number-twisted-points} implies that $\abbreviatedBetti^a$ has twisted polynomial count for any character. Hence, by Theorem \ref{KGWZ}, $E(\abbreviatedBetti^a/F_d^{2g},\mcL_{r\mcD,a};u,v)$ is obtained by replacing $q$ with $uv$ in
\begin{align*}
    \#^{\xi_a^r}_{F_d^{2g}}\betti[\SL_n][\ov{\F}_q]^a &=\frac{1}{d^{2g}}\sum_{b\in F_d^{2g}}S_{a,b} \xi_a(b)^r \\
    &=\frac{1}{d^{2g}}\sum_{b\in F_d^{2g}}\sum_\tau \frac{\left((-1)^nq^{\frac{n^2}{2A}}\mathcal{H}_{\tau'}(q)^A\right)^{2g-1}}{(q-1)^{2g-1+n-1}}\sum_{s}s^{2g}AC_{s,A,\tau}\xi_a(b)^r  
\end{align*}    
where $s$ runs over all common divisors of $\frac{n}{A}$ and $\frac{n}{B}$ such that $A|\Delta(a,b)\gcd(A,\frac{n}{Bs})$. Recall that $A$ and $B$ denote the orders of $a$ and $b$ respectively. For the first equality, recall that $S_{a,b}$ is the number of points of $\betti[\SL_n][\ov{\F}_q]^a$ fixed by $b^{-1}\frob$. 

Replacing $s$ with $As$ and $\tau$ with $A\cdot\tau$, and reordering the sum, by Lemma \ref{lem:C-rescaling} we obtain
\begin{align*}
    E(\abbreviatedBetti^a/F_d^{2g},\mcL_{r\mcD,a};u,v) &= \sum_{\tau} \frac{\left((-1)^n(uv)^{\frac{n^2}{2A}}\mathcal{H}_{\tau'}(uv)\right)^{2g-1}}{(uv-1)^{2g-1+n-1}}\sum_{A|s|n} s^{2g} C_{s,\tau} \sum_{b\in F_{a,s}}\frac{\xi_a(b)^r}{d^{2g}A^{2g-1}}
\end{align*}
where $F_{a,s}\subset F_d^{2g}$ is the subgroup of all $b$ whose order $B$ divides $\frac{nA}{s}$ and $A|\Delta(a,b)\gcd(A,\frac{nA}{Bs})$. Here $\tau$ runs over all multi-partition types of size $n$. Indeed, the types that are not of the form $A\cdot\tau_0$ contribute zero, as $C_{s,\tau}=0$ unless $s$ divides the multiplicity of each partition of $\tau$, and $A|s$. 

We want to compute
\[\Phi_{st}(a,n,d,s,r):=\sum_{b\in F_{a,s}}\frac{\xi_a(b)^r}{d^{2g}A^{2g-1}}\]
Recall from Section \ref{counting-twisted-points} that
\[\Delta(a,b) = \frac{\gcd(A,B)}{|\langle \frac{AB}{d^2}\sum_{i=1}^g(a_{2i-1}b_{2i}-a_{2i}b_{2i-1}) \rangle |}\]
for the isomorphism $F_d\simeq \Z/d\Z$ given by $\omega$. Hence, if $\hat\Delta(a,b)=\frac{AB}{d^2}\sum_{i=1}^g(a_{2i-1}b_{2i}-a_{2i}b_{2i-1})$, 
\[\Delta(a,b)= \gcd(A,B,\hat\Delta(a,b)),\]
\[\Phi_{st}(a,n,d,s,r)=\sum_{b\in F_{a,s}}\frac{\omega^{\frac{rd^2}{AB}\hat\Delta(a,b)}}{d^{2g}A^{2g-1}}\]
and the condition $A|\Delta(a,b)\gcd(A,\frac{nA}{Bs})$ defining $F_{a,s}$ can be replaced by $A|\hat\Delta(a,b)\gcd(A,\frac{nA}{Bs})$ because $A|\gcd(A,B)\gcd(A,\frac{nA}{Bs})$ provided $s|n$. In addition, note that $\omega^{\frac{d^2}{AB}}$ has order dividing $\gcd(A,B)$. Hence, we can think of $\hat\Delta(a,b)$ as an element of $F_{\gcd(A,B)}$.

Set
\[N_a(A,B,\hat\Delta)=|\{b\in F_d^{2g}: \ord b=B,\hat\Delta(a,b)= \hat\Delta \}|\]
and 
\[\hat N_a(A,B,d_2,\hat\Delta)=|\{b\in F_d^{2g}: d_2|b_i, \hat\Delta(a,b)= \hat\Delta \}|\]
for $d_2|B$, $\ord a=A$ and $\hat\Delta\in F_{\gcd(A,B)}$, where $d_2|b_i$ means that $b_i\in F_{B/d_2}\subset F_B\subset F_d$, and $\hat\Delta(a,b)$ is computed with this $B$ even if $\ord b\neq B$. 

Note that, given $a$, $\hat\Delta(a,-):F_B^{2g}\subset F_d^{2g}\to F_{\gcd(A,B)}$ is linear.  In addition, $\gcd(A,d_2)|\hat \Delta(a,b)$ if $d_2|b_i$ for all $i$. On the other hand, any such value of $\hat \Delta$ can be obtained because $\langle a_1,\ldots,a_{2g}\rangle=F_A$. Thus,
\[\hat N_a(A,B,d_2,\hat\Delta)=\left\{\begin{array}{cl}
       \frac{B^{2g}\gcd(A,d_2)}{d_2^{2g}\gcd(A,B)}  & \text{if }\gcd(A,d_2)|\hat\Delta \\
        0 & \text{if not}
    \end{array}\right.\]
and, by Möbius inversion,
\[ N_a(A,B,\hat\Delta) = \sum_{d_2} \frac{B^{2g}\gcd(A,d_2)}{d_2^{2g}\gcd(A,B)}\mu(d_2)\]
where $d_2$ runs over all divisors of $B$ such that $\gcd(A,d_2)|\hat\Delta$. Note that the condition $\gcd(A,d_2)|\hat\Delta$ is equivalent to $\gcd(A,d_2)|\gcd(A,B,\hat\Delta)$. Furthermore, observe that $N_a(A,B,\hat\Delta)$ does not depend on $a$. Hence, we drop it from the notation.

Rewrite $\Phi_{st}(a,n,d,s,r)$ as
\begin{align*}
    \sum_{B|\frac{nA}{s},d}\sum_{\Delta} \frac{N(A,B,\Delta)}{d^{2g}A^{2g-1}}S(\Delta)
\end{align*}
where the second sum is over all divisors $\Delta$ of $\gcd(A,B)$ such that $A|\Delta\gcd(\frac{nA}{Bs},A)$ and $S(\Delta)$ is the sum of $\omega^{\frac{rd^2}{AB}\Delta i}$ for $1\leq i\leq \frac{\gcd(A,B)}{\Delta}$ coprime with $\frac{\gcd(A,B)}{\Delta}$. Looking at all conditions on $\Delta$, we find a sum of the form (for some $\Delta_1,\Delta_2$)
\begin{align*}
    \sum_{\Delta_1|\Delta|\Delta_2} S(\Delta) &= \sum_{1\leq i \leq \frac{\Delta_2}{\Delta_1}}\omega^{\frac{rd^2\Delta_1}{AB}i} = \left\{\begin{array}{cl}
        \frac{\Delta_2}{\Delta_1} & \text{if }\ord\omega | \frac{rd^2\Delta_1}{AB} \\
        0 & \text{if not } 
    \end{array} \right.
\end{align*}
as long as $\Delta_1|\Delta_2$ and $\omega^{\frac{rd^2\Delta_2}{AB}}=1$. Hence, $\Phi_{st}(a,n,d,s,r)$ is
\begin{align*}
    \sum_{B|\frac{nA}{s},d}\sum_{d_2} \frac{B^{2g}\gcd(A,d_2)}{d^{2g}A^{2g-1}d_2^{2g}\lcm\left(\frac{A}{\gcd(\frac{nA}{Bs},A)},\gcd(A,d_2)\right)}\mu(d_2) 
\end{align*}
where $d_2$ runs over all divisors of $B$ with $d| \frac{rd^2}{AB}\lcm\left(\frac{A}{\gcd(\frac{nA}{Bs},A)},\gcd(A,d_2)\right)$

Now,
\begin{align*}
    \frac{\gcd(A,d_2)}{\lcm\left(\frac{A}{\gcd(\frac{nA}{Bs},A)},\gcd(A,d_2)\right)}&=\frac{\gcd(\frac{nA}{Bs},A)}{A}\gcd\left(\frac{A}{\gcd(\frac{nA}{Bs},A)},d_2\right)
    = \frac{\gcd(A, \frac{nAd_2}{Bs})}{A}
\end{align*}
The second divisibility condition on $d_2$ becomes $\gcd(A,\frac{nAd_2}{Bs})|\frac{rd}{B}\gcd(A,d_2)$ and
\begin{align*}
    \Phi_{st}(a,n,d,s,r)&= \sum_{B|\frac{nA}{s},d}\sum_{d_2} \frac{B^{2g}\gcd(A,\frac{nAd_2}{Bs})}{d^{2g}A^{2g}d_2^{2g}}\mu(d_2) 
\end{align*}

Note that $\Phi_{st}(a,n,d,s,r)$ only depends on $A=\ord a$, so we write $\Phi_{st}(A,n,d,s,r)$ for it. Moreover, it is multiplicative, i.e. 
\[\Phi_{st}(A,n,d,s,r)=
\prod_{p|n\text{ prime}}
\Phi_{st}\pare{p^{v_p(A)}, p^{v_p(n)},p^{v_{p}(d)},p^{v_{p}(s)},p^{v_{p}(r)}}
\]

Therefore, assume that $n=p^m$ for some prime $p$. Write $A=p^l$, $s=p^b$, $d=p^c$ and $r=p^k$. We know that $b,c\leq m$, $l\leq b,c$ and $k\leq c$. We will use the notation $\exp_p(i)$ to denote $p^i$ to avoid small superscripts. In this case, putting $d_2 = p^i$ and $B=p^j$,
\begin{align*}
    \Phi_{st}(A,n,d,s,r)=\sum_{i,j} \exp_p(2g(j-i-l-c)+\min(l,m+l-b+i-j))(-1)^{i}
\end{align*}
where $i$ and $j$ satisfy the following inequalities:
\begin{align*}
    &0\leq i\leq \min(1,c,m+l-b), \\
    &i\leq j\leq \min(c,m+l-b), \text{ and} \\
    &\min(l,m+l-b+i-j)\leq k+c-j+\min(l,i)
\end{align*}
If $c$ or $m+l-b$ is zero, $i$ and $j$ are always zero and 
\[  \Phi_{st}(A,n,d,s,r)= \exp_p(- 2gc - (2g-1)l)\]
Assume from now on that $c,m+l-b\geq 1$ so that $i$ can be $1$.

If $l=0$, the third inequality always holds and
\begin{align*}
    \Phi_{st}(A,n,d,s,r)=\sum_{0\leq i\leq 1}\sum_{i\leq j\leq c,m-b} \exp_p(2g(j-i-c))(-1)^{i}
\end{align*}
The terms $(i,j)$ and $(i+1,j+1)$ cancel each other. Only the term $(0,\min(c,m-b))$ survives. Hence
\[  \Phi_{st}(A,n,d,s,r)= \exp_p(2g\min(0,m-b-c))\]

Assume now that $l\geq 1$. Hence, $\min(i,l)=i$. The same cancellation happens. But now the term $(0,\min(c,m-b+l))$ appears if and only if
\begin{align*}
    \min(l,m-b+l-\min(c,m-b+l))&\leq k+c-\min(c,m-b+l)\\
    \min(l+\min(c,m-b+l),m-b+l)&\leq k+c\\
    \min(l+c,m-b+l)&\leq k+c
\end{align*}
Hence, $\Phi_{st}(A,n,d,s,r)$ is zero if the previous inequality is not satisfied and  
\[ \Phi_{st}(A,n,d,s,r)=
    \exp_p(2g\min(-l,m-b-c)+\min(l,m-b+l-\min(c,m-b+l))) \]
otherwise. This formula also works for the special cases we analyzed before. In conclusion, 
\[\Phi_{st}(A,n,d,s,r)=\left\{\begin{array}{cl}
    \left(\frac{\gcd(d,\frac{nA}{s})}{dA}\right)^{2g-1}\frac{\gcd(d,\frac{n}{s})}{d} &  \text{if }\gcd(d,\frac{n}{s})|r\frac{d}{A}\\
    0 & \text{if not}
\end{array}\right.\]
for all possible $A,n,d,s,r$. Hence, $E(\abbreviatedBetti^a/F_d^{2g},\mcL_{r\mcD,a};u,v)(uv)^{F(a)}$ is
\begin{align*}
    \sum_{\tau} \frac{\left((-1)^n(uv)^{\frac{n^2}{2}}\mathcal{H}_{\tau'}(uv)\right)^{2g-1}}{(uv-1)^{2g-1+n-1}}\sum_{\substack{A|s|n\\ \gcd(s,\frac{n}{d})|r\frac{s}{A}}} C_{s,\tau} \left(\gcd\left(\frac{s}{A},\frac{n}{d}\right)\right)^{2g-1}\gcd\left(s,\frac{n}{d}\right)
\end{align*}
as stated in Theorem \ref{stringy-contributions}. Recall that the Fermionic shift is $\frac{n^2}{2}(1-\frac{1}{A})(2g-1)$ by Lemma \ref{lem:fermionic-shift}.

\subsection{Stringy E-polynomial}\label{sec:stringy-E-pol}

Now we turn to Theorem \ref{stringy-E-pol}. By Theorem \ref{stringy-contributions},
\begin{align*}
    E_{st}^{r\mcD}(\abbreviatedBetti[\SL_n/F_d]; u,v) = \sum_{\tau} \frac{\left((-1)^n(uv)^{\frac{n^2}{2}}\mathcal{H}_{\tau'}(uv)\right)^{2g-1}}{(uv-1)^{2g-1+n-1}}\sum_{s|n} s^{2g} C_{s,\tau} \Phi_g(n,d,s,r)
\end{align*}
where
\begin{align*}
    \Phi_g(n,d,s,r)&=\sum_{a} \left(\frac{\gcd(d,\frac{nA}{s})}{dA}\right)^{2g-1}\frac{\gcd(d,\frac{n}{s})}{d}
\end{align*}
where $a$ runs over all elements of $F_d^{2g}$ for which $A|s,d$ and $\gcd(d,\frac{n}{s})|r\frac{d}{A}$. By Möbius inversion,
\begin{align*}
     \Phi_g(n,d,s,r) &=\sum_{A}\sum_{d_1|A} \frac{A\gcd(d,\frac{nA}{s})^{2g-1}\gcd(d,\frac{n}{s})}{d^{2g}d_1^{2g}}\mu(d_1)
\end{align*}
where $A$ varies among all common divisors of $s$ and $d$ such that $\gcd(d,\frac{n}{s})|r\frac{d}{A}$.

As before, we can focus on the case where $n=p^m$ for some prime $p$. Put $s=p^b$, $d=p^c$ and $r=p^k$. Then, for $A=p^i$ and $d_1=p^j$,
\begin{align*}
    \Phi_g(n,d,s,r)&=\sum_{i,j} \exp_p(i+(2g-1)(\min(c,m-b+i))+\min(c,m-b)-2g(c+j)) (-1)^j
\end{align*}
where $i$ and $j$ satisfy $0\leq i\leq \min(b,c, k+c-\min(c,m-b))$ and $0\leq j \leq \min(1,i)$.

Let us split into two ranges: $i\leq c-m+b$ and its complement. In the first one, the terms $(i,j)$ and $(i+1,j+1)$ cancel each other. Only the term $(\min(c-m+b,b,c,k+c-\min(c,m-b)),0)$ survives. This term appears only if $m-b\leq c$. In this case, we obtain
\[\exp_p(2g(\min(c-m+b,b,c,k+c-\min(c,m-b))+m-b-c))=\exp_p(2g\min(0,m-c,m-b,k))=1\]

Let us analyze the second range, i.e. $i>c-m+b$. We have
\begin{align*}
    &\left(\sum_{i} \exp_p(i)\right)\left(\sum_{0\leq j \leq 1} \exp_p(\min(0,m-b-c)-2gj) (-1)^j\right) + \delta_{c+b<m} \exp_p(2g\min(0,m-b-c))
\end{align*}
where $i$ ranges between $\max(1,c-m+b+1)$ and $\min(b,c, k+c-\min(c,m-b))$. 

We can evaluate the sum over $i$ using the geometric series. If $m\leq  b+c$, we obtain the following formula for $\Phi_g(n,d,s,r)$
\begin{align*}
    & \frac{\exp_p(\min(b,c,k+c-m+b)+1)-\exp_p(1+c-m+b)}{p-1}\exp_p({m-b-c})\left( 1-\exp_p(-2g)\right) +1\\
    &= \frac{\exp_p({\min(m-c,m-b,k)})(p-\exp_p(1-2g))+\exp_p(1-2g)-1}{p-1}
\end{align*}

If $m > b+c$,
\begin{align*}
    \Phi_g(n,d,s,r)&=\frac{\exp_p({\min(b,c,k)+1})-p}{p-1}\left( 1-\exp_p(-2g)\right) + 1\\
    &= \frac{\exp_p({\min(b,c,k)})(p-\exp_p(1-2g))+\exp_p(1-2g)-1}{p-1}
\end{align*}

Summing up,
\begin{align*}
    \Phi_g(n,d,s,r)=\frac{(p-\exp_p(1-2g))\exp_p(\min(b,m-c,c,m-b,k)) + \exp_p(1-2g)-1}{p-1}
\end{align*}
which is the formula in Theorem \ref{stringy-E-pol}. The last claim of the theorem follows by noticing that $r$ can be replaced by $\gcd(d,\frac{n}{d},r)$ and $\Phi_g(n,d,s,r)=\Phi_g(n,\frac{n}{d},s,r)$.

\subsection{Isotypic contributions}\label{sec:isotypic-contributions}

To finish, we prove Theorem \ref{isotypic-contributions}. We consider discrete torsion over $F_n$. Let $\omega$ be a primitive $n$-th root of unity and fix a generator $\mcD$ as in the previous sections. Pick an integer $r$. A character $\xi$ of $F_{\frac{n}{d}}^{2g}$ is of the form $\omega^{d\sum b_i^{(0)}b_i}$ for a unique $b^{(0)}\in F_\frac{n}{d}^{2g}$. In this setting, we have

\begin{align*}
    E_{st}^{r\mcD}(\abbreviatedBetti[\SL_n/F_d]; u,v)_\xi
    &=\sum_{\tau} \frac{\left((-1)^n(uv)^{\frac{n^2}{2}}\mathcal{H}_{\tau'}(uv)\right)^{2g-1}}{(uv-1)^{2g-1+n-1}}\sum_{s|n} s^{2g} C_{s,\tau}\Phi_{stp}(b^{(0)},n,d,s,r)
\end{align*}
for
\[\Phi_{stp}(b^{(0)},n,d,s,r)= \sum_{a,b}\frac{1}{n^{2g}A^{2g-1}}\omega^{\frac{rn^2}{AB}\hat\Delta(a,b)+db^{(0)}\cdot b}\]
where the sum runs over all $a\in F_d^{2g}$ and $b\in F_n^{2g}$ such that $A|s$, $B|\frac{nA}{s}$ and $A|\Delta(a,b) \gcd(\frac{nA}{Bs},A)$. Here $\Delta(a,b)$ and $\hat\Delta(a,b)$ are defined as in Section \ref{sec:stringy-contributions}, with $n$ in place of $d$. Therefore, Theorem \ref{isotypic-contributions} follows from
\[\Phi_{stp}(b^{(0)},n,d,s,r)= \left\{\begin{array}{cl}
    \Phi_{st}(B^{(0)},n,\frac{n}{d},s,dr) & \text{if }B^{(0)}|s \\
    0 & \text{if not}
\end{array}\right.\]
if $B^{(0)}:=\ord{b^{(0)}}=\ord a$.

Note that the condition $\gcd(\frac{n}{d},\frac{n}{s})|\frac{rn}{B^{(0)}}$ in $\Phi_{st}$ is superfluous. Hence, we want to prove
\[\Phi_{stp}(b^{(0)},n,d,s,r)= \left\{\begin{array}{cl}
     \left(\frac{\gcd(\frac{n}{d},\frac{nB^{(0)}}{s})}{\frac{n}{d}B^{(0)}}\right)^{2g-1}\frac{\gcd(\frac{n}{d},\frac{n}{s})}{\frac{n}{d}} & \text{if }B^{(0)}|s \\
    0 & \text{if not}
\end{array}\right.\]

Write 
\begin{align*}
    \Phi_{stp}(b^{(0)},n,d,s,r)=\frac{d^{2g}}{n^{2g}}\sum_{b\in F_n^{2g}}\omega^{db^{(0)}\cdot b} \check{\Phi}(b,n,d,s,r)
\end{align*}
where
\[\check{\Phi}(b,n,d,s,r) = \sum_{a\in F_{b,s}}\frac{1}{d^{2g}A^{2g-1}}\omega^{\frac{rn^2}{AB}\hat\Delta(a,b)}\]
and $F_{b,s}\subset F_d^{2g}$ is the subgroup of those $a$ such that $A|s$, $B|\frac{nA}{s}$ and $A|\Delta(a,b)\gcd(\frac{nA}{Bs},A)$.

As in Section \ref{sec:stringy-contributions}, for $B|n$ and $\ord b=B$, let
\[M(A,B,\hat\Delta):=|\{a\in F_d^{2g}: \ord a = A,\hat\Delta(a,b)=\hat\Delta\}| = \sum_{d_1} \frac{A^{2g}\gcd(B,d_1)}{d_1^{2g}\gcd(A,B)}\mu(d_1)\]
where $d_1$ varies among all divisors of $A$ such that $\gcd(B,d_1)|\hat\Delta$. We rewrite
\begin{align*}
    \check{\Phi}(b,n,d,s,r) &= \sum_{A}\sum_{\Delta}\frac{M(A,B,\Delta)}{d^{2g}A^{2g-1}} S(\Delta)
\end{align*} 
where $A$ runs over all common divisors of $s$ and $d$ such that $B|\frac{nA}{s}$, $\Delta$ runs over all divisors of $\gcd(A,B)$ that satisfy $A|\Delta\gcd(\frac{nA}{Bs},A)$, and $S(\Delta)$ is the sum of $\omega^{\frac{rn^2}{AB}\Delta i}$ for $1\leq i\leq \frac{\gcd(A,B)}{\Delta}$ coprime with $\frac{\gcd(A,B)}{\Delta}$. As in the previous section, we get
\begin{align*}     
     \check{\Phi}(b,n,d,s,r)&= \sum_{A}\sum_{d_1} \frac{A\gcd(B,d_1)}{d^{2g}d_1^{2g}\lcm\left(\frac{A}{\gcd(\frac{nA}{Bs},A)},\gcd(B,d_1)\right)}\mu(d_1) 
\end{align*}
where $d_1$ runs over all divisors of $A$ such that $n| \frac{rn^2}{AB}\lcm\left(\frac{A}{\gcd(\frac{nA}{Bs},A)},\gcd(B,d_1)\right)$.

Now,
\begin{align*}
    \frac{A\gcd(B,d_1)}{\lcm\left(\frac{A}{\gcd(\frac{nA}{Bs},A)},\gcd(B,d_1)\right)}&=\gcd(\frac{nA}{Bs},A)\gcd\left(\frac{A}{\gcd(\frac{nA}{Bs},A)},d_1\right)
    = \gcd\left(A, \frac{nAd_1}{Bs}\right)
\end{align*}
The second divisibility condition on $d_1$ becomes $\gcd(A,\frac{nAd_1}{Bs})|\frac{rn}{B}\gcd(B,d_1)$ and
\begin{align*}
    \check{\Phi}(b,n,d,s,r)&= \sum_{A}\sum_{d_1} \frac{\gcd(A,\frac{nAd_1}{Bs})}{d^{2g}d_1^{2g}}\mu(d_1) 
\end{align*}
Note that the sum is empty unless $B|\frac{nd}{s}$. Assume this from now on. Let us look at the conditions on $A$. As $d_1|A$, write $A=d_1A'$. Then $sB|nd_1A'$. Set $B'=\frac{sB}{\gcd(sB,nd_1)}$ so that $B'|A'$. Write $A'=B'A''$. Note that
\begin{align*}
    \gcd\left(B'd_1A'',\frac{nd_1B'}{Bs}d_1A''\right)=d_1A'', && B'd_1 = \frac{s}{\gcd(\frac{s}{d_1},\frac{n}{B})},
\end{align*}
and
\[\frac{d}{B'd_1}=\frac{d\gcd(\frac{s}{d_1},\frac{n}{B})}{s}=\frac{\gcd(\frac{d}{d_1}s,\frac{nd}{sB}s)}{s}=\gcd\left(\frac{d}{d_1},\frac{nd}{Bs}\right)\]
The conditions are then $A''|\frac{s}{d_1}$, $A''|\frac{n}{B}$, $A''|\frac{d}{d_1}$, $A''|\frac{dn}{Bs}$ and $A''|r\frac{n}{\lcm(B,d_1)}$. The last one is superfluous. Indeed, it is implied by the first two as $s|n$. Thus,
\begin{align*}
    \check{\Phi}(b,n,d,s,r)=\delta_{B|\frac{dn}{s}}\sum_{d_1|d,s} \frac{1}{d^{2g}d_1^{2g-1}}\mu(d_1)\sigma_1\left(\gcd\left(\frac{s}{d_1},\frac{n}{B},\frac{d}{d_1},\frac{dn}{Bs}\right)\right)
\end{align*}
where
\[\sigma_1(t)=\sum_{d_3|t}d_3\]
is the divisor-sum function. Note that this value only depends on $B=\ord b$, so we write $\check{\Phi}(B,n,d,s,r)$ for it.

Notice that $\sigma_1$ is multiplicative. Therefore, $\check\Phi$ is too. Assume that $n=p^m$ for some prime $p$. Write $s=p^b$, $d=p^c$, $r=p^k$, $B=p^j$ and $d_1=p^i$. Assume $j\leq m+c-b$. Otherwise, we know the sum is zero. In this case, $\check{\Phi}(B,n,d,s,r)$ is 
\begin{align*}
     \sum_{0\leq i\leq 1,b,c} \exp_p({2g(-i-c)+i})(-1)^{i}\sigma_1(\exp_p({\min(b-i,m-j,c-i,m+c-b-j)}))
\end{align*}

If $c=0$, then $i=0$ and $\check{\Phi}(B,n,d,s,r)=1$. On the other hand, if $b=0$,
\begin{align*}
    \check{\Phi}(B,n,d,s,r)&=\exp_p({-2gc})
\end{align*}
Finally, if $b,c>0$, $i$ can be $1$. Using that $\sigma_1(p^x)=\frac{p^{x+1}-1}{p-1}$ we get 
\begin{align*}
    \check{\Phi}(B,n,d,s,r) &= \exp_p({-2gc})\left(\frac{\exp_p({\min(b,m-j,c,m+c-b-j)+1})-1}{p-1}\right. \\ & -  \left.\exp_p(1-2g)\frac{\exp_p({\min(b-1,m-j,c-1,m+c-b-j)+1})-1}{p-1}\right) 
\end{align*}
provided $j\leq m+c-b$. This formula also recovers the cases where $b$ or $c$ are zero. 

Now, let us look at
\[N_{stp}(b^{(0)},B,\nabla)=|\{b\in F_B^{2g}: \ord b = B,b^{(0)}\cdot b=\nabla\}|\]
and 
\[\hat N_{stp}(b^{(0)},B,d_2,\nabla)=|\{b\in F_B^{2g}: d_2|b_i, b^{(0)}\cdot b=\nabla\}|\]
for $d_2|B$ and $\nabla\in F_{\gcd(B^{(0)},B)}$, where, as for $\hat\Delta$, $b^{(0)}\cdot b$ denotes from now on $\frac{BB^{(0)}}{n^2}d\sum b^{(0)}_ib_i$. The operator $b^{(0)}\cdot-$ is linear and has image in $F_{\gcd(B^{(0)},B)}$. In addition, $\gcd(B^{(0)},d_2)|b^{(0)}\cdot b$ if $d_2|b_i$ and any such value of $\nabla$ can be obtained because $|\langle b^{(0)}_1,\ldots,b^{(0)}_{2g}\rangle|=B^{(0)}$. Hence, 
\[\hat N_{stp}(b^{(0)},B,d_2,\nabla)=\left\{\begin{array}{cl}
       \frac{B^{2g}\gcd(B^{(0)},d_2)}{d_2^{2g}\gcd(B^{(0)},B)}  & \text{if }\gcd(B^{(0)},d_2)|\nabla \\
        0 & \text{if not}
    \end{array}\right.\]
and
\[ N_{stp}(b^{(0)},B,\nabla) = \sum_{d_2} \frac{B^{2g}\gcd(B^{(0)},d_2)}{d_2^{2g}\gcd(B^{(0)},B)}\mu(d_2)\]
where $d_2$ varies among all divisors of $B$ such that $\gcd(B^{(0)},d_2)|\nabla$.

Hence, $\Phi_{stp}(b^{(0)},n,d,s,r)$ is
\begin{align*} 
    &\sum_{d_2|B|n}\sum_{\nabla}  \frac{d^{2g}B^{2g}\gcd(B^{(0)},d_2)}{n^{2g}d_2^{2g}\gcd(B^{(0)},B)}\mu(d_2)\check{\Phi}(B,n,d,s,r) \sum_{i} \omega^{\frac{n^2}{BB^{(0)}}i\nabla} 
\end{align*}
where $\nabla$ runs over all divisors of $\gcd(B^{(0)},B)$ that are divisible by $\gcd(B^{(0)},d_2)$, and $i$ over all integers between $1$ and $\frac{\gcd(B^{(0)},B)}{\nabla}$ coprime with $\frac{\gcd(B^{(0)},B)}{\nabla}$. As before, it follows that
\begin{align*}
    \Phi_{stp}(b^{(0)},n,d,s,r) &=\sum_{d_2,B}  \frac{d^{2g}B^{2g}}{n^{2g}d_2^{2g}}\mu(d_2)\check{\Phi}(B,n,d,s,r)
\end{align*}
where the sum is over all $d_2|B|n$ such that $BB^{(0)}|n\gcd(B^{(0)},d_2)$.

This is a multiplicative function. Hence, we again assume that $n=p^m$ for some prime $p$ and write $s=p^b$, $d=p^c$, $r=p^k$ and $B^{(0)}=p^{j_0}$, and put $d_2=p^i$ and $B=p^j$. In this case, after multiplying by $p-1$ we have three terms
\[\sum_{i,j}\exp_p({2g(j-m-i)})(-1)^i\exp_p({\min(b,m-j,c,m+c-b-j)+1}),\]
\[-\sum_{i,j} \exp_p({2g(j-m-i-1)})(-1)^i\exp_p({\min(b-1,m-j,c-1,m+c-b-j)+2})\]
and
\begin{align*}
    \sum_{i,j}\exp_p({2g(j-m-i)})(-1)^i(-1 +  \exp_p(1-2g)) 
\end{align*}
where $i$ and $j$ satisfy, in all three sums, $0\leq i\leq 1$ and
\begin{align*}
    i\leq j\leq \min(m,m+\min(j_0,i)-j_0,m+c-b)
\end{align*}

In the latter, the terms $(i,j)$ and $(i+1,j+1)$ cancel each other. Only the term $(0,\min(m,m-j_0,m+c-b))$ survives, unless $j_0\geq 1$ and $b-c\leq j_0-1$. Thus, the value is
\[\max(\delta_{j_0=0},\delta_{b-c\geq j_0})(-1+\exp_p(1-2g))\exp_p({2g(\min(m,m-j_0,m+c-b)-m)})\]
 
To compute the difference between the first two sums, compare the terms $(i,j)$ and $(i,j+1)$ respectively. They only differ by one sign. Hence, only the terms $(i,\min(m,m+\min(j_0,i)-j_0,m+c-b))$ from the first one survive and $(0,0)$ and $(1,1)$ from the second one. The last two cancel each other too unless $\min(m,m+\min(j_0,1)-j_0,m+c-b)=0$. In the latter case, the term $(0,0)$ replaces the missing term $(1,\min(m,m+\min(j_0,1)-j_0,m+c-b))$ of the first sum.

Assume that $j_0\geq 1$ and $b-c\leq j_0-1$. So, $\min(m,m+\min(j_0,i)-j_0,m+c-b)=m+i-j_0$. In this case, we obtain
\begin{align*}
    (p-1)\Phi_{stp}(b^{(0)},n,d,s,r)=&\exp_p(-2gj_0)(\exp_p({\min(b,j_0,c,c-b+j_0)+1})\\
    &-\exp_p({\min(b,j_0-1,c,c-b+j_0-1)+1})) \\
    = &(p-1)\exp_p({-2gj_0+\min(j_0,c-b+j_0)})
\end{align*}
provided that $\min(b,j_0,c,c-b+j_0)=\min(j_0,c-b+j_0)$, and zero if not. Note that if $c-b+j_0\leq c$, $j_0\leq b$. And this last condition is equivalent to the one with the minima. Hence,
\[  \Phi_{stp}(b^{(0)},n,d,s,r) = \left\{ \begin{array}{cl}
        \exp_p({-2gj_0+\min(j_0,c-b+j_0)}) & \text{if } j_0\leq b \\
        0 & \text{if not}
    \end{array}\right.\]
Note that the inequality is equivalent to $B^{(0)}|s$.

Assume now that $j_0 = 0$ or $b-c\geq j_0$. In particular, $j_0\leq b$. In this case, $\min(m,m+\min(j_0,i)-j_0,m+c-b)=\min(m,m+c-b)$. We obtain
\begin{align*}
    \Phi_{stp}(b^{(0)},n,d,s,r) &=\exp_p({2g\min(0,c-b)}) \frac{-1+\exp_p(1-2g)+p-\exp_p(1-2g)}{p-1} \\
    &= \exp_p({2g\min(0,c-b)})
\end{align*}

Comparing each case with our previous formula for $\Phi_{st}$, we conclude
\[\Phi_{stp}(b^{(0)},n,d,s,r) =\delta_{\ord{b^{(0)}}|s}
\Phi_{st}
\pare{a,n,\frac{n}{d},s,dr}
\]
provided that $\ord{b^{(0)}}=\ord a$. This proves Theorem \ref{isotypic-contributions}.

\printbibliography

\end{document}